\documentclass[a4paper, intlimits, reqno]{amsart}

\usepackage[english]{babel}
\usepackage[T1]{fontenc}
\usepackage[utf8]{inputenc}

\usepackage{amsmath}
\usepackage{amssymb}
\usepackage{MnSymbol}
\usepackage{amsthm}
\allowdisplaybreaks
\usepackage{amsfonts}
\usepackage{mathrsfs} 
\usepackage{mathtools}
\usepackage{nicefrac}
\usepackage{enumitem}
\usepackage{multicol}
\usepackage{url}
\usepackage{dsfont}
\usepackage[numbers,sort&compress]{natbib}
\usepackage{doi}
\usepackage{prettyref}
\usepackage{xcolor}
\usepackage{orcidlink}
\usepackage{stackengine}
\usepackage{scalerel}
\def\rlwd{.6pt}
\def\rlht{1.3pt}
\def\shatvrule{\rule{\rlwd}{\rlht}}
\def\dhat#1{%
 \ThisStyle{%
  \setbox0=\hbox{$\SavedStyle#1$}%
  \stackon[0pt]{\stackon[1.2pt]{\ensuremath{\SavedStyle#1}}{%
    \shatvrule\kern\wd0\kern-\rlwd\kern-\rlwd\shatvrule}}%
    {\rule{\wd0}{\rlwd}}%
 }%
}
\def\uhat#1{%
 \ThisStyle{%
  \setbox0=\hbox{$\SavedStyle#1$}%
  \stackon[-1.5pt]{\stackon[1.2pt]{\ensuremath{\SavedStyle#1}}{%
    \shatvrule\kern\wd0\kern-\rlwd\kern-\rlwd\shatvrule}}%
    {\rule{\wd0}{\rlwd}}%
 }%
}
\newrefformat{defn}{Definition \ref{#1}}
\newrefformat{rem}{Remark \ref{#1}}
\newrefformat{sect}{Section \ref{#1}}
\newrefformat{app}{Appendix \ref{#1}}
\newrefformat{prop}{Proposition \ref{#1}}
\newrefformat{thm}{Theorem \ref{#1}}
\newrefformat{cor}{Corollary \ref{#1}}
\newrefformat{ex}{Example \ref{#1}}
\newrefformat{prob}{Problem \ref{#1}}

\swapnumbers
\newtheoremstyle{dotless}{}{}{\itshape}{}{\bfseries}{}{}{}
\theoremstyle{dotless}
\theoremstyle{plain}
\newtheorem{thm}{Theorem}[section]

\newtheorem{prop}[thm]{Proposition}
\newtheorem{cor}[thm]{Corollary}
\theoremstyle{definition}
\newtheorem{defn}[thm]{Definition}
\newtheorem{rem}[thm]{Remark}
\newtheorem{exa}[thm]{Example}
\newtheorem{prob}[thm]{Problem}
\newcommand{\N} {\mathbb{N}}
\newcommand{\K} {\mathbb{K}}
\newcommand{\R} {\mathbb{R}}
\newcommand{\C} {\mathbb{C}}
\newcommand{\e}{\mathrm{e}}

\DeclareMathOperator{\id}{id}

\providecommand{\differential}{\mathrm{d}}
\renewcommand{\d}{\differential}
\newcommand\rlim{
\mathchoice{\vcenter{\hbox{${\scriptstyle{+}}$}}}
{\vcenter{\hbox{$\scriptstyle{+}$}}}
{\vcenter{\hbox{$\scriptscriptstyle{+}$}}}
{\vcenter{\hbox{$\scriptscriptstyle{+}$}}}}
\newcommand\llim{
\mathchoice{\vcenter{\hbox{${\scriptstyle{-}}$}}}
{\vcenter{\hbox{$\scriptstyle{-}$}}}
{\vcenter{\hbox{$\scriptscriptstyle{-}$}}}
{\vcenter{\hbox{$\scriptscriptstyle{-}$}}}}
\newcommand{\vertiii}[1]{{\left\vert\kern-0.25ex\left\vert\kern-0.25ex\left\vert #1 
    \right\vert\kern-0.25ex\right\vert\kern-0.25ex\right\vert}}

\makeatletter
\DeclareRobustCommand\widecheck[1]{{\mathpalette\@widecheck{#1}}}
\def\@widecheck#1#2{%
    \setbox\z@\hbox{\m@th$#1#2$}%
    \setbox\tw@\hbox{\m@th$#1%
       \widehat{%
          \vrule\@width\z@\@height\ht\z@
          \vrule\@height\z@\@width\wd\z@}$}%
    \dp\tw@-\ht\z@
    \@tempdima\ht\z@ \advance\@tempdima2\ht\tw@ \divide\@tempdima\thr@@
    \setbox\tw@\hbox{%
       \raise\@tempdima\hbox{\scalebox{1}[-1]{\lower\@tempdima\box
\tw@}}}%
    {\ooalign{\box\tw@ \cr \box\z@}}}
\makeatother

\makeatletter
\newcommand{\fakephantomsection}{%
  \Hy@GlobalStepCount\Hy@linkcounter%
  \Hy@MakeCurrentHref{\@currenvir.\the\Hy@linkcounter}
  \Hy@raisedlink{\hyper@anchorstart{\@currentHref}\hyper@anchorend}%
}
\makeatother

\begin{document}

\title[$\mathrm{C}$-maximal regularity in locally convex spaces]{Continuous maximal regularity in locally convex spaces}
\author[K.~Kruse]{Karsten Kruse\,\orcidlink{0000-0003-1864-4915}}
\address[KK and FLS]{University of Twente, Department of Applied Mathematics, P.O.~Box 217, 7500 AE Enschede, The Netherlands}
\email{k.kruse@utwente.nl}
\author[F.L.~Schwenninger]{Felix L. Schwenninger\,\orcidlink{0000-0002-2030-6504}}
\email{f.l.schwenninger@utwente.nl}



\date{\today}
\begin{abstract}
We study maximal regularity with respect to continuous functions for strongly continuous semigroups on locally convex spaces as well as its relation to the notion of admissible operators. This extends several results for classical strongly continuous semigroups on Banach spaces. In particular, we show that Travis' characterization of $\mathrm{C}$-maximal regularity using the notion of bounded semivariation carries over to the general case. Under some topological assumptions, we further show the equivalence between maximal regularity and admissibility in this context.
\end{abstract}
\maketitle

\section{Introduction}

Maximal regularity for parabolic equations is a classical and indispensable tool in the study of linear and nonlinear parabolic PDEs. To this date, it remains an active area of research. In contrast to $\mathrm{L}^{p}$-maximal regularity with $p\in(1,\infty)$, which is well-studied in theory and its application, in this work we aim to continue research on the less standard case $p =\infty$, or, maximal regularity with respect to the continuous functions. While this case is often considered as rather exotic --- not least due to Baillon's seminal result \cite{baillon1980}, see also  \cite{eberhardt1992} and \cite[Sect.~17.4]{hytonen2023} --- in the context of analytic $C_{0}$-semigroups on Banach spaces, maximal regularity estimates arising in the $\mathrm{L}^{1}$, $\mathrm{L}^{\infty}$ and $\mathrm{C}$-norm, are relevant, in particular when considered on interpolation spaces 
e.g.~\cite{daprato1979,danchin2009,guerre1993,danchin2015,lecrone2011,lorenzirhandi2021}.
Let us briefly recall the setting for abstract evolution equations in Banach spaces to set the stage for what we are aiming for --- the generalization to Hausdorff locally convex spaces. For a $C_{0}$-semigroup $(T(t))_{t\geq 0}$ on a Banach space $(X,\|\cdot\|)$, with generator $A\colon D(A)\to X$, \emph{$\mathrm{L}^{p}$-maximal regularity} refers to the property that for $r>0$ all terms 
in the equation
\begin{equation}\label{eq:maxreg}
u'(t)=Au(t)+f(t),\quad t\in[0,r],\quad u(0)=0,
\end{equation}
have the same time-regularity as all forcing terms $f\colon [0,r]\to X$, say $f\in \mathrm{L}^{p}([0,r];X)$, $p\in[1,\infty]$.
This is equivalent to requiring that for all $f\in\mathrm{L}^{p}([0,r];X)$, it holds that the function
\[
	[0,r]\ni t\mapsto A\int_{0}^{t}T(t-s)f(s)\d s \eqqcolon A(T\ast f)(t)\in X,
\]
belongs to $\mathrm{L}^{p}([0,r];X)$.
We emphasize that the latter property entails that the integral, that is, the (mild) solution to \prettyref{eq:maxreg}, lands in the domain of $A$, which is a-priori not clear. 
Further, we note that $\mathrm{L}^{\infty}$-maximal regularity is in fact equivalent to $\mathrm{C}$-maximal regularity, see \cite[Theorem 17.2.46, p.~616]{hytonen2023}, which is the case we focus on in the following. The power of this notion in the context of nonlinear PDEs lies in the norm estimate that follows automatically by the closed graph theorem. 

Travis \cite{travis1981} showed that $\mathrm{C}$-maximal regularity is already implied  by (and thus is equivalent to) the property that the mild solution $u=T\ast f$ of \prettyref{eq:maxreg} takes values in $D(A)$ for every $t\geq 0$ and $f\in \mathrm{C}([0,r];X)$. Equivalently, since $A$ is closed, this means that $(T_{-1}\ast A_{-1}f)(t)$ lies in $X$ for every $t\geq 0$. Here, the subscript refers to the corresponding unique extensions to the extrapolation space $X_{-1}$, defined as the completion of $X$ with respect to $\|A^{-1}\cdot\|$, assuming, w.l.o.g.~that $A$ is invertible. This very property is also known as $A_{-1}$ being an \emph{admissible control operator with respect to the continuous functions}, arising from infinite-dimensional systems theory \cite{weiss1989a, weiss1989b,staffans2005}, see also \cite{JacoSchwWint2022}, where this relation is discussed in detail. In some sense, that equivalence between maximal regularity and admissible operators can be seen as ``time-regularity from spatial-regularity'', since continuity (in times) follows from the fact that solutions are well-defined.  

Remarkably, the above equivalence of $A_{-1}$ being $\mathrm{C}$-admissible and $(T(t))_{\geq 0}$ having 
$\mathrm{C}$-maximal regularity drastically fails if $\mathrm{C}$ is replaced by $\mathrm{L}^{\infty}$. In \cite{JacoSchwWint2022} it was shown that if $A_{-1}$ is admissible with respect to $\mathrm{L}^{\infty}$, then $A$ extends to a bounded operator on $X$, with the converse being trivial. 

The ``classical'' setting above for strongly continuous semigroups on Banach spaces (with respect to the norm-topology) is the point of departure of this paper in which we aim to clarify \emph{what happens if one drops the assumption of strong continuity w.r.t.~a norm of the underlying Banach space} in the considerations above and replaces it with strong continuity 
w.r.t.~a weaker Hausdorff locally convex topology. This question is not that artificial as it may seem since it takes its motivation e.g.~from parabolic equations modelled on spaces of continuous functions, \cite{kraaij2016,lunardi2013,lorenzirhandi2021,goldys2001,goldys2024}, abstractly reflected in the fact that the generalization of the infinitesimal generator is not densely defined. An important application we have in mind is within the rich theory on Markov-semigroups, see e.g.\ \cite{ethier2005,liggett2010,sharpe1988}, and stochastic processes arising for instance in stochastic differential equations. For a very recent reference, where the shortcoming of the classical semigroup framework above is nicely argued, is \cite{goldys2024}. The locally convex case is non-trivial as e.g.~solution theory, and in particular existence of integrals $T\ast f$ has to be revisited. To overcome such issues, it is worthwile to recall yet another aspect of Travis' result, which equivalently links $\mathrm{C}$-maximal regularity to the property that 
$(T(t))_{\geq 0}$ is of boun\-ded semivariation. In fact, the relation to semigroups of bounded-semivariation, also shows the duality of $\mathrm{C}$-admissibility and $\mathrm{L}^{1}$-estimates of the form 
\begin{equation*}
	\left\|A'T(\cdot)' x'\right\|_{\mathrm{L}^{1}([0,r];X')}\lesssim \|x'\|_{X'},\quad x\in X',
\end{equation*}
which in fact characterizes $\mathrm{L}^{1}$-maximal regularity, see also \cite{guerre1993}, of the dual semigroup \cite{kalton2008} in the classical case. Because of this connection, our efforts for $\mathrm{C}$-maximal regularity may also pave the way for analogous results for $\mathrm{L}^{1}$-maximal regularity. Compared to continuous functions, the latter suffers from intrinsic difficulties of properly defining mild solutions through Bochner-type integrals for functions with values in general topological spaces. 

We are of course not first to study strongly continuous semigroups on Hausdorff locally convex spaces. The broad literature on the subject can be grouped into a range of different assumptions on the topological spaces and the precise definition of the semigroups. The definition of a semigroup used in this article can at least be traced back to \cite{komura1968}. 
We also refer to the introduction in \cite{kraaij2016}, where a nice overview is given on the aspects of the different existing concepts. From the point of the classical theory, the easiest example class outside the classical theory is the one of dual semigroups of those which are strongly continuous with respect to the norm-topology on a Banach space. More generally, bi-continuous semigroups, e.g.~\cite{albanese2004,budde2019,farkas2004,kuehnemund2001}, have been designed to account for the deficiency of norm-strong-continuity.

In \prettyref{sect:notions}, we collect some preliminary statements about operator-valued functions of boun\-ded semivariation on Hausdorff locally convex spaces and continue with discussing fundamentals on inhomogeneous abstract Cauchy problems and $\mathrm{C}$-maximal regularity in \prettyref{sect:ACP_Cmax_reg}. This preparation allows us to use the notion of bounded semivariation in order to characterize 
$\mathrm{C}$-maximal regularity, \prettyref{sect:sg_bounded_semivar}. More generally, we study structured versions of maximal regularity acting only on a subspace of the full space, governed by some ``control operator $B$''. This cumulates in the first main result \prettyref{thm:Cmax_reg_equiv_Cadm}, establishing a generalization of the first part of Travis' result to Hausdorff locally convex spaces. In \prettyref{sect:C_adm} the notion of admissibility is introduced and our second main result yields  that 
$\mathrm{C}$-admissibility is equivalent to $\mathrm{C}$-maximal regularity under certain topological assumptions. This completes the full characterization in the spirit of Travis, \prettyref{cor:adm_max_reg_bounded_semivar}.
To justify the topological assumptions made in \prettyref{sect:C_adm}, we discuss spaces for which our main results hold, which reduces to the question when closed-graph-type theorems hold in a more 
general setting. This part, included in \prettyref{app:C_spaces}, might be interesting in its own right.

\section{Notions and preliminaries}
\label{sect:notions}

For a Hausdorff locally convex space $X$ over the field $\K\coloneqq\R$ or $\C$ we always denote by $\Gamma_{X}$ a 
fundamental system of seminorms. For two Hausdorff locally convex spaces $X$ and $Y$ we use the symbol $\mathcal{L}(X;Y)$ to denote the space of continuous linear maps from $X$ to $Y$. 
Further, we write $\mathcal{L}(X)\coloneqq \mathcal{L}(X;X)$ and $X'\coloneqq \mathcal{L}(X;\K)$. 
Let $I$ be a Hausdorff topological space. We call a map $\alpha\colon I\to \mathcal{L}(X;Y)$ \emph{strongly continuous} 
in $t_{0}\in I$ if the map $\alpha_{x}\colon I\to Y$, $\alpha_{x}(t)\coloneqq \alpha(t)x$, is continuous in $t_{0}$ for 
every $x\in X$. 
For two Hausdorff topological spaces $\Omega$ and $X$ we denote by $\mathrm{C}(\Omega;X)$ the space of continuous functions 
from $\Omega$ to $X$. If $X$ is Hausdorff locally convex, then we denote by $\mathrm{C}_{\operatorname{b}}(\Omega;X)$ 
the space of bounded continuous functions from $\Omega$ to $X$. 
If $X=\K$, we set $\mathrm{C}_{\operatorname{b}}(\Omega)\coloneqq\mathrm{C}_{\operatorname{b}}(\Omega;\K)$. 
If $\Omega$ is compact and $X$ Hausdorff locally convex, then we equip $\mathrm{C}(\Omega;X)=\mathrm{C}_{\operatorname{b}}(\Omega;X)$ 
with the Hausdorff locally convex topology induced by the system of seminorms given by 
\[
\|f\|_{p}\coloneqq\sup_{x\in\Omega}p(f(x)),\quad f\in \mathrm{C}(\Omega;X),
\]
for $p\in\Gamma_{X}$. Let $a,b\in\R$ with $a<b$. We denote by $\mathrm{C}^{1}([a,b];X)$ the space of continuously 
differentiable functions on $[a,b]$ with values in Hausdorff locally convex spaces $X$ where differentiability in $a$ 
means left-differentiability and in $b$ right-differentiability. 
For other unexplained notions on the theory of Hausdorff locally convex spaces we refer the reader 
to \cite{jarchow1981,kaballo2014,meisevogt1997,bonet1987}.

Let us turn to some preliminary results on functions of bounded semivariation and Riemann--Stieltjes integrals. 
We recall the definition of an $\mathcal{L}(X;Y)$-valued function of bounded semivariation from \cite[p.~589]{hoenig1973}. 
Let $a,b\in\R$ with $a<b$. We call a finite real sequence $d\coloneqq (d_{i})_{0\leq i\leq n}$ for $n\in\N$ a 
\emph{partition} of $[a,b]$ if $a=d_{0}$, $b=d_{n}$ and $d_{i-1}<d_{i}$ for all $i\in\N$ with $1\leq i\leq n$, 
and set $|d|\coloneqq n$. 
We denote by $D[a,b]$ the set of all partitions of $[a,b]$. Let $X$ and $Y$ be Hausdorff locally convex spaces with 
fundamental systems of seminorms $\Gamma_{X}$ and $\Gamma_{Y}$, respectively, and $\alpha\colon [a,b]\to \mathcal{L}(X;Y)$. 
We define for $p\in\Gamma_{X}$, $q\in\Gamma_{Y}$ and $d\in D[a,b]$
\[
SV_{q,p;d}(\alpha)\coloneqq\sup\Bigl\{q\Bigl(\sum_{i=1}^{|d|}\bigl(\alpha(d_{i})-\alpha(d_{i-1})\bigr)x_{i}\Bigr)\;|\;
\forall\;1\leq i\leq |d|:\;x_{i}\in X,\,p(x_{i})\leq 1\Bigr\}.
\]
Further, we set $SV_{q,p}(\alpha)\coloneqq SV_{q,p}^{[a,b]}(\alpha)\coloneqq\sup_{d\in D[a,b]}SV_{q,p;d}(\alpha)$ and 
$SV_{q,p}^{[a,a]}(\alpha)\coloneqq 0$. We say that $\alpha$ is of 
\emph{bounded semivariation} on $[a,b]$ if for every $q\in\Gamma_{Y}$ there is $p\in\Gamma_{X}$ such that 
$SV_{q,p}(\alpha)<\infty$. We note that the definition of bounded semivariation does not depend on the choice of $\Gamma_{X}$ 
and $\Gamma_{Y}$. Further, we make the following observation, which follows directly from the definition of bounded semivariation. 

\begin{rem}\label{rem:bounded_semivariation_subinterval}
Let $a,b,c\in\R$ with $a<c<b$, $X$ and $Y$ be Hausdorff locally convex spaces and $\alpha\colon [a,b]\to \mathcal{L}(X;Y)$. 
If $\alpha$ is of bounded semivariation on $[a,b]$, then $\alpha_{\mid [a,c]}$ and $\alpha_{\mid [c,b]}$ are of bounded semivariation on $[a,c]$ and $[c,b]$, respectively, and 
$SV_{q,p}^{[a,b]}(\alpha)=SV_{q,p}^{[a,c]}(\alpha)+SV_{q,p}^{[c,b]}(\alpha)$ for all $q\in\Gamma_{Y},p\in\Gamma_{X}$.
\end{rem}

With the definitions above we say that a function $f\colon[a,b]\to X$ is \emph{Riemann--Stieltjes integrable} w.r.t.~$\alpha\in\mathcal{L}(X;Y)$ 
if there is $y\in Y$ such that for all $\varepsilon>0$ and $q\in\Gamma_{Y}$ there is $\delta>0$ such that for all $d\in D[a,b]$ 
with $\Delta d\coloneqq\max\{d_{i}-d_{i-1}\;|\;1\leq i\leq |d|\}<\delta$ it holds for all finite real sequences 
$(c_{i})_{1\leq i\leq |d|}$ with $c_{i}\in [d_{i-1},d_{i}]$ for all $1\leq i\leq |d|$ that 
\[
q\Bigl(\sum_{i=1}^{|d|}\bigl(\alpha(d_{i})-\alpha(d_{i-1})\bigr)f(c_{i})-y\Bigr)<\varepsilon .
\]
In this case, we note that $y$ is unique since $Y$ is Hausdorff, and define the \emph{Riemann--Stieltjes integral} of $f$ 
w.r.t.~$\alpha$ by $\int_{a}^{b}f(s)\d\alpha(s)\coloneqq y$. Further, we set $\int_{a}^{a}f(s)\d\alpha(s)\coloneqq 0$. 
Again, we note that the definition of Riemann--Stieltjes 
integrability does not depend on the choice of $\Gamma_{X}$ and $\Gamma_{Y}$. 
For the sake of completeness we give a proof of the following result, 
which is stated in \cite{hoenig1973} without a proof (cf.~\cite[Proposition 2.1 (H\"onig), p.~426]{travis1981} 
for Banach spaces $X$ and $Y$). 

\begin{prop}[{\cite[Proposition 1', p.~589]{hoenig1973}}]\label{prop:RS_int_bounde_semivar}
Let $a,b\in\R$ with $a<b$, $X$ and $Y$ be Hausdorff locally convex spaces and $Y$ sequentially complete. 
If $f\in\mathrm{C}([a,b];X)$ and 
$\alpha\colon [a,b]\to \mathcal{L}(X;Y)$ is of bounded semivariation, then $f$ is Riemann--Stieltjes integrable w.r.t.~$\alpha$ 
and the linear map 
\[
I_{\alpha}\colon \mathrm{C}([a,b];X)\to Y,\;I_{\alpha}(g)\coloneqq\int_{a}^{b}g(s)\d\alpha(s),
\]
is continuous. 
\end{prop}
\begin{proof}
First, we prove that $f$ is Riemann--Stieltjes integrable w.r.t.~$\alpha$. 
Using the sequential completeness of $Y$, we obtain analogously to the reasoning given in \cite[Remark 4.1.5, p.~72]{lorenzi2001} 
that it is enough to show that the following Cauchy type condition holds, namely that for all $\varepsilon>0$ and 
$q\in\Gamma_{Y}$ there is $\delta>0$ such that for all $d\in D[a,b]$ 
with $\Delta d<\delta$ it holds for all finite real sequences $(c_{i,j})_{1\leq i\leq |d|}$ with $c_{i,j}\in [d_{i-1},d_{i}]$ 
for all $j=1,2$ and $1\leq i\leq |d|$ that 
\[
q\Bigl(\sum_{i=1}^{|d|}\bigl(\alpha(d_{i})-\alpha(d_{i-1})\bigr)\bigl(f(c_{i,1})-f(c_{i,2})\bigr)\Bigr)<\varepsilon .
\]
Let $\varepsilon>0$ and $q\in\Gamma_{Y}$. Since $\alpha$ is of bounded semivariation, there is $p\in\Gamma_{X}$ such that 
$SV_{q,p}(\alpha)<\infty$. Due to the compactness of $[a,b]$ the function $f$ is uniformly continuous and thus there is 
$\delta>0$ such that for all $t_{1},t_{2}\in[a,b]$ with $|t_{1}-t_{2}|<\delta$ it holds that $p(f(t_{1})-f(t_{2}))<\varepsilon$. 
Let $d\in D[a,b]$ with $\Delta d<\delta$ and $(c_{i,j})_{1\leq i\leq |d|}$ be two finite real sequences with 
$c_{i,j}\in [d_{i-1},d_{i}]$ for all $j=1,2$ and $1\leq i\leq |d|$. 
If there is some $1\leq i_{0}\leq |d|$ such that $p(f(c_{i_{0},1})-f(c_{i_{0},2}))=0$, then we set 
$x_{i_{0}}\coloneqq f(c_{i_{0},1})-f(c_{i_{0},2})$ and we have
$p(tx_{i_{0}})=tp(x_{i_{0}})=0$ for all $t>0$ and thus 
\[
tq((\alpha(d_{i})-\alpha(d_{i-1}))x_{i_{0}})=q((\alpha(d_{i})-\alpha(d_{i-1}))tx_{i_{0}})\leq SV_{q,p}(\alpha)<\infty,
\]
implying $q((\alpha(d_{i})-\alpha(d_{i-1}))x_{i_{0}})=0$. This implies 
\[
q\Bigl(\sum_{i=1}^{|d|}\bigl(\alpha(d_{i})-\alpha(d_{i-1})\bigr)\bigl(f(c_{i,1})-f(c_{i,2})\bigr)\Bigr)
\leq q\Bigl(\sum_{\substack{i=1\\i\neq i_{0}}}^{|d|}\bigl(\alpha(d_{i})-\alpha(d_{i-1})\bigr)\bigl(f(c_{i,1})-f(c_{i,2})\bigr)\Bigr)
\]
by the triangle inequality. Hence we may assume w.l.o.g.~that $p(f(c_{i,1})-f(c_{i,2}))\neq 0$ 
for all $1\leq i\leq |d|$. Then we have 
\begin{flalign*}
&\hspace{0.37cm}q\Bigl(\sum_{i=1}^{|d|}\bigl(\alpha(d_{i})-\alpha(d_{i-1})\bigr)\bigl(f(c_{i,1})-f(c_{i,2})\bigr)\Bigr)\\
&\leq q\Bigl(\sum_{i=1}^{|d|}\bigl(\alpha(d_{i})-\alpha(d_{i-1})\bigr)
       \Bigl(\frac{f(c_{i,1})-f(c_{i,2})}{p(f(c_{i,1})-f(c_{i,2}))}\Bigr)\Bigr)p(f(c_{i,1})-f(c_{i,2}))\\
&\leq SV_{q,p}(\alpha)\varepsilon ,
\end{flalign*}
which proves that $f$ is Riemann--Stieltjes integrable w.r.t.~$\alpha$. 

Second, from the definition of the Riemann--Stieltjes integral it follows that 
\begin{equation}\label{eq:RS_int_continuous}
q\Bigl(\int_{a}^{b}f(s)\d\alpha(s)\Bigr)\leq SV_{q,p}(\alpha)\sup_{s\in[a,b]}p(f(s)),
\end{equation}
yielding the continuity of $I_{\alpha}$.
\end{proof}

Moreover, we note the following observations, which follow from 
\prettyref{rem:bounded_semivariation_subinterval}, \prettyref{prop:RS_int_bounde_semivar} and the definitions of 
bounded semivariation and the Riemann--Stieltjes integral. 

\begin{rem}\label{rem:RS_int_subinterval_split}
Let $a,b,c\in\R$ with $a<c<b$, and $X$, $Y$ and $Z$ be Hausdorff locally convex spaces such that 
$Y$ is sequentially complete and $\alpha\colon [a,b]\to \mathcal{L}(X;Y)$ is of bounded semivariation. 
Then the following assertions hold.
\begin{enumerate}
\item[(a)] If $f\in\mathrm{C}([a,b];X)$, then $f$ is Riemann--Stieltjes integrable on $[a,c]$ and $[c,b]$ 
w.r.t.~$\alpha$ and 
\[
\int_{a}^{b}f(s)\d\alpha(s)=\int_{a}^{c}f(s)\d\alpha(s)+\int_{c}^{b}f(s)\d\alpha(s) .
\]
\item[(b)] Let $B\in\mathcal{L}(Y;Z)$ and $Z$ be sequentially complete. 
Then $B\alpha$ is of bounded semivariation and 
\[
B \int_{a}^{b}f(s)\d\alpha(s)= \int_{a}^{b}f(s)\d B\alpha(s) ,\quad f\in\mathrm{C}([a,b];X).
\]
\item[(c)] Let $B\in\mathcal{L}(Z;X)$. Then $\alpha(\cdot)B$ is of bounded semivariation and 
\[
 \int_{a}^{b}Bf(s)\d\alpha(s)= \int_{a}^{b}f(s)\d \alpha(s)B , \quad f\in\mathrm{C}([a,b];Z).
\]
\end{enumerate}
\end{rem}

\begin{cor}\label{cor:RS_int_to_zero}
Let $a,b,c\in\R$ with $a<c\leq b$, $X$ and $Y$ be Hausdorff locally convex spaces such that $Y$ is sequentially complete, 
$\alpha\colon [a,b]\to \mathcal{L}(X;Y)$ of bounded semivariation such that $\alpha$ is strongly continuous in $s=a$, 
$f\in\mathrm{C}([a,b];X)$ and $(\varphi_{t})_{t\in (a,c]}$ such that $\varphi_{t}\in\mathrm{C}([a,t];[a,b])$ for all 
$t\in (a,c]$. Then it holds 
\[
\lim_{t\to a\rlim}\int_{a}^{t}f(\varphi_{t}(s))\d\alpha(s)=0.
\]
\end{cor}
\begin{proof}
Let $q\in\Gamma_{Y}$. Since $\alpha$ is of bounded semivariation on $[a,b]$, there is $p\in\Gamma_{X}$ such that 
$SV_{q,p}^{[a,t]}(\alpha)\leq SV_{q,p}^{[a,b]}(\alpha)<\infty$ for all $a< t\leq b$. 
Using \prettyref{rem:bounded_semivariation_subinterval} and the strong continiuity of $\alpha$ in $a$, it follows similarly to 
\cite[5.2.2 Proposition (ii), p.~333]{bogachev2007} that $\lim_{t\to a\rlim}SV_{q,p}^{[a,t]}(\alpha)=SV_{q,p}^{[a,a]}(\alpha)=0$. Due to \eqref{eq:RS_int_continuous} this implies for $a< t\leq c$ that
\[
     q\Bigl(\int_{a}^{t}f(\varphi_{t}(s))\d\alpha(s)\Bigr)
\leq SV_{q,p}^{[a,t]}(\alpha)\sup_{s\in[a,t]}p(f(\varphi_{t}(s)))
\leq SV_{q,p}^{[a,t]}(\alpha)\sup_{s\in[a,b]}p(f(s))
\]
which leads to 
\[
 q\Bigl(\int_{a}^{t}f(\varphi_{t}(s))\d\alpha(s)\Bigr) \underset{t\to a\rlim}{\to} 0
\]
and proves our claim.
\end{proof}

In the case that $X$ and $Y$ are Banach spaces, $a=0$, $c=b$ and $\varphi_{t}(s)=s$ for all $t\in(0,b]$ and $s\in[0,t]$, \prettyref{cor:RS_int_to_zero} is given in \cite[Lemma 2.1, p.~426]{travis1981} without a proof. 

\section{The abstract Cauchy problem and \texorpdfstring{$\mathrm{C}$}{C}-maximal regularity}
\label{sect:ACP_Cmax_reg}

Let $r>0$, $X$ be a Hausdorff locally convex space, $A\colon D(A)\subset X\to X$ a linear map, 
$f\in\mathrm{C}([0,r];X)$ and $x\in X$. We consider the \emph{abstract Cauchy problem (ACP)} 
\begin{equation}\label{eq:ACP}
u'(t)=Au(t)+f(t),\quad t\in[0,r],\quad u(0)=x.
\end{equation}

Our goal of this section is to derive necessary and sufficient conditions for the solvability of the ACP \eqref{eq:ACP}. 
For this purpose we need to recall the concept of strongly continuous locally (or quasi-) equicontinuous semigroups and explain 
what we consider as a solution of the ACP \eqref{eq:ACP}.

\begin{defn}[{\cite[p.~294]{choe1985}, \cite[Definition 1.1, p.~259]{komura1968}}]\label{defn:semigroup}
Let $X$ be a Hausdorff locally convex space. A family $(T(t))_{t\geq 0}$ in $\mathcal{L}(X)$ is called
\begin{enumerate}
\item[(i)] a \emph{semigroup} on $X$ if $T(t+s)=T(t)T(s)$ and $T(0)=\id$ for all $t,s\geq 0$,
\item[(ii)] \emph{strongly continuous} if the map $[0,\infty)\to\mathcal{L}(X)$, $t\mapsto T(t)$, is strongly 
continuous in every $t_{0}\in [0,\infty)$, 
\item[(iii)] \emph{locally equicontinuous} if for a fundamental system of seminorms $\Gamma_{X}$ it holds
\[
\forall\;q\in\Gamma_{X},\,t_{0}\geq 0\;\exists\;p\in\Gamma_{X},\,C\geq 0\;
\forall\;t\in [0,t_{0}],\,x\in X:\;q(T(t)x)\leq Cp(x),
\] 
\item[(iv)] \emph{quasi-equicontinuous} if for a fundamental system of seminorms $\Gamma_{X}$ it holds
\[
\exists\;\omega\in\R\;\forall\;q\in\Gamma_{X}\;\exists\;p\in\Gamma_{X},\,C\geq 0\;
\forall\;t\geq 0,\,x\in X:\;q(\mathrm{e}^{-\omega t}T(t)x)\leq Cp(x). 
\] 
\end{enumerate}
\end{defn}

We note that the definitions of local equicontinuity and quasi-equicontinuity do not depend on the choice of 
$\Gamma_{X}$. Quasi-equicontinuity is also called \emph{exponential equicontinuity} 
(see \cite[Definition 2.1, p.~255--256]{albanese2013}). 
Clearly, quasi-equicontinuity implies local equicontinuity. Moreover, some results 
on automatic local equicontinuity are known. For instance, every strongly continuous semigroup on a barrelled or strong 
Mackey space $X$ is locally equicontinuous by \cite[Proposition 1.1, p.~259]{komura1968} and 
\cite[Lemma 3.2, p.~160]{kraaij2016}. 
Here, $X$ is called a \emph{strong Mackey space} if every $\sigma(X',X)$-compact set is equicontinuous in $X'$ 
(see \cite[p.~317]{sentilles1972}) where $\sigma(X',X)$ denotes the weak topology on $X'$. 
In particular, a strong Mackey space $X$ is a \emph{Mackey space}, i.e.~carries the Mackey topology $\mu(X,X')$, 
by \cite[p.~317]{sentilles1972} and the Mackey--Arens theorem. On Fr\'echet spaces every strongly continuous 
semigroup is already locally equicontinuous since Fr\'echet spaces are barrelled. 
On Banach spaces every strongly continuous semigroup is even quasi-equicontinuous 
by \cite[Chap.~I, 5.5 Proposition, p.~39]{engel_nagel2000}. 
However, the situation is different in general Fr\'echet spaces since there are strongly continuous semigroups 
on Fr\'echet spaces which are not quasi-equicontinuous by \cite[Remark 2.2 (iii), p.~256]{albanese2013}. 
On the other hand, leaving the realm of non-normable Fr\'echet spaces, every $\tau$-bi-continuous semigroup on a 
Saks space $(X,\|\cdot\|,\tau)$ is strongly continuous and locally, even quasi-, equicontinuous 
w.r.t.~the mixed topology $\gamma\coloneqq\gamma(\|\cdot\|,\tau)$ by \cite[Theorem 7.4, p.~180]{kraaij2016} 
(cf.~\cite[Theorem 3.17 (a), p.~13]{kruse_schwenninger2022}) if $(X,\gamma)$ is sequentially complete and a 
\emph{C-sequential space}, i.e.~every convex sequentially open subset of $(X,\gamma)$ is already open 
(see \cite[p.~273]{snipes1973}). We refer the reader to \prettyref{app:C_spaces} for the definition of a 
Saks space and the mixed topology.

We recall from \cite[p.~260]{komura1968} that the \emph{generator} $A\colon D(A)\to X$ of a strongly continuous semigroup 
$(T(t))_{t\geq 0}$ on a Hausdorff locally convex space $X$ is defined by 
\[
D(A)\coloneqq \Bigl\{x\in X\;|\;\lim_{t\to 0\rlim}\frac{T(t)x-x}{t}\;\text{exists in }X\Bigr\}
\] 
and 
\[
Ax\coloneqq \lim_{t\to 0\rlim}\frac{T(t)x-x}{t},\quad x\in D(A).
\]
If $X$ is sequentially complete, then $D(A)$ is dense in $X$ by \cite[Proposition 1.3, p.~261]{komura1968}. 

\begin{defn}\label{defn:strict_solution}
Let $r>0$, $X$ be a Hausdorff locally convex space, $A\colon D(A)\subset X\to X$ a linear map, 
$f\in\mathrm{C}([0,r];X)$ and $x\in X$. 
We call $u\in\mathrm{C}^{1}([0,r];X)$ a \emph{strict solution} of the ACP \eqref{eq:ACP} if $u(t)\in D(A)$ for all 
$t\in [0,r]$ and $u$ fulfils \eqref{eq:ACP}.
\end{defn}

If the ACP \eqref{eq:ACP} has a strict solution $u$, then $x=u(0)\in D(A)$ and $Au=u'-f\in\mathrm{C}([0,r];X)$. 
In the case that $X$ is a Banach space and $A$ the generator of a strongly continuous semigroup 
the definition of a strict solution is given in 
\cite[Definition 2.4.1, p.~50]{lorenzirhandi2021}. Strict solutions are also called strong solutions 
(see \cite[p.~425]{travis1981} or \cite[2.1 Definition, p.~35]{agase1987}) or classical solutions 
(see \cite[6.1 Definition, p.~145]{engel_nagel2000}) even though 
one should not confuse them with classical solutions in the sense of \cite[Definition 3.4.1, p.~70]{lorenzirhandi2021} 
(see also \prettyref{defn:classical_solution}). If $X$ is a Banach space, $A$ the generator of 
a strongly continuous semigroup and $u$ a strict solution of the ACP \eqref{eq:ACP}, then $u$ is unique and there 
is an explicit representation of $u$ by a variation of constants formula. 
To extend this result to strongly continuous locally equicontinuous semigroups on 
sequentially complete spaces, we need the concept of the convolution of the semigroup and the inhomogeneity $f$.

\begin{prop}\label{prop:riemann_int}
Let $r>0$, $X$ be a sequentially complete Hausdorff locally convex space, $(T(t))_{t\geq 0}$ a strongly continuous 
locally equicontinuous semigroup on $X$ and $f\in\mathrm{C}([0,r];X)$. 
Then the map $[0,t]\ni s\mapsto T(t-s)f(s)\in X$ is continuous and Riemann integrable for every $0\leq t\leq r$ 
and the \emph{convolution}
\[
T\ast f\colon [0,r] \to X,\;(T\ast f)(t)\coloneqq \int_{0}^{t}T(t-s)f(s)\d s,
\]
is continuous. Moreover, for every $q\in\Gamma_{X}$ there are $p\in\Gamma_{X}$ and $C\geq 0$ such that 
for all $t\in [0,r]$
\[
q((T\ast f)(t))\leq Ct \sup_{s\in [0,t]}p(f(s)).
\]
\end{prop}
\begin{proof}
The proof of the continuity of the map $[0,t]\ni s\mapsto T(t-s)f(s)\in X$ follows similarly to 
\cite[Proposition 5.3, p.~432]{kruse_seifert2022a} and 
then the Riemann integrability follows from \cite[Proposition 1.1, p.~232]{komatsu1964}. 
Thus the convolution $T\ast f$ is well-defined. Let $q\in\Gamma_{X}$. 
Due to the definition of the Riemann integral 
as a limit of Riemann sums and the local equicontinuity of the semigroup 
there are $p\in\Gamma_{X}$ and $C\geq 0$ such that
\[
q((T\ast f)(t))\leq t \sup_{s\in [0,t]}q(T(t-s)f(s))\leq Ct \sup_{s\in [0,r]}p(f(s)).
\]

Now, let us turn to the continuity of $T\ast f$ on $[0,r]$. 
Let $t\in [0,r)$ and $h\in (0,r-t)$. Then we have 
\[
 (T\ast f)(t+h)-(T\ast f)(t)
=\int_{0}^{t}(T(t+h-s)-T(t-s))f(s)\d s+\int_{t}^{t+h}T(t+h-s)f(s)\d s.
\]
Let $q\in\Gamma_{X}$. Then there are $p\in\Gamma_{X}$ and $C\geq 0$ such that 
\begin{flalign*}
&\hspace{0.37cm} q((T\ast f)(t+h)-(T\ast f)(t))\\
&\leq \sup_{s\in[0,t]}q((T(t+h-s)-T(t-s))f(s))+hC\sup_{s\in [t,t+h]}p(f(s))\\
&\leq \sup_{s\in[0,t]}q(T(t-s)(T(h)-\id)f(s))+hC\|f\|_{p}\\
&\leq C\sup_{s\in [0,r]}p((T(h)-\id)f(s))+hC\|f\|_{p}
\end{flalign*}
by the local equicontinuity of the semigroup. 
The local equicontinuity of the semigroup also implies that the family $(T(w)-\id)_{w\in [0,r-t]}$ 
in $\mathcal{L}(X)$ is equicontinuous. Thus $\lim_{h\to 0\rlim}\sup_{s\in[0,r]}p((T(h)-\id)f(s))=0$ 
by \cite[8.5.1 Theorem (b), p.~156]{jarchow1981}, the compactness of $f([0,r])$ and the strong continuity 
of the semigroup. This implies the right-continuity of $T\ast f$ on $[0,r)$. 
The left-continuity of $T\ast f$ on $(0,r]$ follows analogously and so $T\ast f$ is continuous on $[0,r]$.
\end{proof}

\begin{rem}\label{rem:strict_sol_is_mild_sol}
Let $r>0$, $X$ be a sequentially complete Hausdorff locally convex space, $(T(t))_{t\geq 0}$ a strongly continuous locally equicontinuous semigroup on $X$ with generator $A$, $f\in\mathrm{C}([0,r];X)$ and $x\in X$.
If the ACP \eqref{eq:ACP} has a strict solution $u$, then $x\in D(A)$, $Au\in\mathrm{C}([0,r];X)$, 
the strict solution is unique and fulfils
\[
u(t)= T(t)x+\int_{0}^{t}T(t-s)f(s)\d s= T(t)x+(T\ast f)(t), \quad t\in [0,r].
\]
The proof of this statement (cf.~\cite[2.1 Definition, p.~35]{agase1987}) is the same as in 
\cite[Proposition 2.4.3, p.~50]{lorenzirhandi2021}, where $X$ is a Banach space which is not relevant for the proof. 
Moreover, if $u$ is a strict solution, we have for every $q\in\Gamma_{X}$ that
\begin{equation}\label{eq:est_Au}
 \|Au\|_{q}
=\|u'-f\|_{q}  
\leq\|u'\|_{q}+\|f\|_{q}.
\end{equation}

Furthermore, if $x\in D(A)$, then $T(t)x\in D(A)$ for all $t\geq 0$ 
and the map $[0,\infty)\ni t\mapsto T(t)x\in X$ is continuously differentiable with continuous derivative given 
by $AT(t)x=T(t)Ax$ for all $t\geq 0$ by \cite[Proposition 1.2 (1), p.~260]{komura1968}. 
So, if the ACP \eqref{eq:ACP} has a strict solution $u$, then 
$T\ast f=u-T(\cdot)x$ is also continuously differentiable on $[0,r]$ and 
for every $q\in\Gamma_{X}$ there are $p\in\Gamma_{X}$ and $C\geq 0$ such that
\begin{align}\label{eq:estimate_strict_sol}
 \|u\|_{q,1}
\coloneqq & \sup_{k\in\{0,1\}}\|u^{(k)}\|_{q}
\leq \|T(\cdot)x\|_{q}+\|T\ast f\|_{q}+\|T(\cdot)Ax\|_{q}+\|(T\ast f)'\|_{q}\nonumber\\
\leq & C(p(x)+p(Ax)+\|f\|_{p})+\|(T\ast f)'\|_{q}
\end{align}  
by \prettyref{prop:riemann_int} and the local equicontinuity of $(T(t))_{t\geq 0}$. 
Further, $(T\ast f)(t)=u(t)-T(t)x\in D(A)$  and 
\begin{equation}\label{eq:deriv_conv}
(T\ast f)'(t)=u'(t)-AT(t)x= Au(t)+f(t)-AT(t)x=A(T\ast f)(t)+f(t)
\end{equation}
for all $t\in [0,r]$ if the ACP \eqref{eq:ACP} has a strict solution $u$. In particular, $A(T\ast f)=(T\ast f)'-f$ 
is continuous on $[0,r]$ in this case.
\end{rem}

\begin{defn}[{\cite[2.2 Definition, p.~35]{agase1987}}]\label{defn:mild_solution}
Let $r>0$, $X$ be a sequentially complete Hausdorff locally convex space, $(T(t))_{t\geq 0}$ a strongly continuous locally equicontinuous semigroup on $X$ with generator $A$, $f\in\mathrm{C}([0,r];X)$ and $x\in X$. The map 
\begin{equation}\label{eq:mild_solution}
u\colon [0,r]\to X,\; u(t)\coloneqq T(t)x+\int_{0}^{t}T(t-s)f(s)\d s= T(t)x+(T\ast f)(t),
\end{equation}
is called the \emph{mild solution} of the ACP \eqref{eq:ACP}.
\end{defn}

\begin{rem}\label{rem:mild_solution_cont}
Let $r>0$, $X$ be a sequentially complete Hausdorff locally convex space, $(T(t))_{t\geq 0}$ a strongly continuous semigroup locally equicontinuous semigroup on $X$ with generator $A$, $f\in\mathrm{C}([0,r];X)$ and $x\in X$. Then the mild solution $u$ 
of the ACP \eqref{eq:ACP} given by \eqref{eq:mild_solution} fulfils $u\in\mathrm{C}([0,r];X)$ by \prettyref{prop:riemann_int} and the strong continuity of the semigroup, and for every $q\in\Gamma_{X}$ 
there are $p\in\Gamma_{X}$ and $C\geq 0$ such that 
\[
\|u\|_{q}\leq C(p(x)+\|f\|_{p}).
\]
\end{rem}

\begin{rem}\label{rem:riemann_int_A_commute}
Let $a,b\in\R$ with $a<b$, $X$ and $Y$ be Hausdorff locally convex spaces and $f\colon [a,b]\to X$ Riemann integrable. 
If $A\colon D(A)\subset X\to Y$ is a sequentially closed linear map, $f([a,b])\subset D(A)$ and $Af$ Riemann integrable, then 
$\int_{a}^{b}f(s)\d s\in D(A)$ and 
\[
A\int_{a}^{b}f(s)\d s =\int_{a}^{b}Af(s)\d s.
\]
The proof of this statement is the same as in \cite[Proposition A.2.5 (ii), p.~419]{lorenzirhandi2021}, 
where $X$ and $Y$ are Banach spaces which is not relevant for the proof.   
\end{rem}

We need \prettyref{rem:riemann_int_A_commute} to generalise \cite[Remark 3.4.6, p.~73]{lorenzirhandi2021}, 
whose proof we adapt to our setting. This generalisation shows that even though the mild solution might not solve the 
ACP \eqref{eq:ACP} in a strict sense (see e.g.~\cite[Example 2.4.6, p.~51]{lorenzirhandi2021}) it always 
solves an integrated version of the ACP \eqref{eq:ACP}.

\begin{prop}\label{prop:mild_solution_solves_integrated_ACP}
Let $r>0$, $X$ be a sequentially complete Hausdorff locally convex space, $(T(t))_{t\geq 0}$ a strongly continuous semigroup locally equicontinuous on $X$ with generator $A$, $f\in\mathrm{C}([0,r];X)$ and $x\in X$. 
Then the mild solution $u$ of the ACP \eqref{eq:ACP} given by \eqref{eq:mild_solution} fulfils $\int_{0}^{t}u(s)\d s\in D(A)$ 
for all $t\in [0,r]$ and 
\[
u(t)=x+A\int_{0}^{t}u(s)\d s+\int_{0}^{t}f(s)\d s,\quad t\in [0,r].
\]
\end{prop}
\begin{proof}
Let $q\in\Gamma_{X}$. Then we have 
\begin{align*}
 q(T(t)f(s)-T(t_{0})f(s_{0}))
&\leq q(T(t-t_{0})f(s))+q(T(t_{0})(f(s)-f(s_{0}))\\
&\leq \sup_{w\in [0,r]}q(T(t-t_{0})f(w))+q(T(t_{0})(f(s)-f(s_{0}))
\end{align*}
for all $t,t_{0},s,s_{0}\in [0,r]$. Due to \cite[8.5.1 Theorem (b), p.~156]{jarchow1981} combined with the local equicontinuity of the semigroup, the compactness of $f([0,r])$ 
and the strong continuity of the semigroup the first summand converges to $0$ as $t\to t_{0}$. 
The second summand converges to $0$ as $s\to s_{0}$ since $f$ and $T(t_{0})$ are continuous. 
Hence the map $[0,r]^{2}\ni (t,s)\mapsto T(t)f(s)\in X$ is continuous. Since $u$ is continuous 
by \prettyref{rem:mild_solution_cont}, in particular Riemann integrable, we obtain by Fubini's theorem 
\begin{align*}
 \int_{0}^{t}u(s)\d s
&=\int_{0}^{t}T(s)x\d s +\int_{0}^{t}\int_{0}^{s}T(s-w)f(w)\d w\d s\\
&=\int_{0}^{t}T(s)x\d s +\int_{0}^{t}\int_{w}^{t}T(s-w)f(w)\d s\d w 
\end{align*}
for all $t\in [0,r]$. The first summand belongs to $D(A)$ by \cite[Corollary, p.~261]{komura1968}. 
Let us turn to the second summand. We have
\[
\int_{w}^{t}T(s-w)f(w)\d s=\int_{0}^{t-w}T(s)f(w)\d s\in D(A)
\] 
by a change of variables and 
\[
A\int_{w}^{t}T(s-w)f(w)\d s=T(t-w)f(w)-f(w)
\]
for all $w\in [0,t]$ by \cite[Corollary, p.~261]{komura1968}. Setting $\widetilde{f}\colon [0,t]\to X$, 
$\widetilde{f}(w)\coloneqq \int_{w}^{t}T(s-w)f(w)\d s$, we observe that $A\widetilde{f}$ is Riemann integrable on 
$[0,t]$ by \prettyref{prop:riemann_int}, which implies that $\int_{0}^{t}u(s)\d s\in D(A)$ 
by \prettyref{rem:riemann_int_A_commute}. Moreover, we obtain 
\[
A\int_{0}^{t}u(s)\d s=T(t)x-x+\int_{0}^{t}T(t-w)f(w)\d w-\int_{0}^{t}f(w)\d w
\]
by \cite[Corollary, p.~261]{komura1968} and \prettyref{rem:riemann_int_A_commute}, implying our statement. 
\end{proof}

\begin{rem}\label{rem:mean_int_to_zero}
Let $a,b\in\R$ with $a<b$, $X$ a sequentially complete Hausdorff locally convex spaces and $f\in\mathrm{C}([a,b];X)$. 
Then it holds
\[
\lim_{h\to 0\rlim}\frac{1}{h}\int_{t}^{t+h}f(s)\d s=f(t)\text{ and }\lim_{h\to 0\llim}\frac{1}{h}\int_{w+h}^{w}f(s)\d s=-f(w)
\]
for all $t\in [a,b)$ and $w\in (a,b]$. Indeed, for $q\in\Gamma_{X}$ and $t\in [a,b)$ we have  
\[
 q\Bigl(\frac{1}{h}\int_{t}^{t+h}f(s)\d s -f(t)\Bigr)
=q\Bigl(\frac{1}{h}\int_{t}^{t+h}f(s)-f(t)\d s\Bigr)
\leq \sup_{s\in [t,t+h]}q(f(s)-f(t))
\]
for all $h\in (0,b-t]$. The continuity of $f$ implies our statement for $t\in [a,b)$. The statement for $w\in (a,b]$ 
follows analogously. 
\end{rem}

Now, we are ready to give necessary and sufficient conditions for the existence of a strict solution of the 
ACP \eqref{eq:ACP} by means of the mild solution, which generalise \cite[Proposition 10.1.4, p.~110]{isem25}.

\begin{prop}\label{prop:nec_suff_mild_strict}
Let $r>0$, $X$ be a sequentially complete Hausdorff locally convex space, $(T(t))_{t\geq 0}$ a strongly continuous 
locally equicontinuous semigroup on $X$ with generator $A$, $f\in\mathrm{C}([0,r];X)$, $x\in X$ and $u$ 
the mild solution of the ACP \eqref{eq:ACP} given by \eqref{eq:mild_solution}. 
Then the following assertions are equivalent.
\begin{enumerate}
\item[(a)] $u$ is a strict solution of the ACP \eqref{eq:ACP}.
\item[(b)] $u(t)\in D(A)$ for all $t\in [0,r]$ and $Au\in\mathrm{C}([0,r];X)$.
\item[(c)] $u\in\mathrm{C}^{1}([0,r];X)$.
\end{enumerate}
\end{prop}
\begin{proof}
The implications (a)$\Rightarrow$(b) and (a)$\Rightarrow$(c) hold by the definition of a strict solution 
and \prettyref{rem:strict_sol_is_mild_sol}. Concerning the converse implications, we note the following observation. 
Due to \prettyref{prop:mild_solution_solves_integrated_ACP} we have
\begin{equation}\label{eq:strict_mild}
 \frac{u(t+h)-u(t)}{h}
=A\Bigl(\frac{1}{h}\int_{t}^{t+h}u(s)\d s\Bigr)+\frac{1}{h}\int_{t}^{t+h}f(s)\d s
\end{equation}
for all $t\in [0,r)$ and $h\in (0,r-t)$, and 
\[
 \frac{u(t+h)-u(t)}{h}
=-\frac{u(t)-u(t+h)}{h}
=A\Bigl(-\frac{1}{h}\int_{t+h}^{t}u(s)\d s\Bigr)-\frac{1}{h}\int_{t+h}^{t}f(s)\d s
\]
for all $t\in (0,r]$ and $h\in (-t,0)$, respectively.

(b)$\Rightarrow$(a) Let $u(t)\in D(A)$ for all $t\in [0,r]$ and $Au\in\mathrm{C}([0,r];X)$. We have 
\[
A\Bigl(\frac{1}{h}\int_{t}^{t+h}u(s)\d s\Bigr)=\frac{1}{h}\int_{t}^{t+h}Au(s)\d s
\]
for all $t\in [0,r)$ and $h\in (0,r-t)$, and 
\[
A\Bigl(-\frac{1}{h}\int_{t+h}^{t}u(s)\d s\Bigr)=-\frac{1}{h}\int_{t+h}^{t}Au(s)\d s
\]
for all $t\in (0,r]$ and $h\in (-t,0)$ by \prettyref{rem:riemann_int_A_commute}, respectively. Hence we obtain that $u$ is differentiable in $t\in [0,r]$ and 
$u'(t)=Au(t)+f(t)$ by \prettyref{rem:mean_int_to_zero} in combination with the closedness of $A$ by 
\cite[Proposition 1.4, p.~262]{komura1968}. Since $Au$ and $f$ are continuous, 
$u'$ is also continuous on $[0,r]$, yielding that $u$ is a strict solution. 

(c)$\Rightarrow$(a) Let $u\in\mathrm{C}^{1}([0,r];X)$. Then the left-hand side of \eqref{eq:strict_mild} 
converges to $u'$ and we get by \prettyref{rem:mean_int_to_zero} that 
\[
u'(t)-f(t)=\lim_{h\to 0\rlim}A\Bigl(\frac{1}{h}\int_{t}^{t+h}u(s)\d s\Bigr)
\]
for all $t\in [0,r)$. The closedness of $A$ and \prettyref{rem:mean_int_to_zero} imply that $u(t)\in D(A)$ and $u'(t)-f(t)=Au(t)$ for all 
$t\in [0,r]$ where the case $t=r$ is handled analogously. It follows that $u$ is a strict solution. 
\end{proof}

We may also phrase \prettyref{prop:nec_suff_mild_strict} in terms of the convolution $T\ast f$ 
(cf.~\cite[Chap.~4, Theorem 2.4, p.~107]{pazy1983} in the case of a Banach space $X$).

\begin{prop}\label{prop:pre_cmax_reg_strict_sol}
Let $r>0$ and $(T(t))_{t\geq 0}$ a strongly continuous locally 
equicontinuous semigroup on a sequentially complete Hausdorff locally convex space $X$ with generator $A$, 
$f\in\mathrm{C}([0,r];X)$ and $x\in D(A)$. Then the following assertions are equivalent.
\begin{enumerate}
\item[(a)] The ACP \eqref{eq:ACP} has a strict solution.
\item[(b)] $(T\ast f)(t)\in D(A)$ for all $t\in [0,r]$ and $A(T\ast f)\in \mathrm{C}([0,r];X)$. 
\item[(c)] $T\ast f\in\mathrm{C}^{1}([0,r];X)$.
\end{enumerate}
\end{prop}
\begin{proof}
First, we observe that $T(t)x\in D(A)$ for all $t\geq 0$ 
and the map $[0,\infty)\ni t\mapsto T(t)x\in X$ is continuously differentiable with continuous derivative 
equal to $AT(\cdot)x$ by \cite[Proposition 1.2 (1), p.~260]{komura1968} since $x\in D(A)$. 

(a)$\Rightarrow$(b) This implication follows from \prettyref{rem:strict_sol_is_mild_sol}.

(b)$\Rightarrow$(a) Let $u$ denote the mild solution of the ACP \eqref{eq:ACP} given by \eqref{eq:mild_solution}. 
Then $u(t)= T(t)x-(T\ast f)(t)\in D(A)$ and the map 
$[0,r]\ni t\mapsto Au(t)= AT(t)x-A(T\ast f)(t)\in X$ is well-defined and continuous by our first observation 
and our assumption. Hence the mild solution $u$ is a strict solution of the ACP \eqref{eq:ACP} 
by \prettyref{prop:nec_suff_mild_strict}.

(a)$\Leftrightarrow$(c) Using that the mild solution $u$ fulfils $u(t)= T(t)x-(T\ast f)(t)$ for all $t\in[0,r]$, we 
deduce our statement from our first observation, \prettyref{rem:strict_sol_is_mild_sol} 
and \prettyref{prop:nec_suff_mild_strict}.
\end{proof}

Further, we have the following sufficient conditions which guarantee the existence of a strict solution. 
They are well-known in the case of strongly continuous semigroups on Banach spaces 
(see e.g.~\cite[Theorem 10.1.3, p.~110]{isem25}, \cite[Theorem, p.~84]{goldstein1985}, 
\cite[Theorem 2.4.7, p.~51]{lorenzirhandi2021} and \cite[Chap.~4, Corollaries 2.5, 2.6, p.~107--108]{pazy1983}). 
The underlying idea of our proof comes from the proof of \cite[Theorem 10.1.3, p.~110]{isem25}, 
which we adjust to our setting. 

\begin{cor}\label{cor:suff_cond_strict_sol}
Let $r>0$, $X$ be a sequentially complete Hausdorff locally convex space, 
$(T(t))_{t\geq 0}$ a strongly continuous locally equicontinuous semigroup on $X$ with generator $A$ 
and $x\in D(A)$. If 
\begin{enumerate}
\item[(i)] $f\in \mathrm{C}([0,r];X)$, $f(t)\in D(A)$ for all $t\in [0,r]$ and $Af\in\mathrm{C}([0,r];X)$, or
\item[(ii)] $f\in\mathrm{C}^{1}([0,r];X)$, 
\end{enumerate}
then the ACP \eqref{eq:ACP} has a strict solution $u$. Moreover, for every $q\in\Gamma_{X}$ there are 
$p\in\Gamma_{X}$ and $C\geq 0$ such that 
\[
\|u\|_{q,1}+\|Au\|_{q}\leq  C(p(x)+p(Ax)+\|f\|_{p}+\|Af\|_{p})
\]
in case (i) and 
\[
\|u\|_{q,1}+\|Au\|_{q}\leq  C(p(x)+p(Ax)+\|f\|_{p,1})
\]
in case (ii).
\end{cor}
\begin{proof}
(i) If $f\in\mathrm{C}([0,r];X)$, $f(t)\in D(A)$ for all $t\in [0,r]$ and $Af\in \mathrm{C}([0,r];X)$, then 
$AT(t-s)f(s)=T(t-s)Af(s)$ for all $t\in[0,r]$ and $s\in [0,t]$ by \cite[Proposition 1.2 (1), p.~260]{komura1968} 
and the map $[0,t]\ni s\mapsto AT(t-s)f(s)=T(t-s)Af(s)\in X$ is continuous for every $t\in [0,r]$ 
by \prettyref{prop:riemann_int} because $Af\in\mathrm{C}([0,r];X)$. This implies 
that $(T\ast f)(t)=\int_{0}^{t}T(t-s)f(s)\d s\in D(A)$ and 
\begin{equation}\label{eq:A_conv}
 A(T\ast f)(t)
=\int_{0}^{t}T(t-s)Af(s)\d s
=(T\ast Af)(t)
\end{equation}
for every $t\in [0,r]$ by \prettyref{rem:riemann_int_A_commute}. 
Due to \prettyref{prop:riemann_int} $A(T\ast f)=T\ast Af\in \mathrm{C}([0,r];X)$ because $Af\in \mathrm{C}([0,r];X)$. 
Thus the ACP \eqref{eq:ACP} has a strict solution by \prettyref{prop:pre_cmax_reg_strict_sol}.

(ii) By a change of vairables we have
\begin{align*}
&\phantom{=} \frac{(T\ast f)(t+h)-(T\ast f)(t)}{h}
 =\frac{1}{h}\Bigl(\int_{0}^{t+h}T(s)f(t+h-s)\d s-\int_{0}^{t}T(s)f(t-s)\d s\Bigr)\\
&=\int_{0}^{t}T(s)\frac{f(t+h-s)-f(t-s)}{h}\d s+\frac{1}{h}\int_{t}^{t+h}T(s)f(t+h-s)\d s
\eqqcolon I_{1,h}+I_{2,h}
\end{align*}
for all $t\in [0,r)$ and $h\in (0,r-t)$. Since $f\in\mathrm{C}^{1}([0,r];X)$, 
we obtain that $f'$ is uniformly continuous on $[0,r]$. Thus for every $\varepsilon>0$ and $p\in\Gamma_{X}$ 
there is $\delta>0$ such that $p(f'(t_{1})-f'(t_{2}))<\varepsilon$ for all $t_{1},t_{2}\in[0,r]$ 
with $|t_{1}-t_{2}|<\delta$. Let $h<\delta$ and $q\in\Gamma_{X}$. 
Then there are $p\in\Gamma_{X}$ and $C\geq 0$ such that 
\begin{align*}
  q\Bigl(I_{1,h}-\int_{0}^{t}T(s)f'(t-s)\d s\Bigr)
&=q\Bigl(\int_{0}^{t}T(s)\frac{1}{h}\int_{0}^{h}f'(w+t-s)-f'(t-s)\d w\d s\Bigr)\\
&\leq C \sup_{s\in [0,t]}\sup_{w\in [0,h]}p(f'(w+t-s)-f'(t-s)) \leq C\varepsilon
\end{align*}
by the fundamental theorem of calculus and the local equicontinuity of the semigroup. 
Hence we have $\lim_{h\to 0\rlim}I_{1,h}=\int_{0}^{t}T(s)f'(t-s)\d s$ for all $t\in [0,r)$. 
Similarly to \prettyref{rem:mean_int_to_zero} we get $\lim_{h\to 0\rlim}I_{2,h}=T(t)f(0)$ for all $t\in [0,r)$. 
Therefore $T\ast f$ is right-differentiable on $[0,r)$. Analogously we can show that $T\ast f$ is 
left-differentiable on $(0,r]$ and that the left- and right-derivatives coincide on $(0,r)$. 
Thus $T\ast f$ is differentiable on $[0,r]$ and 
\begin{equation}\label{eq:deriv_conv_inhom_deriv}
(T\ast f)'(t)=\int_{0}^{t}T(s)f'(t-s)\d s+T(t)f(0)=(T\ast f')(t)+T(t)f(0)
\end{equation}
for all $t\in[0,r]$ where we used a change of variables in the last equation. 
Due to \prettyref{prop:riemann_int}, $f\in\mathrm{C}^{1}([0,r];X)$ and the strong continuity of the semigroup 
$(T\ast f)'$ is also continuous. 
Thus the ACP \eqref{eq:ACP} has a strict solution by \prettyref{prop:pre_cmax_reg_strict_sol}. 

Let us turn to the estimate of $\|u\|_{q,1}+\|Au\|_{q}$ in our statement. 
Let $q\in\Gamma_X$. By \eqref{eq:estimate_strict_sol} there are $p_0\in\Gamma_X$ and $C_0\geq 0$ such that 
\begin{equation}\label{eq:cont_dep}
\|u\|_{q,1}+\|Au\|_{q}\underset{\eqref{eq:est_Au}, \eqref{eq:estimate_strict_sol}}{\leq} 
 2C_{0}(p_{0}(x)+p_{0}(Ax)+\|f\|_{p_{0}})+2\|(T\ast f)'\|_q+\|f\|_q .
\end{equation}
In case (i) this implies
\begin{align*}
\|u\|_{q,1}+\|Au\|_{q}
&\underset{\mathclap{\eqref{eq:deriv_conv}}}{\leq} 2C_{0}(p_{0}(x)+p_{0}(Ax)+\|f\|_{p_{0}})+2\|A(T\ast f)+f\|_q+\|f\|_q\\
&\underset{\mathclap{\eqref{eq:A_conv}}}{\leq} 2C_{0}(p_{0}(x)+p_{0}(Ax)+\|f\|_{p_{0}})+2\|(T\ast Af)+f\|_q+\|f\|_q\\
&\leq 2C_{0}(p_{0}(x)+p_{0}(Ax)+\|f\|_{p_{0}})+2\|T\ast Af\|_q+3\|f\|_q.
\end{align*}
By \prettyref{prop:riemann_int} there are $p_1\in\Gamma_X$ and $C_1\geq 0$ such that $\|T\ast Af\|_q\leq rC_1\|Af\|_{p_1}$. 
Moreover, as $\Gamma_{X}$ is a fundamental system of seminorms, there are $p\in\Gamma_{X}$ and $C_{2}\geq 0$ 
such that $\max\{q,p_0,p_1\}\leq C_{2} p$. Hence we get
\begin{align*}
\|u\|_{q,1}+\|Au\|_{q}&
\leq 2C_{0}(p_{0}(x)+p_{0}(Ax)+\|f\|_{p_{0}})+2\|T\ast Af\|_q+3\|f\|_q\\
&\leq 2C_{0}C_2(p(x)+p(Ax)+\|f\|_{p})+2rC_1C_2\|Af\|_{p}+3C_2\|f\|_{p}\\
&\leq C(p(x)+p(Ax)+\|f\|_{p}+\|Af\|_{p})
\end{align*}
with $C\coloneqq \max\{2C_0C_2+3C_2,2rC_1C_2\}$, which proves the estimate in case (i). 
The estimate in case (ii) follows similarly from \eqref{eq:deriv_conv_inhom_deriv}, \eqref{eq:cont_dep}, 
\prettyref{prop:riemann_int} and the local equicontinuity of the semigroup.
\end{proof}

Let $r>0$ and $(T(t))_{t\geq 0}$ be a strongly continuous locally equicontinuous semigroup on 
a sequentially complete Hausdorff locally convex space $X$ with generator $A$. 
By \prettyref{prop:pre_cmax_reg_strict_sol} we see that given $f\in\mathrm{C}([0,r];X)$ the ACP \eqref{eq:ACP} 
has a strict solution for every $x\in D(A)$ if condition (b) of \prettyref{prop:pre_cmax_reg_strict_sol} is fulfilled. 
Now, we want to strictly solve the ACP \eqref{eq:ACP} for every $f$ from certain subspaces of $\mathrm{C}([0,r];X)$, 
namely subspaces of the form $B(\mathrm{C}([0,r];U))$ for some operator $B\in\mathcal{L}(U;X)$ on another 
Hausdorff locally convex space $U$. So we want to strictly solve the ACP 
\begin{equation}\label{eq:ACP_B}
u'(t)=Au(t)+Bf(t),\quad t\in[0,r],\quad u(0)=x.
\end{equation}
for every $x\in D(A)$ and $f\in\mathrm{C}([0,r];U)$. 
The ACP \eqref{eq:ACP_B} is called a \emph{control system}, $X$ the \emph{state space}, $U$ the 
\emph{input or control space}, $B$ the \emph{control operator}, 
$f$ the \emph{input or control function} and $u$ the \emph{state function} 
(see e.g.~\cite[p.~432--433]{kruse_seifert2022a} and \cite[p.~527]{weiss1989a}). 
This motivates the following definition of continuous maximal regularity, in short $\mathrm{C}$-maximal regularity.

\begin{defn}
Let $r>0$, $X$ be a sequentially complete Hausdorff locally convex space and $(T(t))_{t\geq 0}$ a strongly continuous 
locally equicontinuous semigroup on $X$ with generator $A$. 
Let $U$ be a Hausdorff locally convex space and $B\in\mathcal{L}(U;X)$. 
We say that $(T(t))_{t\geq 0}$ satisfies \emph{$\mathrm{C}$-maximal regularity} for $(B,r)$ 
if $(T\ast Bf)(t)\in D(A)$ for all $t\in [0,r]$ and $A(T\ast Bf)\in \mathrm{C}([0,r];X)$ 
for all $f\in \mathrm{C}([0,r];U)$. 
If $U=X$ and $B=\id$, then we just say that $(T(t))_{t\geq 0}$ satisfies 
\emph{$\mathrm{C}$-maximal regularity} for $r$ instead of $(\id,r)$. 
\end{defn}

In the case that $X$ is a Banach space, $U=X$ and $B=\id$, 
this definition reduces to the one given in Baillon's work 
\cite[Condition $(\star)$ in Th\'eor\`eme 1, p.~757]{baillon1980}, see also \cite[p.~47]{eberhardt1992}, 
\cite[Definition 17.2.40, p.~614]{hytonen2023} and \cite[Definition 1.1, p.~144]{JacoSchwWint2022}. 
In the literature, maximal regularity is often coined 
``with respect to the generator $A$'' rather than the semigroup generated by $A$. We deliberately chose for the 
(equivalent) wording in line with \cite{eberhardt1992} and \cite{JacoSchwWint2022} as it stresses that the semigroup 
is used in the definition (rather than only the generator).

\begin{cor}\label{cor:cmax_reg_strict_sol}
Let $r>0$, $X$ be a sequentially complete Hausdorff locally convex space and $(T(t))_{t\geq 0}$ a strongly continuous 
locally equicontinuous semigroup on $X$ with generator $A$ and $x\in D(A)$. 
Let $U$ be a Hausdorff locally convex space and $B\in\mathcal{L}(U;X)$. 
Then the following assertions are equivalent.
\begin{enumerate}
\item[(a)] The ACP \eqref{eq:ACP_B} has a strict solution for all $f\in \mathrm{C}([0,r];U)$.
\item[(b)] $(T(t))_{t\geq 0}$ satisfies $\mathrm{C}$-maximal regularity for $(B,r)$. 
\end{enumerate}
\end{cor}

If the semigroup satisfies $\mathrm{C}$-maximal regularity for $(B,r)$ for some $r>0$, then it satisfies 
$\mathrm{C}$-maximal regularity for $(B,r)$ for any $r>0$ which we prove next.

\begin{prop}\label{prop:adm_cmax_for_some_r_for_all}
Let $X$ be a sequentially complete Hausdorff locally convex space and $(T(t))_{t\geq 0}$ a strongly continuous 
locally equicontinuous semigroup on $X$ with generator $A$. 
Let $U$ be a Hausdorff locally convex space and $B\in\mathcal{L}(U;X)$.
\begin{enumerate} 
\item[(a)] Let $r>0$. If $(T\ast Bf)(r)\in D(A)$ for all $f\in\mathrm{C}([0,r];U)$, 
then $(T\ast Bf)(t)\in D(A)$ for all $f\in\mathrm{C}([0,r];U)$ and $t\in [0,r]$.
\item[(b)] If $(T\ast Bf)(r)\in D(A)$ for all $f\in\mathrm{C}([0,r];U)$ for some $r>0$, 
then it also holds for all $r>0$.
\item[(c)] If $(T(t))_{t\geq 0}$ satisfies $\mathrm{C}$-maximal regularity for $(B,r)$ for some $r>0$, then 
it satisfies $\mathrm{C}$-maximal regularity for $(B,r)$ for all $r>0$.
\item[(d)] If $(T(t))_{t\geq 0}$ satisfies $\mathrm{C}$-maximal regularity for some $r>0$, then 
it satisfies $\mathrm{C}$-maximal regularity for $(B,r)$.
\end{enumerate}
\end{prop}
\begin{proof}
First, we remark that $Bf\in\mathrm{C}([0,r];X)$ for any $f\in\mathrm{C}([0,r];U)$ and $r>0$ 
since $B\in\mathcal{L}(U;X)$. Thus $(T\ast Bf)(t)\in X$ for every $t\in [0,r]$ by \prettyref{prop:riemann_int}.

(a) We use the idea of \cite[Chap.~III, 3.3 Corollary, p.~187]{engel_nagel2000} to prove part (a). 
Let $f\in\mathrm{C}([0,r];U)$. For $t\in [0,r]$ we define the function $f_{t}\colon[0,r]\to U$ by 
\[
f_{t}(s)\coloneqq
\begin{cases}
f(0)&,\; s\in [0,r-t],\\
f(s+t-r)&,\; s\in (r-t,r].
\end{cases}
\]
We observe that $f_{t}\in \mathrm{C}([0,r];U)$ and it is easily checked that
\[
 (T\ast Bf)(t)
=\int_{0}^{t}T(t-s)Bf(s)\d s
=(T\ast Bf_{t})(r)-\int_{t}^{r}T(s)Bf(0)\d s
\]
by a change of variables. Further, we have that 
\[
\int_{t}^{r}T(s)Bf(0)\d s=T(t)\int_{0}^{r-t}T(s)Bf(0)\d s\in D(A)
\]
by a change of variables and \cite[Corollary, p.~261]{komura1968}. 
Noting that $(T\ast Bf_{t})(r)\in D(A)$, we deduce our statement. 

(b) Let $r>0$ such that $(T\ast Bf)(r)\in D(A)$ for all $f\in\mathrm{C}([0,r];U)$. Let $r_{0}>0$.
First, we consider the case $r_{0}<r$. Let $f\in\mathrm{C}([0,r_{0}];U)$. We define $f_{1}\colon[0,r]\to U$ by 
\[
f_{1}(s)\coloneqq
\begin{cases}
f(s)&,\; s\in [0,r_{0}],\\
f(r_{0})&,\; s\in (r_{0},r].
\end{cases}
\]
Then $f_{1}\in \mathrm{C}([0,r];U)$ and $(T\ast Bf)(r_{0})=(T\ast Bf_{1})(r_{0})\in D(A)$ by part (a). 

Second, we consider the case $r_{0}>r$. Let $f\in\mathrm{C}([0,r_{0}];U)$. We define the function 
$f_{2}\colon[0,r]\to U$, $f_{2}(s)\coloneqq f(s+r_{0}-r)$, and note that $f_{2}\in\mathrm{C}([0,r];U)$. 
Moreover, we note that
\[
(T\ast Bf)(r_{0})=\int_{0}^{r_{0}-r}T(r_{0}-s)Bf(s)\d s+\int_{r_{0}-r}^{r_{0}}T(r_{0}-s)Bf(s)\d s.
\]
and
\[
 \int_{r_{0}-r}^{r_{0}}T(r_{0}-s)Bf(s)\d s
=\int_{0}^{r}T(r-s)Bf(r_{0}-r+s)\d s
=(T\ast Bf_{2})(r)\in D(A)
\]
by a change of variables as well as 
\[
 \int_{0}^{r_{0}-r}T(r_{0}-s)Bf(s)\d s
=T(r)\int_{0}^{r_{0}-r}T(r_{0}-r-s)Bf(s)\d s
=T(r)\bigl((T\ast Bf)(r_{0}-r)\bigr).
\]
If $r_{0}\leq 2r$, then $r_{0}-r\in [0,r]$ and $(T\ast Bf)(r_{0}-r)\in D(A)$ by part (a) since 
$f_{\mid [0,r]}\in\mathrm{C}([0,r];U)$, which implies $T(r)((T\ast Bf)(r_{0}-r))\in D(A)$ 
by \cite[Proposition 1.2 (1), p.~260]{komura1968}. 
Thus $(T\ast Bf)(r_{0})\in D(A)$ for any $r_{0}>0$ such that $r_{0}\leq 2r$. 
By repetition we obtain our statement.

(c) Let $r>0$ be such that $(T(t))_{t\geq 0}$ satisfies $\mathrm{C}$-maximal regularity for $(B,r)$. 
Let $r_{0}>0$. By our assumption and parts (a) and (b) we obtain that $(T\ast Bf)(t)\in D(A)$ for all 
$f\in\mathrm{C}([0,r_{0}];X)$ and $t\in [0,r_{0}]$. 

First, we consider the case $r_{0}<r$. Let $f\in\mathrm{C}([0,r_{0}];X)$ and 
define $f_{1}\in\mathrm{C}([0,r];X)$ as in part (b). Then $A(T\ast Bf)=A(T\ast Bf_{1})$ on $[0,r_{0}]$, implying 
that $A(T\ast Bf)\in\mathrm{C}([0,r_{0}];X)$ by the $\mathrm{C}$-maximal regularity for $(B,r)$.

Second, we consider the case $r_{0}>r$. Let $f\in\mathrm{C}([0,r_{0}];X)$. Then $A(T\ast Bf)$ is continuous on $[0,r]$ 
by the $\mathrm{C}$-maximal regularity for $(B,r)$. Let $t\in [r,r_{0}]$. 
Then we have
\begin{align*}
  \int_{r}^{t}T(t-s)Bf(s)\d s
&=(T\ast Bf)(t)-\int_{0}^{r}T(t-s)Bf(s)\d s\\
&=(T\ast Bf)(t)-T(t-r)\bigl((T\ast Bf)(r)\bigr)\in D(A)
\end{align*}
by \cite[Proposition 1.2 (1), p.~260]{komura1968} and 
\[
 A(T\ast Bf)(t)
=A\int_{0}^{r}T(t-s)Bf(s)\d s+A\int_{r}^{t}T(t-s)Bf(s)\d s.
\]
We note that 
\[
 A\int_{0}^{r}T(t-s)Bf(s)\d s
=T(t-r)A\int_{0}^{r}T(r-s)Bf(s)\d s
=T(t-r)A(T\ast B
f)(r)
\]
by \cite[Proposition 1.2 (1), p.~260]{komura1968}. If $r_{0}\leq 2r$, then $t-r\in [0,r]$ and 
the function $f_{3}\colon[0,r]\to X$, $f_{3}(s)\coloneqq f(s+r)$, is continuous. We obtain
\[
 A\int_{r}^{t}T(t-s)Bf(s)\d s
=A\int_{0}^{t-r}T(t-r-s)Bf(s+r)\d s
=A(T\ast Bf_{3})(t-r)
\]
by a change of variables and so 
\[
A(T\ast Bf)(t)=T(t-r)A(T\ast Bf)(r)+A(T\ast Bf_{3})(t-r).
\]
The right-hand side is continuous in the variable $t-r\in [0,r]$ by the strong continuity and 
the $\mathrm{C}$-maximal regularity for $(B,r)$ of the semigroup. Hence we obtain that the semigroup satisfies 
$\mathrm{C}$-maximal regularity for $(B,r_{0})$ for any $r_{0}>0$ such that $r_{0}\leq 2r$. 
By repetition we obtain our statement.

(d) This statement is obvious since $Bf\in\mathrm{C}([0,r];X)$ for every $f\in\mathrm{C}([0,r];U)$.
\end{proof}

As a consequence of \prettyref{cor:cmax_reg_strict_sol} and \prettyref{prop:adm_cmax_for_some_r_for_all} (c) 
we obtain the following statement.

\begin{cor}
Let $X$ be a sequentially complete Hausdorff locally convex space and $(T(t))_{t\geq 0}$ a strongly continuous 
locally equicontinuous semigroup on $X$ with generator $A$ and $x\in X$. 
Let $U$ be a Hausdorff locally convex space and $B\in\mathcal{L}(U;X)$.
Then the following assertions are equivalent.
\begin{enumerate}
\item[(a)] The ACP \eqref{eq:ACP_B} has a strict solution for all $f\in \mathrm{C}([0,r];U)$ for some $r>0$.
\item[(b)] The ACP \eqref{eq:ACP_B} has a strict solution for all $f\in \mathrm{C}([0,r];U)$ for all $r>0$.
\end{enumerate}
\end{cor}

\section{Families of bounded semivariation}
\label{sect:sg_bounded_semivar}

Let $r>0$, $X$ be a sequentially complete Hausdorff locally convex space and $(T(t))_{t\geq 0}$ a strongly continuous 
locally equicontinuous semigroup on $X$ with generator $A$. Let $U$ be a Hausdorff locally convex space 
and $B\in\mathcal{L}(U;X)$. If $(T(t))_{t\geq 0}$ satisfies $\mathrm{C}$-maximal regularity for $(B,r)$, 
then $(T\ast Bf)(r)\in D(A)$ for all $f\in\mathrm{C}([0,r];U)$. In this section we want to show that 
the converse of this statement is also true, at least for certain spaces $X$, $U$ and operators $B$. 
Clearly, the converse holds without any further restrictions on $X$, $U$ and $B$ if $A\in\mathcal{L}(X)$. 
To cover other cases, we introduce the notion of families of bounded semivariation.

\begin{defn}
Let $r>0$, $X$ and $U$ be Hausdorff locally convex spaces and $(S(t))_{t\geq 0}$ a family in $\mathcal{L}(U;X)$. 
We say that $(S(t))_{t\in [0,r]}$ is of \emph{bounded semivariation} if the map $[0,r]\ni t\mapsto S(t)\in 
\mathcal{L}(U;X)$ is of bounded semivariation. 
\end{defn} 

\begin{rem}\label{rem:bounded_semivar_for_some_r_all_r}
Let $X$ and $U$ be Hausdorff locally convex spaces, $B\in\mathcal{L}(U;X)$ and $(T(t))_{t\geq 0}$ a semigroup on $X$. 
\begin{enumerate}
\item[(a)] If $(T(t)B)_{t\in [0,r]}$ is of bounded semivariation for some $r>0$, then it is of bounded semivariation for all $r>0$. 
Indeed, let $d\coloneqq (d_{i})_{0\leq i\leq n}\in D[0,2r]$. Then there is some $1\leq j\leq n$ such that 
$d_{j-1}< r\leq d_{j}$. 
W.l.o.g.~$d_{j}=r$ (otherwise we set $\widetilde{d}_{i}\coloneqq d_{i}$ for $0\leq i\leq j-1$, 
$\widetilde{d}_{j}\coloneqq r$ and $\widetilde{d}_{i+1}\coloneqq d_{i}$ for $j\leq i\leq n$). 
Then $(d_{i})_{0\leq i\leq j}\in D[0,r]$ and $(d_{i}-r)_{j\leq i\leq n}\in D[0,r]$. 
Let $q_{0}\in\Gamma_{X}$. Then there are $q_{1}\in\Gamma_{X}$ and $C_{1}\geq 0$ such that 
$q_{0}(T(r)x)\leq C_{1}q_{1}(x)$ for all $x\in X$ because $T(r)\in\mathcal{L}(X)$. 
Moreover, as $\Gamma_{X}$ is a fundamental system of seminorms, there are $q_{2}\in\Gamma_{X}$ and $C_{2}\geq 0$ 
such that $\max\{q_{0},q_{1}\}\leq C_{2} q_{2}$. 
Since $(T(t)B)_{t\in [0,r]}$ is of bounded semivariation, 
there is $p\in\Gamma_{U}$ such that $SV_{q_{2},p}^{[0,r]}(T(\cdot)B)<\infty$. 
For $x_{i}\in U$ with $p(x_{i})\leq 1$ for all $1\leq i\leq n$ we have
\begin{flalign*}
&\hspace{0.37cm} q_{0}\Bigl(\sum_{i=1}^{n}\bigl(T(d_{i})-T(d_{i-1})\bigr)Bx_{i}\Bigr)\\
&\leq q_{0}\Bigl(\sum_{i=1}^{j}\bigl(T(d_{i})-T(d_{i-1})\bigr)Bx_{i}\Bigr)
 +q_{0}\Bigl(T(r)\sum_{i=j+1}^{n}\bigl(T(d_{i}-r)-T(d_{i-1}-r)\bigr)Bx_{i}\Bigr)\\
&\leq  C_{2}q_{2}\Bigl(\sum_{i=1}^{j}\bigl(T(d_{i})-T(d_{i-1})\bigr)Bx_{i}\Bigr)
 +C_{1}C_{2}q_{2}\Bigl(\sum_{i=j+1}^{n}\bigl(T(d_{i}-r)-T(d_{i-1}-r)\bigr)Bx_{i}\Bigr)\\
&\leq  (1+C_{1})C_{2}SV_{q_{2},p}^{[0,r]}(T(\cdot)B), 
\end{flalign*}
implying that $(T(t)B)_{t\in [0,2r]}$ is of bounded semivariation. By repetition 
and \prettyref{rem:bounded_semivariation_subinterval} we obtain our statement.
\item[(b)] If $(T(t))_{t\in [0,r]}$ is of bounded semivariation for some $r>0$, then the family $(T(t)B)_{t\in [0,r]}$ is also of bounded semivariation by \prettyref{rem:RS_int_subinterval_split} (c).
\end{enumerate}
\end{rem}

Our next statement is a generalisation of \cite[Lemma 3.1, p.~427]{travis1981}.

\begin{prop}\label{prop:gen_commute_r_int_rs_int}
Let $r>0$, $X$ be a sequentially complete Hausdorff locally convex space and $(T(t))_{t\geq 0}$ a strongly continuous 
locally equicontinuous semigroup on $X$ with generator $A$. 
Let $U$ be a Hausdorff locally convex space and $B\in\mathcal{L}(U;X)$. 
If $f\in\mathrm{C}([0,r];U)$ and $(T(t)B)_{t\in [0,r]}$ is of bounded semivariation, 
then it holds for every $t\in [0,r]$ that 
$(T\ast Bf)(t)\in D(A)$ and 
\[
A(T\ast Bf)(t)=A\int_{0}^{t}T(t-s)Bf(s)\d s=\int_{0}^{t}f(s)\d T(t-s)B.
\]
\end{prop}
\begin{proof}
Let $t\in [0,r]$. We note that the map $[0,t]\ni s\mapsto T(t-s)B\in \mathcal{L}(X)$ is of bounded semivariation since 
$(T(t)B)_{t\in [0,r]}$ is of bounded semivariation. Therefore the Riemann--Stieltjes integral 
$\int_{0}^{t}f(s)\d T(t-s)B$ exists by \prettyref{prop:RS_int_bounde_semivar}. Further, the Riemann integral 
$(T\ast Bf)(t)=\int_{0}^{t}T(t-s)Bf(s)\d s$ exists in $X$ by \prettyref{prop:riemann_int}. 
For $n\in\N$ we define a partition 
$d^{n}\coloneqq(d_{i}^{n})_{0\leq i\leq n}\in D[0,t]$ by $d_{i}^{n}\coloneqq it/n$ for all $0\leq i\leq n$ 
and the function
$g_{n}\colon [0,t]\to U$ given by $g_{n}(s)\coloneqq T(t-s)Bf(d_{i}^{n})$ for all $d_{i-1}^{n}<s\leq d_{i}^{n}$ where 
$1\leq i\leq n$ and $g_{n}(0)\coloneqq T(t)Bf(0)$. We observe that $g_{n}$ is Riemann integrable on $[0,t]$ and 
\begin{equation}\label{eq:gen_commute_r_int_rs_int}
\int_{0}^{t}g_{n}(s)\d s=\sum_{i=1}^{n}\,\int_{d_{i-1}^{n}}^{d_{i}^{n}}T(t-s)Bf(d_{i}^{n})\d s\in D(A)
\end{equation}
by \cite[Corollary, p.~261]{komura1968}. Let $q\in\Gamma_{X}$. We have the estimate
\[
q\Bigl(\int_{0}^{t}g_{n}(s)-T(t-s)Bf(s)\d s\Bigr)\leq t\sup_{s\in [0,t]}q(g_{n}(s)-T(t-s)Bf(s)).
\]
Let $\varepsilon >0$. Due to the local equicontinuity of the semigroup there are $p\in\Gamma_{X}$ and $C\geq 0$ such that 
$q(T(t-s)x)\leq C p(x)$ for all $s\in [0,t]$ and $x\in X$. Since $Bf$ is uniformly continuous on $[0,t]$, there is $\delta>0$ such that for all $t_{1},t_{2}\in[0,t]$ with $|t_{1}-t_{2}|<\delta$ it holds $p(Bf(t_{1})-Bf(t_{2}))<\varepsilon$. Then we get for all $n\in\N$ 
such that $\Delta d^{n}=t/n<\delta$ the estimate
\begin{align*}
  \sup_{s\in[0,t]}q(g_{n}(s)-T(t-s)Bf(s))
&=\sup_{i\in\N}\sup_{d_{i-1}^{n}<s\leq d_{i}^{n}}q(T(t-s)B(f(d_{i}^{n})-f(s)))\\
&\leq C \sup_{i\in\N}\sup_{d_{i-1}^{n}<s\leq d_{i}^{n}}p(Bf(d_{i}^{n})-Bf(s))
 \leq C\varepsilon .
\end{align*}
This implies the convergence of $(\int_{0}^{t}g_{n}(s)\d s)_{n\in\N}$ to $\int_{0}^{t}T(t-s)Bf(s)\d s$ in $X$. 
Further, we deduce from \eqref{eq:gen_commute_r_int_rs_int} and \cite[Corollary, p.~261]{komura1968} that
\[
A\int_{0}^{t}g_{n}(s)\d s=\sum_{i=1}^{n}(T(t-d_{i}^{n})-T(t-d_{i-1}^{n}))Bf(d_{i}^{n})\underset{n\to\infty}{\to}
\int_{0}^{t}f(s)\d T(t-s)B.
\]
The closedness of $A$ by \cite[Proposition 1.4, p.~262]{komura1968} implies $\int_{0}^{t}T(t-s)Bf(s)\d s\in D(A)$ and 
$A\int_{0}^{t}T(t-s)Bf(s)\d s=\int_{0}^{t}f(s)\d T(t-s)B$.
\end{proof}

Next, we transfer \cite[Lemma 3.2, p.~427]{travis1981} to our setting. 
Let $X$ be a Hausdorff locally convex space and $(T(t))_{t\geq 0}$ a strongly continuous semigroup on $X$. 
We say that a topological linear subspace $U$ of $X$ is an \emph{invariant subspace} w.r.t.~$(T(t))_{t\geq 0}$ 
if $T(t)x\in U$ for all $t\geq 0$ and $x\in U$. Let $U$ be an invariant subspace of $X$ w.r.t.~$(T(t))_{t\geq 0}$ 
and $B\in\mathcal{L}(U;X)$. We say that $B$ is \emph{commuting} with $(T(t))_{t\geq 0}$ if $T(t)Bx=BT(t)x$ for all 
$t\geq 0$ and $x\in U$. Clearly, if $U=X$, then $T(t_{0})$ commutes with $(T(t))_{t\geq 0}$ for any fixed $t_{0}\geq 0$. 
Further, we call $X$ a \emph{generalised Schwartz space} if every bounded subset of $X$ is already precompact (see 
\cite[5.2.50 Definition, p.~93]{kruse2023}). In particular, Schwartz spaces and semi-Montel spaces are generalised 
Schwartz spaces but infinite-dimensional Banach spaces (with the norm-topology) are not.

\begin{prop}\label{prop:continuity_rs_int}
Let $r>0$, $X$ be a sequentially complete Hausdorff locally convex space and $(T(t))_{t\geq 0}$ a strongly continuous 
locally equicontinuous semigroup on $X$. Let $U$ be a Hausdorff locally convex space and $B\in\mathcal{L}(U;X)$.
If $f\in\mathrm{C}([0,r];U)$ and $(T(t)B)_{t\in [0,r]}$ is of bounded semivariation, then 
the map 
\[
I_{T,B,f}\colon [0,r]\to X,\;I_{T,B,f}(t)\coloneqq\int_{0}^{t}f(s)\d T(t-s)B, 
\]
is right-continuous. If in addition 
\begin{enumerate}
\item[(i)] $U$ is an invariant subspace of $X$ and $B$ commuting with $(T(t))_{t\geq 0}$, or
\item[(ii)] $X$ is a generalised Schwartz space,
\end{enumerate}
then $I_{T,B,f}$ is continuous. 
\end{prop}
\begin{proof}
We show that the map $I_{T,B,f}$ is right-continuous. 
Let $t\in [0,r)$ and $h\in (0,r-t]$. Then we have 
by \prettyref{rem:RS_int_subinterval_split} (a) and (b) that
\begin{flalign*}
&\hspace{0.37cm}\int_{0}^{t+h}f(s)\d T(t+h-s)B-\int_{0}^{t}f(s)\d T(t-s)B\\
&=T(h)\int_{0}^{t}f(s)\d T(t-s)B+\int_{t}^{t+h}f(s)\d T(t+h-s)B-\int_{0}^{t}f(s)\d T(t-s)B\\
&=(T(h)-\id)\int_{0}^{t}f(s)\d T(t-s)B+\int_{0}^{h}f(t+h-s)\d T(t-s)B
\end{flalign*}
The first summand converges to $0$ in $X$ as $h\to 0\rlim$ by the strong continuity of the semigroup 
and the second summand converges to $0$ in $X$ as $h\to 0\rlim$ by \prettyref{cor:RS_int_to_zero} with 
$\varphi_{h}(s)\coloneqq t+h-s$ for $s\in [0,h]$ and $c\coloneqq r-t$. 
Thus $I_{T,B,f}$ is right-continuous on $[0,r)$. 

(i) Now, suppose that $U$ is an invariant subspace of $X$ and $B$ commuting with $(T(t))_{t\geq 0}$. 
We show that $I_{T,B,f}$ is also left-continuous in this case 
with coinciding left- and right-limits. 
Let $t\in (0,r]$ and $h\in (0,t]$. Since $U$ is an invariant subspace of $X$ and $B$ commuting with the semigroup, 
we observe that 
\[
T(t-s)Bx=T(t-h-s)T(h)Bx=T(t-h-s)BT(h)x
\]
for all $s\in [0,t-h]$ and $x\in U$ and $T(h)f\in\mathrm{C}([0,r];U)$. 
Then we obtain by \prettyref{rem:RS_int_subinterval_split} (a) and (c) and our observation above that
\begin{flalign}\label{eq:continuity_rs_int}
&\hspace{0.37cm}\int_{0}^{t}f(s)\d T(t-s)B-\int_{0}^{t-h}f(s)\d T(t-h-s)B\nonumber\\
&=\int_{0}^{t-h}T(h)f(s)\d T(t-h-s)B+\int_{t-h}^{t}f(s)\d T(t-s)B-\int_{0}^{t-h}f(s)\d T(t-h-s)B\nonumber\\
&=\int_{0}^{t-h}(T(h)-\id)f(s)\d T(t-h-s)B+\int_{0}^{h}f(t-s)\d T(s)B.
\end{flalign}
The second summand converges to $0$ in $X$ as $h\to 0\rlim$ by \prettyref{cor:RS_int_to_zero} 
with $\varphi_{h}(s)\coloneqq t-s$ for $s\in [0,h]$ and $c\coloneqq t$. 
The first summand converges to $0$ in $X$ as $h\to 0\rlim$, too. Indeed, let $q\in\Gamma_{X}$ where $\Gamma_{X}$ 
is a fundamental system of seminorms for the topology of $X$. Then the restricted system 
$\Gamma_{U}\coloneqq (\Gamma_{X})_{\mid U}$ is a fundamental system of seminorms for the topology of $U$.
Since $(T(t)B)_{t\in [0,r]}$ is of bounded semivariation, the map $[0,t-h]\ni s\mapsto T(t-h-s)B\in\mathcal{L}(U;X)$ 
is also of bounded semivariation. 
Hence there is $p\in\Gamma_{U}$ such that it holds by \eqref{eq:RS_int_continuous} that
\begin{flalign*}
&\hspace{0.37cm} q\Bigl(\int_{0}^{t-h}(T(h)-\id)f(s)\d T(t-h-s)B\Bigr)\\
&\leq SV_{q,p}^{[0,t-h]}(T(t-h-\cdot)B)\sup_{s\in[0,t-h]}p((T(h)-\id)f(s))\\
&\leq SV_{q,p}^{[0,r]}(T(\cdot)B)\sup_{s\in[0,r]}p((T(h)-\id)f(s)).
\end{flalign*}
The local equicontinuity of the semigroup implies that the family $(T(w)-\id)_{w\in [0,t]}$ 
in $\mathcal{L}(X)$ is equicontinuous. Thus $\lim_{h\to 0\rlim}\sup_{s\in[0,r]}p((T(h)-\id)f(s))=0$ 
by \cite[8.5.1 Theorem (b), p.~156]{jarchow1981}, the compactness of $f([0,r])$, the strong continuity 
of the semigroup and since $p=\widetilde{p}_{\mid U}$ for some $\widetilde{p}\in\Gamma_{X}$. 
We conclude that the first summand of \eqref{eq:continuity_rs_int} converges to $0$ in $X$ as $h\to 0\rlim$, 
yielding the left-continuity of $I_{T,B,f}$ on $(0,r]$.

(ii) Suppose that $X$ is a generalised Schwartz space. Instead of \eqref{eq:continuity_rs_int} we have 
by \prettyref{rem:RS_int_subinterval_split} (a) and (b) for all $t\in (0,r]$ and $h\in (0,t]$ that
\begin{flalign}\label{eq:continuity_rs_int_1}
&\hspace{0.37cm}\int_{0}^{t}f(s)\d T(t-s)B-\int_{0}^{t-h}f(s)\d T(t-h-s)B\nonumber\\
&=(T(h)-\id)\int_{0}^{t-h}f(s)\d T(t-h-s)B+\int_{0}^{h}f(t-s)\d T(s)B.
\end{flalign}
The second summand is the same as in \eqref{eq:continuity_rs_int} and converges to $0$ in $X$ as $h\to 0\rlim$. 
Let us turn to the first summand. Let $q\in\Gamma_{X}$. 
Since $(T(t)B)_{t\in [0,r]}$ is of bounded semivariation, there is $p\in\Gamma_{U}$ 
such that
\[
    q\Bigl(\int_{0}^{t-h}f(s)\d T(t-h-s)B\Bigr)
\leq SV_{q,p}^{[0,r]}(T(\cdot)B)\sup_{s\in[0,r]}p(f(s)).
\]
This implies that the set 
\[
M\coloneqq \Big\{\int_{0}^{t-h}f(s)\d T(t-h-s)B\;|\; t\in (0,r],\,h\in (0,t]\Bigr\}
\]
is bounded in $X$. Since $X$ is a generalised Schwartz space, the set $M$ is also precompact in $X$. 
As in part (i) it follows that $\lim_{h\to 0\rlim}\sup_{x\in M}q((T(h)-\id)x)=0$ 
by \cite[8.5.1 Theorem (b), p.~156]{jarchow1981}, the precompactness of $M$, the local equicontinuity and 
the strong continuity of the semigroup. Hence the first summand of \eqref{eq:continuity_rs_int_1} 
converges to $0$ in $X$ as $h\to 0\rlim$, implying the left-continuity of $I_{T,B,f}$ on $(0,r]$.
\end{proof}

\prettyref{prop:continuity_rs_int} (i) generalises \cite[Lemma 3.2, p.~427]{travis1981} where $X=U$ is a Banach space 
and $B=\id$. Looking at the proof of the addendum in case (ii), the equation \prettyref{eq:continuity_rs_int_1} 
still holds without the assumption that $X$ is a generalised Schwartz space and its second summand converges to $0$ 
as $h\to 0\rlim$. But we do not know how to control the first summand without this assumption.

\begin{prob}\label{prob:continuity_rs_int}
Let $r>0$, $X$ be a sequentially complete Hausdorff locally convex space and $(T(t))_{t\geq 0}$ a strongly continuous 
locally equicontinuous semigroup on $X$. Let $U$ be a Hausdorff locally convex space, 
$B\in\mathcal{L}(U;X)$, $f\in\mathrm{C}([0,r];U)$ and $(T(t)B)_{t\in [0,r]}$ of bounded semivariation. 
Is $I_{T,B,f}$ continuous without the restrictions (i) and (ii) of \prettyref{prop:continuity_rs_int}?
\end{prob}

\begin{cor}\label{cor:orbit_in_domain_for_pos_time}
Let $r>0$, $X$ be a sequentially complete Hausdorff locally convex space and $(T(t))_{t\geq 0}$ a strongly continuous 
locally equicontinuous semigroup on $X$ with generator $A$. 
If $(T(t))_{t\in [0,r]}$ is of bounded semivariation, then $T(t)X\subset D(A)$ for all $t>0$ 
and the map $[0,\infty)\ni t\mapsto AtT(t)x\in X$ is continuous for all $x\in X$. 
\end{cor}
\begin{proof}
Let $U\coloneqq X$ and $B\coloneqq \id$. For $x\in X$ we define the map $f\colon [0,r]\to X$, $f(t)\coloneqq T(t)x$. Then $f\in\mathrm{C}([0,r];X)$ since the semigroup 
is strongly continuous, and 
\[
tT(t)x=\int_{0}^{t}T(t-s)f(s)\d s=(T\ast f)(t)\in D(A)
\]
for all $t\in [0,r]$ by \prettyref{prop:gen_commute_r_int_rs_int}. Hence $T(t)x\in D(A)$ for all $t\in (0,r]$ and $x\in X$.  
Due to \prettyref{prop:gen_commute_r_int_rs_int} and \prettyref{prop:continuity_rs_int} the map 
$I_{T,\id,f}\colon [0,r]\to X$, $I_{T,\id,f}(t)=AtT(t)x$, is continuous. 
We deduce our statement from \prettyref{rem:bounded_semivar_for_some_r_all_r} (a). 
\end{proof}

Now, we are ready to introduce the spaces we mentioned in the beginning of this section. 
Let $r>0$, $X$ be a sequentially complete Hausdorff locally convex space and $(T(t))_{t\geq 0}$ a strongly continuous 
locally equicontinuous semigroup on $X$ with generator $A$. 
Let $U$ be a Hausdorff locally convex space and $B\in\mathcal{L}(U;X)$.
The linear map 
\[
\Psi_{r}^{B}\colon \mathrm{C}([0,r];U)\to X,\;\Psi_{r}^{B}(f)\coloneqq (T\ast Bf)(r)=\int_{0}^{r}T(r-s)Bf(s)\d s,
\]
is continuous by the estimate in \prettyref{prop:riemann_int} and since $B\in\mathcal{L}(U;X)$. 
Suppose that $(T\ast Bf)(r)\in D(A)$ for every $f\in\mathrm{C}([0,r];U)$. 
Since $\Psi_{r}^{B}(f)\in D(A)$ for all $f\in\mathrm{C}([0,r];U)$, the map 
$A\Psi_{r}^{B}\colon\mathrm{C}([0,r];U)\to X$ 
is well-defined and linear. Further, the linear map $A\Psi_{r}^{B}$ is closed as $A$ is closed by 
\cite[Proposition 1.4, p.~262]{komura1968} and $\Psi_{r}^{B}$ is continuous.

\begin{defn}
Let $r>0$, $X$ be a sequentially complete Hausdorff locally convex space, $(T(t))_{t\geq 0}$ a strongly continuous 
locally equicontinuous semigroup on $X$ with generator $A$. 
Let $U$ be a Hausdorff locally convex space and $B\in\mathcal{L}(U;X)$ 
such that $\Psi_{r}^{B}(f)=(T\ast Bf)(r)\in D(A)$ for every $f\in\mathrm{C}([0,r];U)$. 
We call $X$ a \emph{$3\mathrm{C}_{U,A,r}^{B}$-space} if the \textbf{closed} linear map $A\Psi_{r}^{B}$ 
from the space of \textbf{continuous} functions $\mathrm{C}([0,r];U)$ to $X$ is already \textbf{continuous}. 
If $U=X$ and $B=\id$, then we just say that $X$ is a \emph{$3\mathrm{C}_{A,r}$-space} instead of 
$3\mathrm{C}_{X,A,r}^{\id}$-space.
\end{defn}

\begin{rem}\label{rem:3CA_cont_gen}
Let $X$ be a sequentially complete Hausdorff locally convex space and $(T(t))_{t\geq 0}$ be a strongly continuous 
locally equicontinuous semigroup on $X$ with generator $A\in\mathcal{L}(X)$. 
Then $\Psi_{r}^{B}(f)\in X=D(A)$ for every $f\in\mathrm{C}([0,r];U)$ and any $r>0$, Hausdorff locally convex space 
$U$ and $B\in\mathcal{L}(U;X)$. Further, 
$X$ is a $3\mathrm{C}_{U,A,r}^{B}$-space for any $r>0$, Hausdorff locally convex space 
$U$ and $B\in\mathcal{L}(U;X)$ since $A\in\mathcal{L}(X)$.
\end{rem}

Let $r_{1},r_{2}>0$ and $U$ be a Hausdorff locally convex space. 
Then the map $\mathrm{C}([0,r_{1}];U)\ni f\mapsto f(\tfrac{r_{1}}{r_{2}}\,\cdot)\in \mathrm{C}([0,r_{2}];U)$ is 
a topological isomorphism. 

\begin{defn}
Let $X$ and $U$ be Hausdorff locally convex spaces. We call $X$ a \emph{$3\mathrm{C}_{U}$-space} if for some 
(equivalently all) $r>0$ any closed linear map $C\colon \mathrm{C}([0,r];U)\to X$ is already continuous. 
If $U=X$, then we just say that $X$ is a \emph{$3\mathrm{C}$-space} instead of $3\mathrm{C}_{X}$-space.
\end{defn}

\begin{rem}\label{3C_implies_3CA}
Let $X$ and $U$ be Hausdorff locally convex spaces and $B\in\mathcal{L}(U;X)$.
If $X$ is a sequentially complete $3\mathrm{C}_{U}$-space, then $X$ is a $3\mathrm{C}_{U,A,r}^{B}$-space 
for any $r>0$ and any strongly continuous locally equicontinuous semigroup $(T(t))_{t\geq 0}$ on $X$ with 
generator $A$ such that $\Psi_{r}^{B}(f)=(T\ast Bf)(r)\in D(A)$ for every $f\in\mathrm{C}([0,r];U)$. 
\end{rem}

A list of complete $3\mathrm{C}_{U}$-spaces and $3\mathrm{C}$-spaces can be found in
 \prettyref{prop:C_spaces} and \prettyref{cor:list_C_spaces}, respectively.
In particular, every Fr\'echet space $X$ is a $3\mathrm{C}_{U}$-space for any Fr\'echet space $U$.
Let us come to our main theorem of this section, which we prove by adapting the proof of 
\cite[Proposition 3.1, p.~428]{travis1981}.

\begin{thm}\label{thm:Cmax_reg_equiv_Cadm}
Let $r>0$, $X$ be a sequentially complete Hausdorff locally convex space and $(T(t))_{t\geq 0}$ a strongly continuous 
locally equicontinuous semigroup on $X$ with generator $A$. 
Let $U$ be a Hausdorff locally convex space and $B\in\mathcal{L}(U;X)$. 
Consider the following assertions.
\begin{enumerate}
\item[(a)] $(T(t))_{t\geq 0}$ satisfies $\mathrm{C}$-maximal regularity for $(B,r)$. 
\item[(b)] $(T\ast Bf)(t)\in D(A)$ for all $t\in [0,r]$ and $A(T\ast Bf)$ is right-continuous on $[0,r]$
for all $f\in \mathrm{C}([0,r];U)$. 
\item[(c)] $(T\ast Bf)(r)\in D(A)$ for all $f\in\mathrm{C}([0,r];U)$.
\item[(d)] $(T(t)B)_{t\in [0,r]}$ is of bounded semivariation.
\end{enumerate}
Then we have $(a)\Rightarrow (b)$ and $(d)\Rightarrow(b)\Rightarrow (c)$. 
If $X$ is a $3\mathrm{C}_{U,A,r}^{B}$-space, then $(c)\Rightarrow(d)$. 
If 
\begin{enumerate}
\item[(i)] $U$ is an invariant subspace of $X$ and $B$ commuting with $(T(t))_{t\geq 0}$, or
\item[(ii)] $X$ is a generalised Schwartz space,
\end{enumerate} 
then $(d)\Rightarrow (a)$.
\end{thm}
\begin{proof}
The implications (a)$\Rightarrow$(b) and (b)$\Rightarrow$(c) are obvious. 
Further, the implication (d)$\Rightarrow$(b) is a consequence 
of \prettyref{prop:gen_commute_r_int_rs_int} and \prettyref{prop:continuity_rs_int}. 
The implication (d)$\Rightarrow$(a) also follows
from \prettyref{prop:gen_commute_r_int_rs_int} and \prettyref{prop:continuity_rs_int} if 
condition (i) or (ii) is assumed.

Let us turn to the implication (c)$\Rightarrow$(d) and suppose that $X$ is a $3C_{U,A,r}^{B}$-space. 
Then the closed linear map $A\Psi_{r}^{B}$ is continuous by the definition of a $3C_{U,A,r}^{B}$-space. 
For $n\in\N$ let $d\coloneqq (d_{i})_{0\leq i\leq n}\in D[0,r]$ and $\varepsilon >0$ such that 
$\varepsilon<\min\{d_{i}-d_{i-1}\;|\;1\leq i\leq n\}$. 
Let $q\in\Gamma_{X}$. Due to the continuity of $A\Psi_{r}^{B}$ there are $p\in\Gamma_{U}$ and $C\geq 0$ such that 
$q(A\Psi_{r}^{B}(f))\leq C\sup_{s\in[0,r]}p(f(s))$ for all $f\in\mathrm{C}([0,r];U)$. 
For a finite sequence $(x_{i})_{1\leq i\leq n+1}$ in $U$ such that $p(x_{i})\leq 1$ for all $1\leq i\leq n+1$ 
we define the continuous function $f_{d,\varepsilon}\colon [0,r]\to U$ given for $1\leq i\leq n$ by 
\[
f_{d,\varepsilon}(s)\coloneqq
\begin{cases}
x_{i}&,\;s\in [d_{i-1},d_{i}-\varepsilon),\\
x_{i+1}+(x_{i+1}-x_{i})(s-d_{i})/\varepsilon &,\; s\in [d_{i}-\varepsilon, d_{i}] .
\end{cases}
\]
Due to Bauer's maximum principle (see e.g.~\cite[Korollar, p.~392]{bauer1958}) we have 
\[
\sup_{s\in [d_{i}-\varepsilon, d_{i}]}p(f_{d,\varepsilon}(s))=\max\{p(x_{i}),p(x_{i+1})\}
\]
for all $1\leq i\leq n$ and thus 
\[
\sup_{s\in[0,r]}p(f_{d,\varepsilon}(s))=\max\{p(x_{i})\;|\;1\leq i\leq n+1\}\leq 1, 
\]
implying 
\begin{equation}\label{eq:c_admiss_max_reg}
 q(A\Psi_{r}^{B}(f_{d,\varepsilon}))
\leq C\sup_{s\in[0,r]}p(f_{d,\varepsilon}(s))
\leq C.
\end{equation}
Using \cite[Corollary, p.~261]{komura1968} and \prettyref{rem:RS_int_subinterval_split} (a), 
it follows by the same calculation 
as in the proof of \cite[Proposition 3.1, p.~428]{travis1981} that 
\begin{align*}
 A\Psi_{r}^{B}(f_{d,\varepsilon})
=&-\sum_{i=1}^{n}(T(r-d_{i})-T(r-d_{i-1}))Bx_{i}
  +\sum_{i=1}^{n}\frac{1}{\varepsilon}\int_{d_{i}-\varepsilon}^{d_{i}}T(r-s)B(x_{i+1}-x_{i})\d s \\
&-\sum_{i=1}^{n}T(r-d_{i})B(x_{i+1}-x_{i}).
\end{align*}
This implies that 
\begin{flalign*}
&\hspace{0.37cm}q\Bigl(\sum_{i=1}^{n}(T(r-d_{i})-T(r-d_{i-1}))Bx_{i}\Bigr)\\
&\leq q(A\Psi_{r}^{B}(f_{d,\varepsilon}))+\sum_{i=1}^{n}q\Bigl(\frac{1}{\varepsilon}\int_{d_{i}-\varepsilon}^{d_{i}}T(r-s)B(x_{i+1}-x_{i})\d s-T(r-d_{i})B(x_{i+1}-x_{i})\Bigr)\\
&\underset{\eqref{eq:c_admiss_max_reg}}{\leq} C+\sum_{i=1}^{n}q\Bigl(\frac{1}{\varepsilon}\int_{d_{i}-\varepsilon}^{d_{i}}T(r-s)B(x_{i+1}-x_{i})\d s-T(r-d_{i})B(x_{i+1}-x_{i})\Bigr).
\end{flalign*}
The estimate
\begin{flalign*}
&\hspace{0.37cm}q\Bigl(\frac{1}{\varepsilon}\int_{d_{i}-\varepsilon}^{d_{i}}T(r-s)B(x_{i+1}-x_{i})\d s-T(r-d_{i})B(x_{i+1}-x_{i})\Bigr)\\
&=\frac{1}{\varepsilon}q\Bigl(\int_{d_{i}-\varepsilon}^{d_{i}}(T(r-s)-T(r-d_{i}))B(x_{i+1}-x_{i})\d s\Bigr)\\
&\leq \sup_{s\in [d_{i}-\varepsilon, d_{i}]}q((T(r-s)-T(r-d_{i}))B(x_{i+1}-x_{i}))
\end{flalign*}
in combination with the strong continuity of the semigroup yields by taking the limit as $\varepsilon\to 0\rlim$ that 
\[
q\Bigl(\sum_{i=1}^{n}(T(r-d_{i})-T(r-d_{i-1}))Bx_{i}\Bigr)\leq C.
\]
We conclude that $SV_{q,p}(T(\cdot)B)=SV_{q,p}(T(r-\cdot)B)\leq C$, meaning that $(T(t)B)_{t\in [0,r]}$ is of bounded semivariation.
\end{proof}

\prettyref{thm:Cmax_reg_equiv_Cadm} has the following corollary, which generalises 
\cite[Proposition 3.1, p.~428]{travis1981} where $X=U$ is a Banach space and $B=\id$.

\begin{cor}\label{cor:strict_sol_equiv_bounded_semivar}
Let $r>0$, $X$ be a sequentially complete Hausdorff locally convex space, $(T(t))_{t\geq 0}$ a strongly continuous 
locally equicontinuous semigroup on $X$ with generator $A$ and $x\in D(A)$. 
Let $U$ be a Hausdorff locally convex space and $B\in\mathcal{L}(U;X)$. 
Consider the following assertions.
\begin{enumerate}
\item[(a)] The ACP \eqref{eq:ACP_B} has a strict solution for all $f\in\mathrm{C}([0,r];U)$.
\item[(b)] $(T\ast Bf)(t)\in D(A)$ for all $t\in [0,r]$ and $A(T\ast Bf)$ is right-continuous on $[0,r]$
for all $f\in \mathrm{C}([0,r];U)$. 
\item[(c)] $(T\ast Bf)(r)\in D(A)$ for all $f\in\mathrm{C}([0,r];U)$.
\item[(d)] $(T(t)B)_{t\in [0,r]}$ is of bounded semivariation.
\end{enumerate}
Then we have $(a)\Rightarrow (b)$ and $(d)\Rightarrow(b)\Rightarrow (c)$. 
If $X$ is a $3\mathrm{C}_{U,A,r}^{B}$-space, then $(c)\Rightarrow(d)$. 
If 
\begin{enumerate}
\item[(i)] $U$ is an invariant subspace of $X$ and $B$ commuting with $(T(t))_{t\geq 0}$, or
\item[(ii)] $X$ is a generalised Schwartz space,
\end{enumerate} 
then $(d)\Rightarrow (a)$.
\end{cor}
\begin{proof}
The implications (a)$\Rightarrow$(b) and (d)$\Rightarrow$(b)$\Rightarrow$(c) follow from 
\prettyref{cor:cmax_reg_strict_sol} and \prettyref{thm:Cmax_reg_equiv_Cadm}. 
The same is true for the implication (d)$\Rightarrow$(a) if condition (i) or (ii) is fulfilled.
The implication (c)$\Rightarrow$(d) follows from \prettyref{thm:Cmax_reg_equiv_Cadm} if $X$ is a 
$3\mathrm{C}_{U,A,r}^{B}$-space.
\end{proof}

\begin{exa}\label{ex:sg_bounded_semivar}
(a) Let $X$ be a sequentially complete Hausdorff locally convex space which is barrelled or a strong Mackey space, 
and $(T(t))_{t\geq 0}$ a strongly continuous semigroup on $X$ with generator $A\in\mathcal{L}(X)$. 
Let $U$ be a Hausdorff locally convex space and $B\in\mathcal{L}(U;X)$. 
Then $(T(t))_{t\geq 0}$ is locally equicontinuous since $X$ is barrelled or a strong Mackey space, 
and $X$ is a $3\mathrm{C}_{U,A,r}^{B}$-space for every $r>0$ by \prettyref{rem:3CA_cont_gen}.
Hence $(T(t))_{t\in [0,r]}$ is of bounded semivariation for every $r>0$ 
by case (i) of the equivalence (c)$\Leftrightarrow$(d) of \prettyref{thm:Cmax_reg_equiv_Cadm} 
since $D(A)=X$ and $X$ is a $3\mathrm{C}_{A,r}$-space. 
Due to \prettyref{rem:bounded_semivar_for_some_r_all_r} (b) this implies that 
$(T(t)B)_{t\in [0,r]}$ is also of bounded semivariation for every $r>0$. 

(b) Let $c_{0}\coloneqq \{x\in\K^{\N}\;|\;\lim_{n\to\infty}x_{n}=0\}$ be the space of null sequences 
equipped with the supremum norm $\|\cdot\|_{\infty}$ and $(T_{0}(t))_{t\geq 0}$ the multiplication semigroup 
on $c_0$ given by 
\[
T_{0}(t)x\coloneqq(\e^{-nt}x_{n})_{n\in\N},\quad x\in c_{0},\, t\geq 0.
\]
Then $(T_{0}(t))_{t\geq 0}$ is strongly continuous and quasi-equicontinuous w.r.t.~the topology induced by 
$\|\cdot\|_{\infty}$, $(T_{0}(t))_{t\in [0,r]}$ is of bounded semivariation for every $r>0$ and 
$D(A_{0})\neq c_{0}$ holds for its generator $A_{0}$ by \cite[Example, p.~429]{travis1981} 
(see also \cite[p.~48]{eberhardt1992}).

(c) Let $\ell^{\infty}\coloneqq\{x\in\K^{\N}\;|\;\|x\|_{\infty}\coloneqq\sup_{n\in\N}|x_{n}|<\infty\}=\mathrm{C}_{\operatorname{b}}(\N)$, $m\in\R^{\N}$ such that $\sup_{n\in\N}m_{n}<0$ and 
$(T(t))_{t\geq 0}$ be the multiplication semigroup on $\ell^{\infty}$ given by 
\[
T(t)x\coloneqq(\e^{m_{n}t}x_{n})_{n\in\N},\quad x\in\ell^{\infty},\, t\geq 0.
\]
Then $(T(t))_{t\geq 0}$ is strongly continuous and quasi-equicontinuous w.r.t.~$\beta_{0}=\mu(\ell^{\infty},\ell^{1})$ 
by \cite[Theorem 4.1, p.~19]{kruse_schwenninger2022} 
where $\beta_{0}$ is the substrict topology (see the comments above 
\prettyref{cor:weakly_seq_complete_dual_substrict}) and  $\ell^{1}$ denotes the space of absolutely summable sequences 
in $\K^{\N}$ with dual pairing 
\[
\langle y,x \rangle\coloneqq\sum_{n=1}^{\infty}y_{n}x_{n},\quad y\in\ell^{1},\,x\in\ell^{\infty}.
\]
Its generator is the multiplication operator 
$
A\colon D(A)\to \ell^{\infty},\;Ax =(m_{n}x_{n})_{n\in\N},
$
with domain $D(A)=\{x\in\ell^{\infty}\;|\;(m_{n}x_{n})_{n\in\N}\in\ell^{\infty}\}$ by \cite[p.~353--354]{budde2019} (cf.~\cite[Example 4.9, p.~259--260]{kruse_schwenninger2024}). 
Since $\N$ equipped with the metric induced by the absolute value is a discrete space, we obtain that 
$(\ell^{\infty},\beta_{0})$ is $B_{r}$-complete by \cite[4.6 Corollary, p.~367]{collins1968} and thus a $3\mathrm{C}$-space by \prettyref{cor:Cb_substrict_Cspace} (i). Further, $(\ell^{\infty},\beta_{0})$ is 
a generalised Schwartz space by \cite[Theorem 4.1, p.~365]{collins1968}. 
Let $r>0$, $f\in\mathrm{C}([0,r];(\ell^{\infty},\beta_{0}))$ and denote by $(e_{n})_{n\in\N}$ the canonical 
Schauder basis of $\ell^{1}=(\ell^{\infty},\beta_{0})'$.
Then we have $(T\ast f)(r)=\int_{0}^{r}T(r-s)f(s)\d s\in \ell^{\infty}$ by \prettyref{prop:riemann_int} and 
\begin{align*}
  \langle e_{k},(T\ast f)(r)\rangle
&=\int_{0}^{r}\langle e_{k}, T(r-s)f(s)\rangle\d s
 =\int_{0}^{r}\langle e_{k}, (\e^{m_{n}(r-s)}f_{n}(s))_{n\in\N}\rangle\d s\\
&=\int_{0}^{r}\e^{m_{k}(r-s)}f_{k}(s)\d s\eqqcolon \psi_{k},
\end{align*}
implying 
\begin{equation}\label{eq:better_than_l_infty}
|\psi_{k}|\leq \frac{1}{|m_{k}|}(1-\e^{m_{k}r})\sup_{s\in[0,r]}|f_{k}(s)|\leq \frac{1}{|m_{k}|}\sup_{s\in[0,r]}\|f(s)\|_{\infty}
\end{equation}
for all $k\in\N$. We observe that the set $\{f(s)\;|\;s\in[0,r]\}$ is $\beta_{0}$-bounded because 
$f\in\mathrm{C}([0,r];(\ell^{\infty},\beta_{0}))$. Due to \cite[Theorem 4.7, p.~320]{sentilles1972} this yields that 
$\{f(s)\;|\;s\in[0,r]\}$ is $\|\cdot\|_{\infty}$-bounded. Thus we obtain from \eqref{eq:better_than_l_infty} that 
$(m_{n}\psi_{n})_{n\in\N}\in\ell^{\infty}$, which implies that $(T\ast f)(r)\in D(A)$ since $(T\ast f)(r)_{n}=\psi_{n}$ for all $n\in\N$. 
Hence $(T(t))_{t\in [0,r]}$ is of bounded semivariation for every $r>0$ by \prettyref{rem:bounded_semivar_for_some_r_all_r} (a) and case (i) of the equivalence (c)$\Leftrightarrow$(d) 
of \prettyref{thm:Cmax_reg_equiv_Cadm}. Moreover, if $(m_{n})_{n\in\N}\notin\ell^{\infty}$, then $D(A)\neq \ell^{\infty}$ and 
$A\not\in\mathcal{L}((\ell^{\infty},\beta_{0}))$. 
\end{exa}

Using \prettyref{cor:strict_sol_equiv_bounded_semivar} (i) with $B=\id$ and $U=X$, 
an example of a strongly continuous semigroup $(T(t))_{t\geq 0}$ on a Banach space $X$ such that 
$(T(t))_{t\in [0,r]}$ is not of bounded semivariation for some (every) $r>0$ is for instance given in 
\cite[Example 2.4.6, p.~51]{lorenzirhandi2021}. 
Further, \prettyref{ex:sg_bounded_semivar} (a) and (b) are not some random examples. 
Baillon's theorem actually tells us that it is quite rare for a
strongly continuous semigroup on a Banach space to be of (locally) bounded semivariation. 
Namely, if $(T(t))_{t\geq 0}$ is a strongly continuous semigroup 
on a Banach space $X$ with generator $A$ such that $(T(t))_{t\in [0,r]}$ is of bounded semivariation 
for some $r>0$, then $A\in\mathcal{L}(X)$ or $X$ contains an isomorphic copy of $(c_{0},\|\cdot\|_{\infty})$ 
by \cite[Lemma 3.4, p.~429]{travis1981} (cf.~\cite[Th\'eor\`eme 1, p.~757]{baillon1980} and 
\cite[Theorem 0.5 (Baillon's Theorem), p.~49]{eberhardt1992}). 

That being said, let us take a closer look at the semigroup in \prettyref{ex:sg_bounded_semivar} (c). 
The semigroups in \prettyref{ex:sg_bounded_semivar} (b) and (c) are related in the following way, which follows from 
\cite[Theorem 5.6 (K\"uhnemund), p.~340, p.~354]{budde2019} 
and \cite[Theorem 2.8 (g), p.~243--244]{kruse_schwenninger2024} 
(see also \cite[Example 2.3, p.~147--148]{JacoSchwWint2022}). If $m_{n}\coloneqq -n$ for $n\in\N$, 
then we have that $(T_{0}(t))_{t\geq 0}$ is the restriction of $(T(t))_{t\geq 0}$ to $c_{0}$ 
which is the space of $\|\cdot\|_{\infty}$-strong continuity of $(T(t))_{t\geq 0}$ and also $\beta_{0}$-dense 
in $\ell^{\infty}$. Further, $A_{0}$ is the part of $A$ in $c_{0}$ and so $A_{0}=A_{\mid c_{0}}$ 
with $D(A_{0})=\{x\in c_{0}\;|\;(-nx_{n})_{n\in\N}\in c_{0}\}$. 
The space $(\ell^{\infty},\beta_{0})$ does not contain $(c_{0},\|\cdot\|_{\infty})$ as a topological subspace 
since
\[
\overline{c_{0}}^{\|\cdot\|_{\infty}}=c_{0}\text{ and }\overline{c_{0}}^{\beta_{0}}=\ell^{\infty}
\]
where the closures are taken w.r.t.~$\|\cdot\|_{\infty}$ and $\beta_{0}$, respectively. 
Moreover, the space $(\ell^{\infty},\beta_{0})$ is not normable by \cite[Theorem 4.8, p.~321]{sentilles1972} 
(and the comments concerning $\beta_{0}$ after its proof). 
Due to \prettyref{ex:sg_bounded_semivar} (c) we also know that $(T(t))_{t\in [0,r]}$ is of bounded semivariation 
for every $r>0$ and $A\not\in\mathcal{L}((\ell^{\infty},\beta_{0}))$. 
Thus the conclusion of Baillon's theorem does not hold for the semigroup $(T(t))_{t\geq 0}$ on $\ell^{\infty}$. 
However, this is no contradiction as this semigroup is not strongly continuous w.r.t.~$\|\cdot\|_{\infty}$. 
On the other hand, this raises the question whether something similar to Baillon's theorem holds for 
strongly continuous locally equicontinuous semigroups of (locally) bounded semivariation on sequentially complete 
non-normable Hausdorff locally convex spaces. 

\begin{prob}
Let $r>0$, $X$ be a sequentially complete non-normable Hausdorff locally convex space, 
$(T(t))_{t\geq 0}$ a strongly continuous locally equicontinuous semigroup on $X$ with generator $A$. 
Is it true that if $(T(t))_{t\in [0,r]}$ is of bounded semivariation, 
then $A\in\mathcal{L}(X)$ or $X$ contains an isomorphic copy of  $(\ell^{\infty},\beta_{0})$?
\end{prob}

Now, let us turn back to the ACP \eqref{eq:ACP}. 
If we drop the regularity requirement in $t=0$ of the ACP \eqref{eq:ACP}, 
then we will see that the ACP is actually solvable 
for all $x\in X$, not just $x\in D(A)$, if $A$ is the generator of a strongly continuous locally 
equicontinuous semigroup of (locally) bounded semivariation on a sequentially complete space $X$. 

\begin{defn}\label{defn:classical_solution}
Let $r>0$, $X$ be a Hausdorff locally convex space, $A\colon D(A)\subset X\to X$ a linear map, 
$f\in\mathrm{C}([0,r];X)$ and $x\in X$. We consider the abstract Cauchy problem (ACP) 
\begin{equation}\label{eq:ACP_outside_0}
u'(t)=Au(t)+f(t),\quad t\in (0,r],\quad u(0)=x.
\end{equation}
We call $u\in\mathrm{C}^{1}((0,r];X)\cap\mathrm{C}([0,r];X)$ a \emph{classical solution} of the 
ACP \eqref{eq:ACP_outside_0} if $u(t)\in D(A)$ for all $t\in (0,r]$ and $u$ fulfils \eqref{eq:ACP_outside_0}. 
Here, $\mathrm{C}^{1}((0,r];X)$ denotes the space of continuously 
differentiable functions on $(0,r]$ with values in $X$ where differentiability in $r$ 
means right-differentiability. 
\end{defn}

If the ACP \eqref{eq:ACP_outside_0} has a classical solution $u$, then $Au=u'-f\in\mathrm{C}((0,r];X)$. 
The difference between strict and classical solutions is the regularity in $t=0$ and 
to which sets the initial values $x$ necessarily belong to. 
If the ACP \eqref{eq:ACP_outside_0} has a classical solution $u$, then 
$x=u(0)=\lim_{t\to 0\rlim}u(t)\in \overline{D(A)}$ since $u\in\mathrm{C}([0,r];X)$ and 
$u(t)\in D(A)$ for all $t\in (0,r]$ where $\overline{D(A)}$ denotes the closure of $D(A)$ in $X$. 
If $x\in D(A)$, then a strict solution is also a classical solution. 
In the case that $X$ is a Banach space and $A$ a sectorial operator 
the definition of a classical solution is given in \cite[Definition 3.4.1, p.~70]{lorenzirhandi2021}. 

\begin{cor}\label{cor:classical_sol_bounded_semivar}
Let $r>0$, $X$ be a sequentially complete Hausdorff locally convex space and $(T(t))_{t\geq 0}$ a strongly continuous 
locally equicontinuous semigroup on $X$ with generator $A$. 
If $(T(t))_{t\in [0,r]}$ is of bounded semivariation, then the mild solution $u$ of the ACP \eqref{eq:ACP} 
given by \eqref{eq:mild_solution} is a classical solution of the ACP \eqref{eq:ACP_outside_0} for every 
$f\in\mathrm{C}([0,r];X)$ and $x\in X$.
\end{cor}
\begin{proof}
Let $f\in\mathrm{C}([0,r];X)$, $x\in X$ and $u$ denote the mild solution of the ACP \eqref{eq:ACP} 
given by \eqref{eq:mild_solution}. 
Then we have $u(t)= T(t)x-(T\ast f)(t)$ for all $t\in[0,r]$ and $u\in\mathrm{C}([0,r];X)$ 
by \prettyref{rem:mild_solution_cont}. Since $(T(t))_{t\in [0,r]}$ is of bounded semivariation, 
$(T(t))_{t\geq 0}$ satisfies $\mathrm{C}$-maximal regularity for $r$ by \prettyref{thm:Cmax_reg_equiv_Cadm} 
(i) with $U\coloneqq X$ and $B\coloneqq \id$. 
Further, $T(t)x\in D(A)$ for all $t>0$ and the map $(0,r]\to X$, $t\mapsto AT(t)x$, is continuous 
by \prettyref{cor:orbit_in_domain_for_pos_time}. 
Hence $u(t)\in D(A)$ for all $t\in (0,r]$ and the map $(0,r]\ni t\mapsto Au(t)= AT(t)x-A(T\ast f)(t)\in X$ 
is also continuous. 
Looking at the proof of the implication (b)$\Rightarrow$(a) of \prettyref{prop:nec_suff_mild_strict}, we
see that it still works for all $t\in (0,r]$ by noting that for every such $t$ there is some $a>0$ 
with $a<t$ and $Au\in\mathrm{C}([a,r];X)$. Therefore we can still apply \prettyref{rem:riemann_int_A_commute} 
and \prettyref{rem:mean_int_to_zero} and obtain that $u$ is differentiable in $t\in (0,r]$ and 
$u'(t)=Au(t)+f(t)$. Since $Au$ and $f$ are continuous on $(0,r]$, $u'$ is also continuous on $(0,r]$, 
yielding that $u$ is a classical solution.
\end{proof}

In the case that $X$ is a Banach space, it also follows from \prettyref{cor:orbit_in_domain_for_pos_time} 
that $(T(t))_{t\geq 0}$ is analytic (holomorphic) if it is a strongly continuous semigroup on $X$ such that 
$(T(t))_{t\in [0,r]}$ is of bounded semivariation for some $r>0$ (see \cite[Lemma 3.3, p.~429]{travis1981}). 
Is this still true in our more general setting where one might use one of the three equivalent assertions in 
\cite[Theorem 1, p.~154]{dembart1974} to define analyticy of a semigroup on a sequentially complete 
complex Hausdorff locally convex space?

\begin{prob}\label{prob:bounded_sv_sg_analytic}
Let $r>0$, $X$ be a sequentially complete Hausdorff locally convex space over the field $\C$ 
and $(T(t))_{t\geq 0}$ a strongly continuous locally equicontinuous semigroup on $X$. 
Is $(T(t))_{t\geq 0}$ an analytic semigroup if $(T(t))_{t\in [0,r]}$ is of bounded semivariation?
\end{prob}

\section{\texorpdfstring{$\mathrm{C}$}{C}-admissibility}
\label{sect:C_adm}

Let $X$ and $U$ be Hausdorff locally convex spaces and $(T(t))_{t\geq 0}$ a strongly continuous 
locally equicontinuous semigroup on $X$ with generator $A$, $x\in X$ and $f\in\mathrm{C}([0,r];U)$. 
In this section we consider a modified version of the ACP \eqref{eq:ACP_B}, namely
\begin{equation}\label{eq:ACP_adm}
u'(t)=Au(t)+Bf(t),\quad t\in[0,r],\quad u(0)=x,
\end{equation}
where the control operator $B$ need not be 
$X$-valued anymore but has values in larger Hausdorff locally convex space $V$ 
which contains $X$ (in the sense of a topological embedding), so $B\in\mathcal{L}(U;V)$ 
(see \cite[p.~527--528]{weiss1989a}).
Now, we cannot apply our theory from \prettyref{sect:ACP_Cmax_reg} and \prettyref{sect:sg_bounded_semivar} 
anymore since for instance the mild solution $u(t)=T(t)x+(T\ast Bf)(t)$ for $t\geq 0$ is not defined. 
Nevertheless, there is a choice of $V$ that allows us to extend our semigroup to $V$ so that the mild solution is 
well-defined again. This choice of $V$ is the extrapolation space $X_{-1}$ of $X$. 
However, to define $X_{-1}$ and to obtain the properties we would like to have, we have to strengthen our 
assumptions on $X$ and $(T(t))_{t\geq 0}$ a bit.

We recall the following construction of the extrapolation space $X_{-1}$ from \cite{wegner2014}. 
Let $X$ be a complete Hausdorff locally convex space and $\Gamma_{X}$ a fundamental system of seminorms 
of its topology $\tau$. 
Let $(T(t))_{t\geq 0}$ be a strongly continuous quasi-equicontinuous semigroup on $X$ with generator $A$. 
As in \cite[p.~258]{albanese2013} we define the \emph{resolvent set} $\rho(A)$ of $A$ by 
\[
\rho(A)\coloneqq\{\lambda\in\C\;|\;(\lambda-A)\colon D(A)\to X\text{ is bijective and }(\lambda-A)^{-1}
\in\mathcal{L}(X)\}.
\]
Suppose that $0\in\rho(A)$, which can always be realized by rescaling (see \cite[Lemma 1, p.~450]{wegner2014} 
and the comments after its proof). We define the system of seminorms
\[
\widecheck{p}_{-1}(x)\coloneqq p(A^{-1}x),\quad x\in X,
\]
for $p\in\Gamma_{X}$ and we denote the associated Hausdorff locally convex topology on $X$ by 
$\widecheck{\tau}_{-1}$. Then we define $(X_{-1},\tau_{-1})$ to be the completion of $(X,\widecheck{\tau}_{-1})$. 
For every $t\geq 0$ there exists a unique continuous extension 
$T_{-1}(t)\colon (X_{-1},\tau_{-1})\to (X_{-1},\tau_{-1})$ of $T(t)$. The family $(T_{-1}(t))_{t\geq 0}$ is a 
strongly continuous quasi-equicontinuous semigroup on $(X_{-1},\tau_{-1})$ whose generator is denoted by 
$(A_{-1},D(A_{-1}))$. We have $D(A_{-1})=X$ and $A_{-1}\colon (X,\tau)\to (X_{-1},\tau_{-1})$ is a 
topological isomorphism which is the unique continuous extension of $A\colon (D(A),\tau_{1})\to (X,\tau)$  
by \cite[Theorem 1, p.~451]{wegner2014} where $\tau_{1}$ is the Hausdorff locally convex topology on $D(A)$ 
induced by the system of seminorms given by 
\[
p_{1}(x)\coloneqq p(A x),\quad x\in D(A),
\]
for $p\in\Gamma_{X}$. Furthermore, it holds that $T_{-1}(t)=A_{-1}T(t)A_{-1}^{-1}$ for every $t\geq 0$ by 
\cite[Corollary 1 (ii), p.~454]{wegner2014}. Given $x\in X$, $f\in\mathrm{C}([0,r];U)$ and 
$B\in\mathcal{L}(U;X_{-1})$, the extension of the ACP \eqref{eq:ACP_adm} is now the ACP 
\[
u'(t)=A_{-1}u(t)+Bf(t),\quad t\in[0,r],\quad u(0)=x.
\]
Its mild solution $u$ given by $u(t)=T_{-1}(t)x+(T_{-1}\ast Bf)(t)$ for $t\geq 0$ is a well-defined 
continuous function on $[0,r]$ with values in $X_{-1}$ by \prettyref{rem:mild_solution_cont}. 
For $t\geq 0$ we note that $T_{-1}(t)x\in X$ since $T_{-1}(t)$ is an extension of $T(t)$ and $x\in X$,  
but $(T_{-1}\ast Bf)(t)\in X_{-1}$ in general since it is defined by a Riemann integral in $(X_{-1},\tau_{-1})$. 
However, it might happen that $(T_{-1}\ast Bf)(t)\in X$ for all $t\geq 0$ and that results in the following definition.

\begin{defn}\label{defn:C_adm}
Let $r>0$, $X$ be a complete Hausdorff locally convex space and $(T(t))_{t\geq 0}$ a strongly continuous 
quasi-equicontinuous semigroup on $X$ with generator $A$ such that $0\in\rho(A)$.
Let $U$ be a Hausdorff locally convex space. 
An operator $B\in \mathcal{L}(U;X_{-1})$ is called \emph{$\mathrm{C}$-admissible} for $r$
if the map 
\[
 \Phi_{r}^{B}\colon \mathrm{C}([0,r];U)\to X_{-1},\; \Phi_{r}^{B}f\coloneqq (T_{-1}\ast Bf)(r)\coloneqq \int_{0}^{r}T_{-1}(r-s)Bf(s)\d s,
\]
has range in $X$ where the integral is a Riemann integral in $(X_{-1},\tau_{-1})$. 
\end{defn}

For Banach spaces $X$ and $U$ the definition above is already presented in 
\cite[Definition 1.2 (1), p.~144--145]{JacoSchwWint2022}. 
It follows from \prettyref{prop:adm_cmax_for_some_r_for_all} (b) that 
$B$ is $\mathrm{C}$-admissible for all $r>0$ if $B$ is $\mathrm{C}$-admissible for some $r>0$.

\begin{prop}\label{prop:equiv_C_admiss}
Let $r>0$, $X$ be a complete Hausdorff locally convex space and $(T(t))_{t\geq 0}$ a strongly continuous 
quasi-equicontinuous semigroup on $X$ with generator $A$ such that $0\in\rho(A)$.  
Let $U$ be a Hausdorff locally convex space and $B\in\mathcal{L}(U;X_{-1})$. 
Then the following assertions are equivalent.
\begin{enumerate}
\item[(a)] $(T\ast A_{-1}^{-1}Bf)(t)\in D(A)$ for all $f\in\mathrm{C}([0,r];U)$ and $t\in[0,r]$.
\item[(b)] $B$ is $\mathrm{C}$-admissible for $r$.
\end{enumerate}
If one of the two equivalent conditions above is fulfilled, then
\[
 (T_{-1}\ast Bf)(t)
=A(T\ast A_{-1}^{-1}Bf)(t)
\]
for all $f\in\mathrm{C}([0,r];U)$ and $t\in[0,r]$.
\end{prop}
\begin{proof}
$(T_{-1}\ast Bf)(t)$ and $(T\ast A_{-1}^{-1}Bf)(t)$ are well-defined in $X_{-1}$ and $X$, respectively, 
for every $t\in [0,r]$ and $f\in\mathrm{C}([0,r];U)$ by \prettyref{prop:riemann_int} since $(T_{-1}(t))_{t\geq 0}$ is a 
strongly continuous quasi-equicontinuous semigroup on the complete space $(X_{-1},\tau_{-1})$, 
$B\in\mathcal{L}(U;X_{-1})$ and $A_{-1}\in\mathcal{L}(X;X_{-1})$ is a topological isomorphism.

Using that $T_{-1}(t)=A_{-1}T(t)A_{-1}^{-1}$ for every $t\geq 0$ and the continuity of $A_{-1}$, we get 
for every $f\in\mathrm{C}([0,r];U)$ and $t\in [0,r]$ that
\begin{align*}
 (T_{-1}\ast Bf)(t)
&=\int_{0}^{t}T_{-1}(t-s)Bf(s)\d s
 =\int_{0}^{t}A_{-1}T(t-s)A_{-1}^{-1}Bf(s)\d s\\
&=A_{-1}\int_{0}^{t}T(t-s)A_{-1}^{-1}Bf(s)\d s
 =A_{-1}(T\ast A_{-1}^{-1}Bf)(t).
\end{align*}
$A_{-1}$ is the unique continuous extension of the bijective map 
$A\colon (D(A),\tau_{1})\to (X,\tau)$. 
Hence, if $(T\ast A_{-1}^{-1}Bf)(t)\in D(A)$, then 
\[
 (T_{-1}\ast Bf)(t)
=A_{-1}(T\ast A_{-1}^{-1}Bf)(t)
=A(T\ast A_{-1}^{-1}Bf)(t)\in X,
\]
which gives that (a) implies (b).

On the other hand, if $(T_{-1}\ast Bf)(r)\in X$, then $(T_{-1}\ast Bf)(t)\in X$ for all $f\in\mathrm{C}([0,r];U)$ 
and $t\in [0,r]$ by \prettyref{prop:adm_cmax_for_some_r_for_all} (a) applied to the 
strongly continuous quasi-equicontinuous semigroup $(T_{-1}(t))_{t\geq 0}$  
on the complete space $(X_{-1},\tau_{-1})$ with $X=D(A_{-1})$. It follows that
\[
(T\ast A_{-1}^{-1}Bf)(t)=A^{-1}A_{-1}(T\ast A_{-1}^{-1}Bf)(t)=A^{-1}(T_{-1}\ast Bf)(t)\in D(A).
\]
Thus (b) implies (a).
\end{proof}

Our next result shows that the $\mathrm{C}$-admissibility of $A_{-1}$ already implies the 
$\mathrm{C}$-admissibility of any other control operator $B$. This is the analogon 
of \prettyref{prop:adm_cmax_for_some_r_for_all} (d) for $\mathrm{C}$-maximal regularity and 
\prettyref{rem:bounded_semivar_for_some_r_all_r} (b) for bounded semivariation.

\begin{cor}\label{cor:equiv_C_admiss}
Let $r>0$, $X$ be a complete Hausdorff locally convex space and $(T(t))_{t\geq 0}$ a strongly continuous 
quasi-equicontinuous semigroup on $X$ with generator $A$ such that $0\in\rho(A)$. 
Then the following assertions are equivalent.
\begin{enumerate}
\item[(a)] $(T\ast f)(t)\in D(A)$ for all $f\in\mathrm{C}([0,r];X)$ and $t\in[0,r]$.
\item[(b)] $A_{-1}$ is $\mathrm{C}$-admissible for $r$.
\item[(c)] For any Hausdorff locally convex space $U$ every operator $B\in\mathcal{L}(U;X_{-1})$ 
is $\mathrm{C}$-admissible for $r$.
\end{enumerate} 
\end{cor}
\begin{proof}
(c)$\Rightarrow$(b) This implication is obvious since $A_{-1}\in\mathcal{L}(X;X_{-1})$.

(b)$\Rightarrow$(c) The map $f\mapsto A_{-1}f$ is a topological isomorphism from $\mathrm{C}([0,r];X)$ to 
$\mathrm{C}([0,r];X_{-1})$ 
since $A_{-1}\colon (X,\tau)\to (X_{-1},\tau_{-1})$ is a topological isomorphism where 
$\tau$ denotes the topology of $X$. Let $U$ be a Hausdorff locally convex space and 
suppose that $B\in\mathcal{L}(U;X_{-1})$ is $\mathrm{C}$-admissible for $r$. We note that for any 
$u\in\mathrm{C}([0,r];U)$ it holds that $Bu=A_{-1}\widetilde{u}$ with 
$\widetilde{u}\coloneqq A_{-1}^{-1}Bu\in\mathrm{C}([0,r];X)$. Thus we have 
$\Phi_{r}^{B}u=(T_{-1}\ast Bu)(r)=(T_{-1}\ast A_{-1}\widetilde{u})(r)\in X$ since 
$A_{-1}$ is $\mathrm{C}$-admissible for $r$, proving that $B$ is $\mathrm{C}$-admissible for $r$.

(a)$\Leftrightarrow$(b) This equivalence follows from \prettyref{prop:equiv_C_admiss} with $B\coloneqq A_{-1}$.
\end{proof}

The proof of the equivalence (b)$\Leftrightarrow$(c) in \prettyref{cor:equiv_C_admiss} is just an adaptation 
of the proof of \cite[Proposition 1.4, p.~146]{JacoSchwWint2022}. 
By combining \prettyref{thm:Cmax_reg_equiv_Cadm} and \prettyref{prop:equiv_C_admiss} we obtain 
the following description of $\mathrm{C}$-admissibility of an operator $B\in\mathcal{L}(U;X_{-1})$, 
which generalises \cite[Proposition 2.1 (i)$\Leftrightarrow$(ii), p.~4--5]{arora2025} where $X$ and $U$ 
are Banach spaces.

\begin{cor}\label{cor:adm_max_reg_bounded_semivar}
Let $r>0$, $X$ be a complete Hausdorff locally convex space and $(T(t))_{t\geq 0}$ a strongly continuous 
quasi-equicontinuous semigroup on $X$ with generator $A$ such that $0\in\rho(A)$.  
Let $U$ be a Hausdorff locally convex space and $B\in\mathcal{L}(U;X_{-1})$. 
Consider the following assertions.
\begin{enumerate}
\item[(a)] $(T(t))_{t\geq 0}$ satisfies $\mathrm{C}$-maximal regularity for $(A_{-1}^{-1}B,r)$.
\item[(b)] $(T\ast A_{-1}^{-1}Bf)(t)\in D(A)$ for all $t\in [0,r]$ and $A(T\ast A_{-1}^{-1}Bf)$ 
is right-continuous on $[0,r]$ for all $f\in \mathrm{C}([0,r];U)$. 
\item[(c)] $B$ is $\mathrm{C}$-admissible for $r$.
\item[(d)] $(T(t)A_{-1}^{-1}B)_{t\in [0,r]}$ is of bounded semivariation.
\end{enumerate}
Then we have $(a)\Rightarrow (b)$ and $(d)\Rightarrow(b)\Rightarrow (c)$. 
If $X$ is a $3\mathrm{C}_{U,A,r}^{A_{-1}^{-1}B}$-space, then $(c)\Rightarrow(d)$. 
If 
\begin{enumerate}
\item[(i)] $U$ is an invariant subspace of $X$ and $A_{-1}^{-1}B$ commuting with $(T(t))_{t\geq 0}$, or
\item[(ii)] $X$ is a generalised Schwartz space,
\end{enumerate} 
then $(d)\Rightarrow (a)$.
\end{cor}

We note that a complete Hausdorff locally convex space is a generalised Schwartz space if and only if it is a 
semi-Montel space by \cite[3.5.1 Theorem, p.~64]{jarchow1981}. 
As a consequence of the preceding corollary we get by the choice $B\coloneqq A_{-1}$ 
that $\mathrm{C}$-maximal regularity of a semigroup 
with generator $A$ and $\mathrm{C}$-admissibility of $A_{-1}$ are equivalent 
for strongly continuous quasi-equicontinuous semigroups on complete $3\mathrm{C}_{A,r}$-spaces $X$. 
In the case that $X$ is a Banach space this has already been observed in 
\cite[Proposition 2.2, p.~146]{JacoSchwWint2022}.

\begin{cor}\label{cor:adm_A_-1_max_reg_bounded_semivar}
Let $r>0$, $X$ be a complete Hausdorff locally convex space and $(T(t))_{t\geq 0}$ a strongly continuous 
quasi-equicontinuous semigroup on $X$ with generator $A$ such that $0\in\rho(A)$. 
Consider the following assertions.
\begin{enumerate}
\item[(a)] $(T(t))_{t\geq 0}$ satisfies $\mathrm{C}$-maximal regularity for $r$.
\item[(b)] $A_{-1}$ is $\mathrm{C}$-admissible for $r$.
\item[(c)] $(T(t))_{t\in [0,r]}$ is of bounded semivariation.
\end{enumerate}
Then we have $(c)\Rightarrow(a)\Rightarrow (b)$. 
If $X$ is a $3\mathrm{C}_{A,r}$-space, then $(b)\Rightarrow(c)$.
\end{cor}

\begin{exa}
(a) Let $(T(t))_{t\geq 0}$ be a strongly continuous quasi-equicontinuous semigroup on a complete Hausdorff 
locally convex space $X$ with generator $A\in\mathcal{L}(X)$ such that $0\in\rho(A)$. 
Then $X_{-1}=X$, $T_{-1}(t)=T(t)$ for $t\geq 0$ and $A_{-1}=A$ 
is $\mathrm{C}$-admissible for every $r>0$ by \prettyref{ex:sg_bounded_semivar} (a) 
and \prettyref{cor:adm_A_-1_max_reg_bounded_semivar}.

(b) Let $(T_{0}(t))_{t\geq 0}$ be the strongly continuous multiplication semigroup 
on the Banach space $(c_0,\|\cdot\|_{\infty})$ from \prettyref{ex:sg_bounded_semivar} (b) 
with generator $(A_{0},D(A_{0}))$ and $\K=\C$.
Then $0\in\rho(A_{0})$, 
$
(c_{0})_{-1}=\{x\in\C^{\N}\;|\;(\tfrac{1}{n}x_{n})_{n\in\N}\in c_{0}\} 
$
and $(T_{0})_{-1}(t)x=(\e^{-nt}x_{n})_{n\in\N}$ for $x\in(c_{0})_{-1}$ and $t\geq 0$ 
as well as $(A_{0})_{-1}x=(-nx_{n})_{n\in\N}$ for $x\in c_{0}$ by \cite[Example 1, p.~456]{wegner2014}. 
Due to \prettyref{ex:sg_bounded_semivar} (b) and \prettyref{cor:adm_A_-1_max_reg_bounded_semivar} 
$(A_{0})_{-1}$ is $\mathrm{C}$-admissible for every $r>0$ (cf.~\cite[Example 2.3, p.~147--148]{JacoSchwWint2022}).

(c) Let $(T(t))_{t\geq 0}$ be the strongly continuous quasi-equicontinuous multiplication semigroup on 
the complete $3\mathrm{C}$-space $(\ell^{\infty},\beta_{0})$ from \prettyref{ex:sg_bounded_semivar} (c) 
with generator $A$ and $\K=\C$. We note that $\rho(A)=\C\setminus \overline{\{m_{n}\;|\;n\in\N\}}$ by 
\cite[Example 3.14]{kruse_seifert2022b}.
Since $\sup_{n\in\N}m_{n}<0$, we have $0\in\rho(A)$. 
Then similarly to \cite[Example 1, p.~456]{wegner2014} 
(or using \cite[Theorem 2.15, p.~330]{budde2019} and the fact that the semigroup is bi-continuous) we have
$
\ell^{\infty}_{-1}=\{x\in\C^{\N}\;|\;(\tfrac{1}{m_{n}}x_{n})_{n\in\N}\in\ell^{\infty}\} 
$
and $T_{-1}(t)x=(\e^{m_{n}t}x_{n})_{n\in\N}$ for $x\in\ell^{\infty}_{-1}$ and $t\geq 0$ 
as well as $A_{-1}x=(m_{n}x_{n})_{n\in\N}$ for $x\in \ell^{\infty}$. 
Due to \prettyref{ex:sg_bounded_semivar} (c) and \prettyref{cor:adm_A_-1_max_reg_bounded_semivar} $A_{-1}$ is 
$\mathrm{C}$-admissible for every $r>0$.
\end{exa}

\appendix
\section{\texorpdfstring{$3\mathrm{C}_{U}$}{3CU}-spaces and \texorpdfstring{$3\mathrm{C}$}{3C}-spaces}
\label{app:C_spaces}

Let $\Omega$ be a completely regular Hausdorff space (see \cite[Definition 11.1, p.~180]{james1999}) and 
$X$ and $U$ Hausdorff locally convex spaces. We denote by $\tau_{\operatorname{co}}$ the \emph{compact-open topology} on 
the space $\mathrm{C}(\Omega;U)$, i.e.~the topology of uniform convergence on compact subsets of $\Omega$. 
In this section we are interested in the following question. 
For which combination of spaces $\Omega$, $X$ and $U$ is every closed linear operator 
$C\colon (\mathrm{C}(\Omega;U),\tau_{\operatorname{co}})\to X$ already
continuous? In the case $\Omega=[0,r]$ for some $r>0$ the admissible spaces $X$ for an affirmative answer 
to this question are examples of $3\mathrm{C}_{U}$-spaces.

In order to answer this question we need to recall some notions from general topology and the theory of 
locally convex spaces. We restrict our recall to the lesser known notions. 
For the notions we do not explain we refer the reader to \cite{engelking1989} in the 
case of general topology and again to \cite{jarchow1981,kaballo2014,meisevogt1997,bonet1987} in the case of 
locally convex spaces. 

A completely regular space $\Omega$ is called a \emph{$k_{\R}$-space} if any map $f\colon\Omega\to\R$ 
whose restriction to each compact $K\subset\Omega$ is continuous, is already continuous on $\Omega$ 
(see \cite[p.~487]{michael1973}). Examples of Hausdorff $k_{\R}$-spaces are locally compact Hausdorff spaces by 
\cite[3.3.1 Theorem, p.~148]{engelking1989} and \cite[p.~152]{engelking1989} as well as metrisable spaces 
by \cite[Proposition 11.5, p.~181]{james1999} and \cite[3.3.20, 3.3.21 Theorems, p.~152]{engelking1989}. 
A subset $B$ of a completely regular Hausdorff space $\Omega$ is called \emph{bounding} if $f(B)$ is 
bounded in $\K$ for every $f\in\mathrm{C}(\Omega)$ (see \cite[Definition 10.1.16, p.~373]{bonet1987}). 
A completely regular Hausdorff space $\Omega$ is called a \emph{$\mu$-space} if every bounding subset of $\Omega$ is 
relatively compact (see \cite[Definition 10.1.18, p.~374]{bonet1987}). 
$\Omega$ is called a \emph{$\mu$-$k_{\R}$-space} if $\Omega$ is a $\mu$-space and a $k_{\R}$-space. 
In particular, every realcompact Hausdorff space (see \cite[p.~214]{engelking1989}) is a $\mu$-space 
by \cite[Observation 10.1.19, p.~374]{bonet1987}. 
Thus every (regular) Lindel\"{o}f Hausdorff space (see \cite[p.~192]{engelking1989} where regularity is included 
in the definition) is a realcompact Hausdorff space by \cite[3.11.12 Theorem, p.~216]{engelking1989}. 
Second countable regular spaces are Lindel\"of by \cite[3.8.1 Theorem, p.~192]{engelking1989}. 
Hence separably metrisable spaces are Hausdorff $\mu$-$k_{\R}$-spaces by \cite[4.1.16 Corollary, p.~256]{engelking1989}.
Further, by \cite[3.8.5 Theorem, p.~192--193]{engelking1989} every regular $\sigma$-compact 
space is Lindel\"of. Thus $\sigma$-compact locally compact Hausdorff spaces are Hausdorff $\mu$-$k_{\R}$-spaces, too. 
We also recall again that a Hausdorff locally convex space $X$ is called a \emph{Mackey space} if its topology coincides 
with the Mackey topology $\mu(X,X')$. 

\begin{prop}\label{prop:cont_space_mackey}
Let $\Omega$ be a Hausdorff $\mu$-$k_{\R}$-space and $U$ a complete Mackey space. 
Then $(\mathrm{C}(\Omega;U),\tau_{\operatorname{co}})$ is a complete Mackey space.
\end{prop}
\begin{proof}
Due to \cite[2.4 Theorem, p.~138--139]{bierstedt1973b} the space $(\mathrm{C}(\Omega;U),\tau_{\operatorname{co}})$ is 
topologically isomorphic to the $\varepsilon$-product $(\mathrm{C}(\Omega),\tau_{\operatorname{co}})\varepsilon U
\coloneqq\mathcal{L}_{e}((\mathrm{C}(\Omega),\tau_{\operatorname{co}})_{\kappa}';U)$ 
since $\Omega$ is a $k_{\R}$-space and $U$ complete 
where $(\mathrm{C}(\Omega),\tau_{\operatorname{co}})_{\kappa}'$ denotes the space 
$(\mathrm{C}(\Omega),\tau_{\operatorname{co}})'$ equipped with the topology of uniform convergence on the 
absolutely convex compact subsets of 
$(\mathrm{C}(\Omega),\tau_{\operatorname{co}})$ and 
$\mathcal{L}_{e}((\mathrm{C}(\Omega),\tau_{\operatorname{co}})_{\kappa}';U)$ denotes 
the space of continuous linear operators $\mathcal{L}((\mathrm{C}(\Omega),\tau_{\operatorname{co}})_{\kappa}';U)$ 
equipped with the topology of uniform convergence 
on the equicontinuous subsets of $(\mathrm{C}(\Omega),\tau_{\operatorname{co}})'$. 
The space $(\mathrm{C}(\Omega),\tau_{\operatorname{co}})$ is complete and barrelled by 
\cite[Theorems 10.1.20, 10.1.24 p.~374, 376]{bonet1987} since $\Omega$ is a 
$\mu$-$k_{\R}$-space. In particular, $(\mathrm{C}(\Omega),\tau_{\operatorname{co}})$ is a Mackey space by 
\cite[Chap.~IV, 3.4, p.~132]{schaefer1971}.
Therefore $(\mathrm{C}(\Omega),\tau_{\operatorname{co}})\varepsilon U$ 
is also a complete Mackey space by \cite[Satz 10.3, p.~234]{kaballo2014} and 
\cite[Corollary 6.4.12, p.~117]{kerjean2018} 
since it is the $\varepsilon$-product of two complete Mackey spaces. 
So $(\mathrm{C}(\Omega;U),\tau_{\operatorname{co}})$ is a complete Mackey space as well. 
\end{proof}

To answer our initial question we need some further notions. 
A topological space $\Omega$ is called \emph{hemicompact} if there is a sequence 
$(K_{n})_{n\in\N}$ of compact subsets of $\Omega$ such that for every compact set $K\subset\Omega$ there is $N\in\N$ 
with $K\subset K_{N}$ (see \cite[Exercises 3.4.E, p.~165]{engelking1989}). 
We note that every hemicompact space is $\sigma$-compact by \cite[Exercises 3.8.C (a), p.~194]{engelking1989} and 
for locally compact Hausdorff spaces both notions are equivalent by \cite[Exercises 3.8.C (b), p.~195]{engelking1989}.

Let $X$ and $Y$ be linear spaces such that $\langle X,Y\rangle$ is a dual pairing. 
We denote by $\sigma(X,Y)$ the \emph{weak topology} on $X$ w.r.t.~to $\langle X,Y\rangle$, 
i.e.~the topology of uniform convergence on finite subsets of $Y$, and by $\beta(X,Y)$ the 
\emph{strong topology} on $X$ w.r.t.~to $\langle X,Y\rangle$, i.e.~the topology of uniform convergence on 
$\sigma(Y,X)$-bounded subsets of $Y$ (see e.g.~\cite[p.~171, 174]{kaballo2014}). 
If $X$ is a Hausdorff locally convex space and $Y=X'$, then we set $X_{b}'\coloneqq (X',\beta(X',X))$.

A Hausdorff locally convex space $X$ is called a \emph{$B_{r}$-complete} space if every $\sigma(X',X)$-dense 
$\sigma^{f}(X',X)$-closed linear subspace of $X'$ equals $X'$ where $\sigma^{f}(X',X)$ is the finest topology 
coinciding with $\sigma(X',X)$ on all equicontinuous sets in $F'$ (see \cite[\S34, p.~26]{koethe1979}). 
A Hausdorff locally convex space $X$ is called \emph{$B$-complete} if every $\sigma^{f}(X',X)$-closed linear subspace 
of $X'$ is weakly closed. 
In particular, $B$-complete spaces are clearly $B_{r}$-complete, and $B_{r}$-complete spaces are complete 
by \cite[\S34, 2.(1), p.~26]{koethe1979}.
The definitions of $B_{r}$- and $B$-complete spaces above are equivalent to the original definitions 
by Pt\'{a}k \cite[Definitions 2, 5, p.~50, 55]{ptak1958} due to \cite[\S34, 2.(2), p.~26--27]{koethe1979} 
and we note that such spaces are also called \emph{infra-Pt\'{a}k spaces} and \emph{Pt\'{a}k spaces} (\emph{fully complete spaces}), respectively.

Let $X$ be a Hausdorff locally convex space and $\theta$ be the set $\theta_{bs}$ of bounded sequences in 
$(X^{\ast},\sigma(X^{\ast},X))$ or the set $\theta_{c_{0}}$ of null sequences in $(X^{\ast},\sigma(X^{\ast},X))$ where 
$X^{\ast}$ denotes the algebraic dual of $X$. For linear subspaces $S$ and $M$ of $X^{\ast}$ we set 
\begin{align*}
\dhat{S}^{\theta}&\coloneqq\{y\in X^{\ast}\;|\;\exists\,(y_{n})_{n\in\N}\text{ in }S:\;(y_{n})_{n\in\N}\in\theta,
\;\sigma(X^{\ast},X)\text{-}\lim_{n\to\infty}y_{n}=y\}\\
\uhat{M}^{\theta}&\coloneqq\bigcap\{S\;|\;M\subset S,\;S\text{ linear subspace of }X^{\ast},\;S=\dhat{S}^{\theta}\}.
\end{align*}
$X$ is called a \emph{$\theta_{r}$-space} if ${\uhat{M\cap X'}^{\phantom{\cdot}}}^{\hspace{-0.05cm}\theta}=X'$ for each linear subspace $M$ of 
$X^{\ast}$ such that $M\cap X'$ is $\sigma(X',X)$-dense in $X'$ (see \cite[p.~505, 507]{boos1997}). 
If $\theta=\theta_{bs}$, then a $\theta_{r}$-space is also called an 
\emph{$L_{r}$-space} (see \cite[p.~17]{boos1993}, \cite[p.~508]{boos1997} and \cite[Definition 2, p.~390]{qiu1985}).

A Hausdorff locally convex space $X$ is called \emph{(quasi-)$c_{0}$-barrelled} if 
any $\sigma(X',X)$-null sequence ($\beta(X',X)$-null sequence) in $X'$ is equicontinuous 
(see \cite[p.~249]{jarchow1981}). $c_{0}$-barrelled spaces are also called \emph{sequentially barrelled} 
(see \cite[Definition, p.~353]{webb1968}). 
$X$ is called a \emph{gDF-space} if it is quasi-$c_{0}$-barrelled and has a countable basis of bounded sets 
(see \cite[p.~257]{jarchow1981}). $X'$ is called \emph{weakly sequentially complete} if $(X',\sigma(X',X))$ 
is sequentially complete.

We also recall the definition of the mixed topology on a Saks space (see \cite[Section 2.1]{wiweger1961}, 
\cite[I.3.2 Definition, p.~27--28]{cooper1978} and \cite[Definition 2.2, p.~3]{kruse_schwenninger2022}).
Let $(X,\|\cdot\|)$ be a normed space and $\tau$ a Hausdorff locally convex topology on $X$ that is coarser 
than the $\|\cdot\|$-topology $\tau_{\|\cdot\|}$. Then the \emph{mixed topology} 
$\gamma \coloneqq \gamma(\|\cdot\|,\tau)$ is defined as the finest linear topology on $X$ that coincides with 
$\tau$ on $\|\cdot\|$-bounded sets and such that $\tau\leq \gamma \leq \tau_{\|\cdot\|}$. 
The mixed topology $\gamma$ is Hausdorff locally convex and 
our definition is equivalent to the one from the literature \cite[Section 2.1]{wiweger1961} 
due to \cite[Lemmas 2.2.1, 2.2.2, p.~51]{wiweger1961}.
The triple $(X,\|\cdot\|,\tau)$ is called a \emph{Saks space} if there exists fundamental system of seminorms 
$\Gamma_{\tau}$ of $(X,\tau)$ such that
\[
\|x\|=\sup_{p\in\Gamma_{\tau}} p(x), \quad x\in X.
\]
We say that a Saks space $(X,\|\cdot\|,\tau)$ is a \emph{semireflexive (Mackey--)Saks space} if $(X,\gamma)$ 
is a semireflexive (Mackey) space. Examples of semireflexive Mackey--Saks spaces are given in 
\cite[Corollary 5.6, p.~269]{kruse_schwenninger2024}.

\begin{prop}\label{prop:C_spaces}
Let $\Omega$ be a completely regular Hausdorff space and $X$ and $U$ Hausdorff locally convex spaces. 
Then any closed linear operator $C\colon (\mathrm{C}(\Omega;U),\tau_{\operatorname{co}})\to X$ is 
continuous if one of the following conditions is fulfilled.
\begin{enumerate}
\item[(i)] $(\mathrm{C}(\Omega;U),\tau_{\operatorname{co}})$ is ultrabornological and $X$ webbed. 
In particular, the first condition is fulfilled if $\Omega$ is a hemicompact $k_{\R}$-space and $U$ a Fr\'echet space.
\item[(ii)] $\Omega$ contains an infinite compact subset, $U$ is a barrelled gDF-space and $X$ $B_{r}$-complete. 
\item[(iii)] $\Omega$ is a $\mu$-$k_{\R}$-space such that the countable union of compact subsets of $\Omega$ 
is relatively compact, $U$ a complete Mackey gDF space and $X$ a semireflexive gDF-space. 
\item[(iv)] $\Omega$ is a $\mu$-$k_{\R}$-space, $(\mathrm{C}(\Omega;U),\tau_{\operatorname{co}})'$ 
weakly sequentially complete, $U$ a complete Mackey space and $X$ an $L_{r}$-space. 
\item[(v)] $\Omega$ is a $\mu$-$k_{\R}$-space, $(\mathrm{C}(\Omega;U),\tau_{\operatorname{co}})$ $c_0$-barrelled, 
$U$ a complete Mackey space and $X$ a $\theta_{r}$-space for $\theta=\theta_{c_{0}}$. 
\end{enumerate}
\end{prop}
\begin{proof}
(i) The first part follows from the \cite[Closed graph theorem 24.31, p.~289]{meisevogt1997} of de Wilde. 

If $\Omega$ is a hemicompact $k_{\R}$-space and $U$ a Fr\'echet space, 
then $(\mathrm{C}(\Omega;U),\tau_{\operatorname{co}})$ is also a Fr\'echet space by \cite[p.~53--54]{hollstein1982}, 
thus ultrabornological by \cite[Remark 24.15 (c), p.~283]{meisevogt1997}.

(ii) Since $U$ is a gDF-space, $U_{b}'$ is a Fr\'echet space by \cite[12.4.2 Theorem, p.~258]{jarchow1981} and thus 
$U_{b}'$ has property $(B)$ of Pietsch by \cite[1.5.8 Theorem, p.~31]{pietsch1972}, which is also called fundamental-$\ell^{1}$-boundedness in \cite[Definition 4.8.2 (ii), p.~139--140]{bonet1987}. 
Therefore $(\mathrm{C}(\Omega;U),\tau_{\operatorname{co}})$ 
is barrelled by \cite[11.10.1 (ii), p.~441]{bonet1987} because $\Omega$ contains an infinite compact subset 
and $U$ is also barrelled. We deduce our statement from the closed graph theorem 
\cite[11.1.7 Theorem (c), p.~221]{jarchow1981} and the $B_{r}$-completeness of $X$. 

(iii) $(\mathrm{C}(\Omega;U),\tau_{\operatorname{co}})$ is a complete Mackey space 
by \prettyref{prop:cont_space_mackey}. 
Further, $(\mathrm{C}(\Omega;U),\tau_{\operatorname{co}})$ is a gDF-space by \cite[Corollary, p.~230]{schmets1987} 
since the countable union of compact subsets of $\Omega$ is relatively compact and $U$ a gDF-space. 
Thus $(\mathrm{C}(\Omega;U),\tau_{\operatorname{co}})_{b}'$ and $X_{b}'$ are Fr\'echet spaces 
by \cite[12.4.2 Theorem, p.~258]{jarchow1981}. We conclude our statement by the closed graph 
theorem \cite[Theorem 1 (vii), p.~398]{mcintosh1969}.

(iv) We know that $(\mathrm{C}(\Omega;U),\tau_{\operatorname{co}})$ is a Mackey space 
by \prettyref{prop:cont_space_mackey}.
The closed graph theorem \cite[Theorem 1, p.~390]{qiu1985} (and its correction \cite[Proposition 3.1, p.~17]{boos1993}) 
yields our statement because $X$ is an $L_{r}$-space. 

(v) Again, by \prettyref{prop:cont_space_mackey} we know that $(\mathrm{C}(\Omega;U),\tau_{\operatorname{co}})$ 
is a Mackey space. Our statement follows from the closed graph theorem \cite[Theorem 3.17, p.~513]{boos1997}. 
\end{proof}

A (non-empty, non-singleton) compact interval $\Omega$ fulfils all the requirements of \prettyref{prop:C_spaces}, so the spaces 
$X$ in \prettyref{prop:C_spaces} are examples of $3\mathrm{C}_{U}$-spaces under the given requirements on $U$. 
Regarding part (i) of \prettyref{prop:C_spaces}, further sufficient conditions on $U$ for the ultrabornologicity of 
$(\mathrm{C}(\Omega;U),\tau_{\operatorname{co}})$ for compact $\Omega$ are given in \cite{frerick2008}. 
In part (ii) one can replace the 
$B_{r}$-completeness by the more general condition that $X$ is a $\Gamma_{r}$-space 
(see \cite[Definition 7.1.9, Theorem 7.1.2, Proposition 7.2.2 (i), p.~202--204]{bonet1987}, 
in \cite[p.~45]{koethe1979} such a space is called an infra-$(s)$-space). 
Comparing parts (iv) and (v), we note that for the Mackey space 
$Y\coloneqq(\mathrm{C}(\Omega;U),\tau_{\operatorname{co}})$ 
the weak sequential completeness of its dual is equivalent to the property that any $\sigma(Y',Y)$-Cauchy sequence 
in $Y'$ is equicontinuous by \cite[\S34, 11.(6), 11.(8), p.~51--52]{koethe1979}. 
So the weak sequential completeness of the dual $Y'$ implies the $c_{0}$-barrelledness of $Y$. 
Looking at the other requirements in \prettyref{prop:C_spaces}, we note the following facts. 

\begin{rem}\fakephantomsection\label{rem:examples_X_for_closed_graph} 
\begin{enumerate}
\item[(a)] Examples of webbed spaces are Fr\'echet spaces, LF-spaces, strong duals of LF-spaces 
or sequentially complete gDF-spaces by \cite[5.2.2 Proposition, p.~90]{jarchow1981}, 
\cite[5.3.3 Corollary (b), p.~92]{jarchow1981}, \cite[Satz 7.25, p.~165]{kaballo2014} and 
\cite[12.4.6 Proposition, p.~260]{jarchow1981}.
\item[(b)] Examples of barrelled spaces are Fr\'echet spaces and reflexive spaces 
by \cite[Propositions 23.22, 23.23, p.~272]{meisevogt1997}. In particular, strong duals of reflexive spaces 
are barrelled by \cite[11.4.5 Proposition (f), p.~228]{jarchow1981}. 
\item[(c)] Examples of Mackey spaces are barrelled spaces by \cite[Chap.~IV, 3.4, p.~132]{schaefer1971}. 
\item[(d)] If $(X,\|\cdot\|,\tau)$ is a Saks space, then $(X,\gamma)$ is a gDF-space by 
\cite[I.1.27 Remark, p.~19]{cooper1978}. Further examples of gDF-spaces are DF-spaces by \cite[p.~257]{jarchow1981}, 
and examples of complete DF-spaces are strong duals of Fr\'echet spaces by \cite[12.4.5 Theorem, p.~260]{jarchow1981}.\item[(e)] Examples of $B$-complete, thus $B_{r}$-complete, spaces are Fr\'echet spaces and 
semireflexive gDF-spaces by \cite[9.5.2 Krein-\u{S}mulian Theorem, p.~184]{jarchow1981} and 
\cite[12.5.7 Proposition, p.~265]{jarchow1981}. In particular, semireflexive Saks spaces and strong duals of 
reflexive Fr\'echet spaces are $B$-complete. 
\end{enumerate}
\end{rem}

Regarding part (iv) of \prettyref{prop:C_spaces}, 
we also recall the following two properties of a Hausdorff locally convex space $X$. 
The space $X$ is called \emph{transseparable} if for every $0$-neighbourhood $U$ in $X$ there exists a 
countable set $A\subset X$ such that $X=A+U$ (see \cite[Definition 2.5.1, p.~53]{bonet1987}). 
Clearly, $X$ is transseparable if it is separable. 
$X$ is called a \emph{WCG-space} if there exists an absolutely convex $\sigma(X,X')$-compact set $K\subset X$ 
such that the span of $K$ is dense in $X$ (see \cite[Definition, p.~86]{hunter1977}).
$X$ is called a \emph{subWCG-space} if it is topologically isomorphic to a 
linear subspace of a WCG-space (see \cite[Definition, p.~93]{hunter1977}). 
By a remark in \cite[p.~86]{hunter1977} and \cite[Theorem 3.1, p.~93]{hunter1977} $X$ is a subWCG-space if it 
is separable. In addition, $X$ is a subWCG-space by a remark in \cite[p.~93]{hunter1977} and 
\cite[Exemples, p.~13]{buchwalter1973} if it is a Schwartz space. 

\begin{prop}\label{prop:examples_Lr}
Let $X$ be a $B_{r}$-complete Hausdorff locally convex space. If $X$ is trans\-separable or a subWCG-space, then 
it is a complete $L_{r}$-space. 
\end{prop}
\begin{proof}
$B_{r}$-complete spaces are complete by \cite[9.5.1 Proposition (b), p.~183]{jarchow1981}. 
If $X$ is trans\-separable, then it follows from the closed graph theorem \cite[Theorem 1.4, p.~183]{kraaij2016a} 
and \cite[Theorem 3, p.~391]{qiu1985} that it is an $L_{r}$-space. 
If $X$ is a subWCG-space, then our statement follows from \cite[Theorem 3.3, p.~17]{boos1993}.
\end{proof}

Further examples of $L_{r}$-spaces can be found in \cite{qiu1985,qiu1995}. 
For instance, every semireflexive DF-space is an $L_{r}$-space by \cite[Theorem 4, p.~168]{qiu1995}. 
Concerning the weak sequential completeness of $(\mathrm{C}(\Omega;U),\tau_{\operatorname{co}})'$ 
in \prettyref{prop:C_spaces} (iv), we remark the following observation, which uses the notion of a Mazur space. 
We recall from \cite[p.~40]{wilansky1981} that 
a Hausdorff locally convex space $X$ is called a \emph{Mazur space} if 
\[
X'=\{x'\colon X\to\K\;|\;x'\;\text{is linear and sequentially continuous}\}.
\]
In particular, every C-sequential space is a Mazur space by \cite[Theorem 7.4, p.~52]{wilansky1981}.

\begin{prop}\label{prop:weakly_seq_complete_dual}
Let $\Omega$ be a Hausdorff $\mu$-$k_{\R}$-space, $U$ a complete Mackey space and 
$(\mathrm{C}(\Omega;U),\tau_{\operatorname{co}})$ a Mazur space. 
Then $(\mathrm{C}(\Omega;U),\tau_{\operatorname{co}})$ is a Mackey--Mazur space and 
$(\mathrm{C}(\Omega;U),\tau_{\operatorname{co}})'$ weakly sequentially complete.
\end{prop}
\begin{proof}
The space $(\mathrm{C}(\Omega;U),\tau_{\operatorname{co}})$ is a Mackey space by \prettyref{prop:cont_space_mackey}. 
Thus it is a Mackey--Mazur space by assumption. 
It follows from \cite[Propositions 4.3, 4.4, p.~354]{webb1968} that 
$(\mathrm{C}(\Omega;U),\tau_{\operatorname{co}})'$ is weakly sequentially complete.
\end{proof}

We see in \prettyref{prop:cont_space_mackey} that $(\mathrm{C}(\Omega;U),\tau_{\operatorname{co}})$ is a 
Mackey space if $\Omega$ is a $\mu$-$k_{\R}$-space and $U$ a complete Mackey space. 
Is this true for the Mazur property as well?

\begin{prob}
Let $\Omega$ be a Hausdorff $\mu$-$k_{\R}$-space and $U$ a complete Hausdorff locally convex space. 
Is $(\mathrm{C}(\Omega;U),\tau_{\operatorname{co}})$ a Mazur space if $U$ is a Mazur space?
\end{prob}

We prepare an application of \prettyref{prop:C_spaces} (iv). First, we recall again some notions from 
general topology. A topological space $\Omega$ is called \emph{submetrisable} 
if there exist a metric space $Y$ and an injective continuous map $f\colon\Omega\to Y$. If, in addition, 
$Y$ can be chosen to be separable, then $\Omega$ is called \emph{separably submetrisable} 
(see \cite[p.~508]{summers1972}). In particular, \emph{Polish spaces}, i.e.~separably completely metrisable spaces,
are separably submetrisable.

Second, let $\Omega$ be a completely regular Hausdorff space and $\mathcal{V}$ denote the set of all 
non-negative bounded functions $v$ on $\Omega$ that \emph{vanish at infinity}, i.e.~for every $\varepsilon>0$ the 
set $\{x\in\Omega\;|\;v(x)\geq\varepsilon\}$ is compact. 
Let $\beta_{0}$ be the Hausdorff locally convex topology on $\mathrm{C}_{\operatorname{b}}(\Omega)$ that is induced 
by the seminorms 
\[
|f|_{v}\coloneqq\sup_{x\in\Omega}|f(x)|v(x),\quad f\in\mathrm{C}_{\operatorname{b}}(\Omega),
\]
for $v\in\mathcal{V}$. The topology $\beta_{0}$ is called the \emph{substrict topology} 
(see \cite[p.~315--316]{sentilles1972}). 
$(\mathrm{C}_{\operatorname{b}}(\Omega),\|\cdot\|_{\infty},\tau_{\operatorname{co}})$ is a Saks space and 
$\beta_{0}=\gamma(\|\cdot\|_{\infty},\tau_{\operatorname{co}})$ by \cite[Theorem 2.4, p.~316]{sentilles1972}. 
Here, $\|\cdot\|_{\infty}$ denotes the supremum norm on $\mathrm{C}_{\operatorname{b}}(\Omega)$. 
If $\Omega$ is compact, then $\beta_{0}=\tau_{\|\cdot\|_{\infty}}=\tau_{\operatorname{co}}$. 
If $\Omega$ is not compact, then $(\mathrm{C}_{\operatorname{b}}(\Omega),\beta_{0})$ is neither barrelled 
nor bornological, in particular not metrisable, by \cite[Theorem 4.8, p.~321]{sentilles1972} 
(and the comments concerning $\beta_{0}$ after its proof). 
The interest in the space $(\mathrm{C}_{\operatorname{b}}(\Omega),\beta_{0})$ comes for instance from 
transition semigroups that are considered on it (see \cite{goldys2024,kruse_schwenninger2022} and the references therein).

\begin{cor}\label{cor:weakly_seq_complete_dual_substrict}
Let $\Omega$ and $\Omega_{0}$ be both Polish spaces, or both hemicompact Hausdorff $k_{\R}$-spaces. 
If $\Omega$ is compact, 
then $(\mathrm{C}(\Omega;(\mathrm{C}_{\operatorname{b}}(\Omega_{0}),\beta_{0})),\tau_{\operatorname{co}})$  
is a Mackey--Mazur space and 
$(\mathrm{C}(\Omega;(\mathrm{C}_{\operatorname{b}}(\Omega_{0}),\beta_{0})),\tau_{\operatorname{co}})'$ 
weakly sequentially complete. 
\end{cor}
\begin{proof}
The space $(\mathrm{C}_{\operatorname{b}}(\Omega_{0}),\beta_{0})$ is complete by 
\cite[3.6.9 Theorem, p.~72]{jarchow1981}. 
Further, $(\mathrm{C}_{\operatorname{b}}(\Omega_{0}),\beta_{0})$ is a Mackey space by 
\cite[Theorems 5.7, 5.8 (b), 9.1 (a), p.~325, 332]{sentilles1972} if $\Omega_{0}$ is Polish, 
and by \cite[Theorem 5.2, p.~884]{mosiman1972} if $\Omega_{0}$ is a hemicompact Hausdorff $k_{\R}$-space. 

The space $(\mathrm{C}(\Omega;(\mathrm{C}_{\operatorname{b}}(\Omega_{0}),\beta_{0})),\tau_{\operatorname{co}})$ 
is topologically isomorphic to 
the space $(\mathrm{C}_{\operatorname{b}}(\Omega\times\Omega_{0}),\beta_{0})$ 
by \cite[1.2 Theorem, p.~123]{bierstedt1974} 
(and the subsequent comments). The space $\Omega\times\Omega_{0}$ is Polish if $\Omega$ and $\Omega_{0}$ are Polish, 
and $\Omega\times\Omega_{0}$ is a hemicompact Hausdorff $k_{\R}$-space by \cite[Lemme (2.4), p.~55]{buchwalter1972} 
if $\Omega$ and $\Omega_{0}$ are hemicompact Hausdorff $k_{\R}$-spaces. 
Therefore $(\mathrm{C}_{\operatorname{b}}(\Omega\times\Omega_{0}),\beta_{0})$ is C-sequential, thus Mazur,  
by \cite[Remark 3.19 (a), p.~14]{kruse_schwenninger2022}. 
Hence the topologically isomorphic space 
$(\mathrm{C}(\Omega;(\mathrm{C}_{\operatorname{b}}(\Omega_{0}),\beta_{0})),\tau_{\operatorname{co}})$  
is also a Mazur space. We deduce our statement from \prettyref{prop:weakly_seq_complete_dual}. 
\end{proof}

In the case that $\Omega$ is a compact Polish space and $\Omega_{0}$ a Polish space 
the weak sequential completeness of the space $(\mathrm{C}_{\operatorname{b}}(\Omega\times\Omega_{0}),\beta_{0})'$ 
is also a consequence of \cite[Lemma 1.9, p.~183]{kraaij2016a}. 

\begin{cor}\label{cor:Cb_substrict_Cspace}
If 
\begin{enumerate}
\item[(i)] $\Omega$ is a Polish space or a separably submetrisable hemicompact Hausdorff $k_{\R}$-space, or
\item[(ii)] $\Omega$ is a hemicompact Hausdorff $k_{\R}$-space such that every compact subset is metrisable, 
\end{enumerate} 
and $(\mathrm{C}_{\operatorname{b}}(\Omega),\beta_{0})$ is $B_{r}$-complete, 
then $(\mathrm{C}_{\operatorname{b}}(\Omega),\beta_{0})$ is a $3\mathrm{C}$-space.
\end{cor}
\begin{proof}
(i) The space $(\mathrm{C}_{\operatorname{b}}(\Omega),\beta_{0})$ is separable by 
\cite[2.1 Theorem, p.~509]{summers1972} since $\Omega$ is a separably submetrisable completely regular 
Hausdorff space in both cases. 

(ii) The space $(\mathrm{C}_{\operatorname{b}}(\Omega),\beta_{0})$ is transseparable 
by \cite[Theorem 2, p.~683]{khan2008} with $E\coloneqq \K$ and \cite[Remark 2, p.~685]{khan2008} because $\Omega$ 
is a completely regular Hausdorff space such that every compact subset is metrisable.

The space $(\mathrm{C}_{\operatorname{b}}(\Omega),\beta_{0})$ is a complete Mackey space in both cases 
by the proof of \prettyref{cor:weakly_seq_complete_dual_substrict}. 
We deduce from \prettyref{prop:C_spaces} (iv), \prettyref{prop:examples_Lr} 
and \prettyref{cor:weakly_seq_complete_dual_substrict} that 
$X\coloneqq U\coloneqq (\mathrm{C}_{\operatorname{b}}(\Omega),\beta_{0})$ is a $3\mathrm{C}$-space in 
both cases.
\end{proof}

If $\Omega$ is a compact Hausdorff space, then $\beta_{0}=\tau_{\|\cdot\|_{\infty}}$ 
and so $(\mathrm{C}_{\operatorname{b}}(\Omega),\beta_{0})$ 
is a completely normable space, hence $B$-complete and so $B_{r}$-complete by 
\cite[9.5.2 Krein-\u{S}mulian Theorem, p.~184]{jarchow1981}. 
If $\Omega$ is a discrete space, then $(\mathrm{C}_{\operatorname{b}}(\Omega),\beta_{0})$ is $B$-complete, 
thus $B_{r}$-complete, by \cite[4.6 Corollary, p.~367]{collins1968}. 
Unfortunately, apart from these two cases not much seems to be known when 
$(\mathrm{C}_{\operatorname{b}}(\Omega),\beta_{0})$ is 
$B_{r}$-complete or $B$-complete (see \cite[p.~1202]{summers1970} as well as 
\cite[Proposition 1.2, Theorem 1.7 p.~182--183]{kraaij2016a} and its corrigendum \cite{kraaij2019}).

\begin{prob}
Is $(\mathrm{C}_{\operatorname{b}}(\Omega),\beta_{0})$ a $B_{r}$-complete space if $\Omega$ 
is Polish or a hemicompact Hausdorff $k_{\R}$-space?
\end{prob}

Now, let us turn to \prettyref{prop:C_spaces} (v). 
We say that a Hausdorff locally convex space $X$ has the \emph{Banach--Mackey property} 
if every $\sigma(X,X')$-bounded subset of $X$ is already $\beta(X,X')$-bounded. 
Such spaces are also called \emph{Banach--Mackey spaces} in \cite[p.~216]{qiu1991} and named after the 
Banach--Mackey theorem. However, they need not be Mackey spaces (or Banach spaces).

\begin{prop}\label{prop:Saks_quasi_c0_barrelled}
Let $\Omega$ be a completely regular Hausdorff space such that the countable union of compact subsets of 
$\Omega$ is relatively compact and $(U,\|\cdot\|,\tau)$ be a Saks space. Then the following assertions hold. 
\begin{enumerate}
\item[(a)] $(U,\gamma)$ and $(\mathrm{C}(\Omega;(U,\gamma)),\tau_{\operatorname{co}})$ are quasi-$c_{0}$-barrelled. 
\item[(b)] If $(U,\|\cdot\|)$ is complete, then $(U,\gamma)$ is locally complete and has the Banach--Mackey property.
\item[(c)] If $\Omega$ is a $k_{\R}$-space and $(U,\gamma)$ quasi-complete, 
then $(\mathrm{C}(\Omega;(U,\gamma)),\tau_{\operatorname{co}})$ has the Banach--Mackey property. 
\item[(d)] If $\Omega$ is second-countable and locally compact and $(U,\gamma)$ sequentially complete, 
then $(\mathrm{C}(\Omega;(U,\gamma)),\tau_{\operatorname{co}})$ has the Banach--Mackey property. 
\end{enumerate}
\end{prop}
\begin{proof}
(a) $(U,\gamma)$ is a gDF-space by \cite[I.1.27 Remark, p.~19]{cooper1978} and so 
quasi-$c_{0}$-barrelled by \cite[p.~257]{jarchow1981}. Further, the space 
$(\mathrm{C}(\Omega;(U,\gamma)),\tau_{\operatorname{co}})$ is a 
gDF-space by \cite[Corollary, p.~230]{schmets1987} since the countable union of compact subsets of 
$\Omega$ is relatively compact and $(U,\gamma)$ a gDF-space. 
Hence $(\mathrm{C}(\Omega;(U,\gamma)),\tau_{\operatorname{co}})$ is also quasi-$c_{0}$-barrelled.

(b) First, let $B\subset U$ be absolutely convex, $\gamma$-bounded and $\gamma$-closed. To prove that $(U,\gamma)$ is 
locally complete, we need to show that $B$ is a Banach disk by \cite[10.2.1 Proposition, p.~197]{jarchow1981}. 
Since $B$ is $\gamma$-bounded, it is also $\|\cdot\|$-bounded by \cite[I.1.11 Proposition, p.~10]{cooper1978}. Further, 
$B$ is $\|\cdot\|$-closed since it is $\gamma$-closed and $\gamma$ is coarser than the norm topology $\tau_{\|\cdot\|}$. 
This implies that $B$ is a Banach disk by \cite[10.2.1 Proposition, p.~197]{jarchow1981} 
since the Banach space $(U,\|\cdot\|)$ is locally complete.

Second, since $(U,\gamma)$ is locally complete, it is fast complete by \cite[Proposition 5.1.6, p.~152]{bonet1987} and 
\cite[Definition 1 (a), p.~216--217]{qiu1991}. Hence $(U,\gamma)$ has the Banach--Mackey property 
by \cite[Corollary, p.~217]{qiu1991}.

(c) The space $(\mathrm{C}(\Omega;(U,\gamma)),\tau_{\operatorname{co}})$ is quasi-complete by 
\cite[p.~14]{hollstein1982}. Since quasi-complete spaces are 
locally complete, our statement follows as in the second part of the proof of (b). 

(d) By \cite[4.2.12 Example, p.~47]{kruse2023} we have that the spaces 
$(\mathrm{C}(\Omega),\tau_{\operatorname{co}})\varepsilon (U,\gamma)$ and 
$(\mathrm{C}(\Omega;(U,\gamma)),\tau_{\operatorname{co}})$ are topologically isomorphic since the sequentially 
complete space $(U,\gamma)$ has the metric convex compactness property and $\Omega$ is locally compact and 
second-countable. Therefore the space $(\mathrm{C}(\Omega;(U,\gamma)),\tau_{\operatorname{co}})$ 
is sequentially complete by \cite[Satz 10.3, p.~234]{kaballo2014}. Since sequentially complete spaces are 
locally complete, our statement follows as in the second part of the proof of (b). 
\end{proof}

\begin{cor}\label{cor:c0_barrelled}
Let $\Omega$ be a Hausdorff $\mu$-$k_{\R}$-space such that the countable union of compact subsets of 
$\Omega$ is relatively compact and $(U,\|\cdot\|,\tau)$ a Saks space. 
Then the following assertions hold.
\begin{enumerate}
\item[(a)] If $(U,\|\cdot\|)$ is complete and $(U,\gamma)$ a Mackey space, then $(U,\gamma)$ is $c_{0}$-barrelled. 
\item[(b)] If $(U,\gamma)$ is a complete Mackey space, then 
$(\mathrm{C}(\Omega;(U,\gamma)),\tau_{\operatorname{co}})$ is $c_0$-barrelled.
\end{enumerate}
\end{cor}
\begin{proof}
(a) The Mackey space $(U,\gamma)$ is quasi-$c_{0}$-barrelled and has the Banach--Mackey property 
by \prettyref{prop:Saks_quasi_c0_barrelled} (a) and (b). Thus it is $c_{0}$-barrelled 
by \cite[Lemma 3.2, p.~677]{bosch2000}.

(b) $(\mathrm{C}(\Omega;(U,\gamma)),\tau_{\operatorname{co}})$ is a quasi-$c_{0}$-barrelled Mackey space 
with Banach--Mackey property 
by \prettyref{prop:cont_space_mackey}, \prettyref{prop:Saks_quasi_c0_barrelled} (a) and (c). 
We deduce that the space $(\mathrm{C}(\Omega;(U,\gamma)),\tau_{\operatorname{co}})$ is $c_0$-barrelled 
by \cite[Lemma 3.2, p.~677]{bosch2000}.
\end{proof}

\prettyref{cor:c0_barrelled} (a) is interesting in itself since the $c_{0}$-barrelledness of $(U,\gamma)$ is 
a part of a sufficient condition that guarantees the existence of a dual bi-continuous semigroup of a 
bi-continuous semigroup in the sun dual theory for bi-continuous semigroups on sequentially complete Saks spaces 
(see \cite[3.8 Theorem (b), p.~247]{kruse_schwenninger2024}). 
Using that a (non-empty, non-singleton) compact interval $\Omega$ fulfils all the requirements 
of \prettyref{prop:C_spaces}, 
we obtain by \prettyref{rem:examples_X_for_closed_graph}, \prettyref{prop:weakly_seq_complete_dual} 
and \prettyref{cor:c0_barrelled} (b) with $U=X$ the following list of complete $3\mathrm{C}$-spaces.

\begin{cor}\label{cor:list_C_spaces} 
The following Hausdorff locally convex spaces are complete $3\mathrm{C}$-spaces. 
\begin{enumerate}
\item[(i)] Fr\'echet spaces,
\item[(ii)] barrelled $B_{r}$-complete gDF-spaces, in particular strong duals of reflexive Fr\'echet spaces, 
\item[(iii)] semireflexive Mackey gDF-spaces, in particular semireflexive Mackey--Saks spaces, 
\item[(iv)] complete Mackey $L_{r}$-spaces $X$ such that $(\mathrm{C}([0,r_{0}];X),\tau_{\operatorname{co}})$ 
is a Mazur space for some $r_{0}>0$,
\item[(v)] complete Mackey--Saks $\theta_{r}$-spaces for $\theta=\theta_{c_{0}}$. 
\end{enumerate}
\end{cor}

\bibliography{biblio_C_max_reg}

\begin{thebibliography}{87}
\providecommand{\natexlab}[1]{#1}
\providecommand{\url}[1]{\texttt{#1}}
\expandafter\ifx\csname urlstyle\endcsname\relax
  \providecommand{\doi}[1]{doi: #1}\else
  \providecommand{\doi}{doi: \begingroup \urlstyle{rm}\Url}\fi

\bibitem[Agase(1987)]{agase1987}
S.B. Agase.
\newblock Existence and stability of mild solutions of semilinear differential
  equations in locally convex spaces.
\newblock \emph{Yokohama Math. J.}, 35\penalty0 (1{\&}2):\penalty0 33--45,
  1987.

\bibitem[Albanese and Mangino(2004)]{albanese2004}
A.A. Albanese and E.~Mangino.
\newblock Trotter--{K}ato theorems for bi-continuous semigroups and
  applications to {F}eller semigroups.
\newblock \emph{J. Math. Anal. Appl.}, 289\penalty0 (2):\penalty0 477--492,
  2004.
\newblock \doi{10.1016/j.jmaa.2003.08.032}.

\bibitem[Albanese et~al.(2013)Albanese, Bonet, and Ricker]{albanese2013}
A.A. Albanese, J.~Bonet, and W.J. Ricker.
\newblock Montel resolvents and uniformly mean ergodic semigroups of linear
  operators.
\newblock \emph{Quaest. Math.}, 36\penalty0 (2):\penalty0 253--290, 2013.
\newblock \doi{10.2989/16073606.2013.779978}.

\bibitem[Albanese et~al.(2021)Albanese, Lorenzi, Mangino, and Rhandi]{isem25}
A.A. Albanese, L.~Lorenzi, E.M. Mangino, and A.~Rhandi.
\newblock {25th Internet Seminar on Spectral Theory for Operators and
  Semigroups}, 2021.

\bibitem[Arora and Schwenninger(2025)]{arora2025}
S.~Arora and F.L. Schwenninger.
\newblock Admissible operators for sun-dual semigroups, 2025.
\newblock arXiv preprint \url{https://arxiv.org/abs/2408.02150v2}.

\bibitem[Baillon(1980)]{baillon1980}
J.-B. Baillon.
\newblock Caract\`ere born\'e de certains g\'en\'erateurs de semi-groupes
  lin\'eaires dans les espaces de {B}anach.
\newblock \emph{C. R. Acad. Sci. Paris}, 290:\penalty0 757--760, 1980.

\bibitem[Bauer(1958)]{bauer1958}
H.~Bauer.
\newblock Minimalstellen von {F}unktionen und {E}xtremalpunkte.
\newblock \emph{Arch. Math. (Basel)}, 9\penalty0 (4):\penalty0 389--393, 1958.
\newblock \doi{10.1007/BF01898615}.

\bibitem[Bierstedt(1973)]{bierstedt1973b}
K.D. Bierstedt.
\newblock {Gewichtete R\"{a}ume stetiger vektorwertiger Funktionen und das
  injektive Tensorprodukt. II}.
\newblock \emph{J. Reine Angew. Math.}, 260:\penalty0 133--146, 1973.
\newblock \doi{10.1515/crll.1973.260.133}.

\bibitem[Bierstedt(1974)]{bierstedt1974}
K.D. Bierstedt.
\newblock {Injektive Tensorprodukte und Slice-Produkte gewichteter R\"{a}ume
  stetiger Funktionen}.
\newblock \emph{J. Reine Angew. Math.}, 266:\penalty0 121--131, 1974.
\newblock \doi{10.1515/crll.1974.266.121}.

\bibitem[Bogachev(2007)]{bogachev2007}
V.I. Bogachev.
\newblock \emph{Measure theory, Volume I}.
\newblock Springer, Berlin, 2007.
\newblock \doi{10.1007/978-3-540-34514-5}.

\bibitem[Boos and Leiger(1993)]{boos1993}
J.~Boos and T.~Leiger.
\newblock Some new classes in topological sequence spaces related to
  {$L_{r}$}-spaces and an inclusion theorem for {$K(X)$}-spaces.
\newblock \emph{Z. Anal. Anwend.}, 12\penalty0 (1):\penalty0 13--26, 1993.
\newblock \doi{10.4171/ZAA/582}.

\bibitem[Boos and Leiger(1997)]{boos1997}
J.~Boos and T.~Leiger.
\newblock {`Restricted' closed graph theorems}.
\newblock \emph{Z. Anal. Anwend.}, 16\penalty0 (3):\penalty0 503--518, 1997.
\newblock \doi{10.4171/ZAA/775}.

\bibitem[Bosch and Garc\'{i}a(2000)]{bosch2000}
C.~Bosch and A.~Garc\'{i}a.
\newblock Banach-{M}ackey, locally complete spaces, and
  $\ell_{p,q}$-summability.
\newblock \emph{Int. J. Math. Math. Sci.}, 23\penalty0 (10):\penalty0 675--679,
  2000.
\newblock \doi{10.1155/S0161171200002209}.

\bibitem[Buchwalter(1972)]{buchwalter1972}
H.~Buchwalter.
\newblock Produit topologique, produit tensoriel et $c$-repl\'etion.
\newblock In \emph{Colloque d'analyse fonctionnelle (Bordeaux, 1971)}, number
  31-32 in Bull. Soc. Math. France, M\'em., pages 51--71. Soc. Math. France,
  Paris, 1972.
\newblock \doi{10.24033/msmf.65}.

\bibitem[Buchwalter(1973)]{buchwalter1973}
H.~Buchwalter.
\newblock Espaces localement convexes semi-faibles.
\newblock \emph{Publ. D\'ep. Math., Lyon}, 10\penalty0 (2):\penalty0 13--28,
  1973.

\bibitem[Budde and Farkas(2019)]{budde2019}
C.~Budde and B.~Farkas.
\newblock Intermediate and extrapolated spaces for bi-continuous operator
  semigroups.
\newblock \emph{J. Evol. Equ.}, 19\penalty0 (2):\penalty0 321--359, 2019.
\newblock \doi{10.1007/s00028-018-0477-8}.

\bibitem[Choe(1985)]{choe1985}
Y.H. Choe.
\newblock {$C_0$-semigroups on a locally convex space}.
\newblock \emph{J. Math. Anal. Appl.}, 106\penalty0 (2):\penalty0 293--320,
  1985.
\newblock \doi{10.1016/0022-247X(85)90115-5}.

\bibitem[Collins(1968)]{collins1968}
H.S. Collins.
\newblock On the space {$l^{\infty}(S)$}, with the strict topology.
\newblock \emph{Math. Z.}, 106\penalty0 (5):\penalty0 361--373, 1968.
\newblock \doi{10.1007/BF01115085}.

\bibitem[Cooper(1978)]{cooper1978}
J.B. Cooper.
\newblock \emph{Saks spaces and applications to functional analysis}.
\newblock North-Holland Math. Stud. 28. North-Holland, Amsterdam, 1978.

\bibitem[Da~Prato and Grisvard(1979)]{daprato1979}
G.~Da~Prato and P.~Grisvard.
\newblock Equations d'\'evolution abstraites non lin\'eaires de type
  parabolique.
\newblock \emph{Ann. Mat. Pura Appl. (4)}, 120\penalty0 (1):\penalty0 329--396,
  1979.
\newblock \doi{10.1007/BF02411952}.

\bibitem[Danchin and Mucha(2009)]{danchin2009}
R.~Danchin and P.B. Mucha.
\newblock A critical functional framework for the inhomogeneous
  {N}avier--{S}tokes equations in the half-space.
\newblock \emph{J. Funct. Anal.}, 256\penalty0 (3):\penalty0 881--927, 2009.
\newblock \doi{10.1016/j.jfa.2008.11.019}.

\bibitem[Danchin and Mucha(2015)]{danchin2015}
R.~Danchin and P.B. Mucha.
\newblock Critical functional framework and maximal regularity in action on
  systems of incompressible flows.
\newblock \emph{M{\'e}m. Soc. Math. Fr. (N.S.)}, 1:\penalty0 1--151, 2015.
\newblock \doi{10.24033/msmf.451}.

\bibitem[Dembart(1974)]{dembart1974}
B.~Dembart.
\newblock On the theory of semigroups of operators on locally convex spaces.
\newblock \emph{J. Funct. Anal.}, 16\penalty0 (2):\penalty0 123--160, 1974.
\newblock \doi{10.1016/0022-1236(74)90061-5}.

\bibitem[Eberhardt and Greiner(1992)]{eberhardt1992}
B.~Eberhardt and G.~Greiner.
\newblock Baillon's theorem on maximal regularity.
\newblock \emph{Acta Appl. Math.}, 27:\penalty0 47--54, 1992.
\newblock \doi{10.1007/BF00046635}.

\bibitem[Engel and Nagel(2000)]{engel_nagel2000}
K.-J. Engel and R.~Nagel.
\newblock \emph{One-parameter semigroups for linear evolution equations}.
\newblock Grad. Texts in Math. 194. Springer, New York, 2000.
\newblock \doi{10.1007/b97696}.

\bibitem[Engelking(1989)]{engelking1989}
R.~Engelking.
\newblock \emph{General topology}.
\newblock Sigma Series Pure Math. 6. Heldermann, Berlin, 1989.

\bibitem[Ethier and Kurtz(2005)]{ethier2005}
S.N. Ethier and T.G. Kurtz.
\newblock \emph{Markov processes. {Characterization} and convergence.}
\newblock Wiley Ser. Probab. Stat. John Wiley \& Sons, Hoboken, NJ, 2005.
\newblock \doi{10.1002/9780470316658}.

\bibitem[Farkas(2004)]{farkas2004}
B.~Farkas.
\newblock Perturbations of bi-continuous semigroups.
\newblock \emph{Studia Math.}, 161\penalty0 (2):\penalty0 147--161, 2004.
\newblock \doi{10.4064/sm161-2-3}.

\bibitem[Frerick and Wengenroth(2008)]{frerick2008}
L.~Frerick and J.~Wengenroth.
\newblock {(LB)}-spaces of vector-valued continuous functions.
\newblock \emph{Bull. London Math. Soc.}, 40\penalty0 (3):\penalty0 505--515,
  2008.
\newblock \doi{10.1112/blms/bdn033}.

\bibitem[Goldstein(1985)]{goldstein1985}
J.A. Goldstein.
\newblock \emph{Semigroups of linear operators and applications}.
\newblock Oxford Math. Monogr. Oxford Univ. Press/Clarendon Press, New York,
  Oxford, 1985.

\bibitem[Goldys and Kocan(2001)]{goldys2001}
B.~Goldys and M.~Kocan.
\newblock Diffusion semigroups in spaces of continuous functions with mixed
  topology.
\newblock \emph{J. Differential Equations}, 173\penalty0 (1):\penalty0 17--39,
  2001.
\newblock \doi{10.1006/jdeq.2000.3918}.

\bibitem[Goldys et~al.(2024)Goldys, Nendel, and R{\"o}ckner]{goldys2024}
B.~Goldys, M.~Nendel, and M.~R{\"o}ckner.
\newblock Operator semigroups in the mixed topology and the infinitesimal
  description of {M}arkov processes.
\newblock \emph{J. Differential Equations}, 412:\penalty0 23--86, 2024.
\newblock \doi{10.1016/j.jde.2024.08.024}.

\bibitem[Guerre-Delabri{\`e}re(1993)]{guerre1993}
S.~Guerre-Delabri{\`e}re.
\newblock {$L_ 1$}-regularity of the {Cauchy} problem and geometry of {Banach}
  spaces.
\newblock In \emph{S\'eminaire d'initiation \`a l'analyse. 32\`eme ann\'ee:
  1992/1993}, pages expo. no. 11, 8. Universit{\'e} Pierre et Marie Curie,
  Paris, 1993.

\bibitem[Hollstein(1982)]{hollstein1982}
R.~Hollstein.
\newblock Permanence properties of {$C(X,E)$}.
\newblock \emph{Manuscripta Math.}, 38\penalty0 (1):\penalty0 41--58, 1982.
\newblock \doi{10.1007/BF01168385}.

\bibitem[H{\"o}nig(1973)]{hoenig1973}
C.S. H{\"o}nig.
\newblock The {G}reen function of a linear differential equation with a lateral
  condition.
\newblock \emph{Bull. Amer. Math. Soc.}, 79\penalty0 (3):\penalty0 587--593,
  1973.
\newblock \doi{10.1090/S0002-9904-1973-13214-8}.

\bibitem[Hunter and Lloyd(1977)]{hunter1977}
R.J. Hunter and J.~Lloyd.
\newblock Weakly compactly generated locally convex spaces.
\newblock \emph{Math. Proc. Cambridge Philos. Soc.}, 82\penalty0 (1):\penalty0
  85--98, 1977.
\newblock \doi{10.1017/S0305004100053706}.

\bibitem[Hyt{\"o}nen et~al.(2023)Hyt{\"o}nen, van Neerven, Veraar, and
  Weis]{hytonen2023}
T.~Hyt{\"o}nen, J.~van Neerven, M.~Veraar, and L.~Weis.
\newblock \emph{Analysis in {B}anach Spaces: {Volume III}: {H}armonic Analysis
  and Spectral Theory}.
\newblock Ergeb. Math. Grenzgeb. (3) 76. Springer, Cham, 2023.
\newblock \doi{10.1007/978-3-031-46598-7}.

\bibitem[Jacob et~al.(2022)Jacob, Schwenninger, and
  Wintermayr]{JacoSchwWint2022}
B.~Jacob, F.L. Schwenninger, and J.~Wintermayr.
\newblock A refinement of {B}aillon's theorem on maximal regularity.
\newblock \emph{Studia Math.}, 263\penalty0 (2):\penalty0 141--158, 2022.
\newblock \doi{10.4064/sm200731-20-3}.

\bibitem[James(1999)]{james1999}
I.M. James.
\newblock \emph{Topologies and uniformities}.
\newblock Springer Undergr. Math. Ser. Springer, London, 1999.
\newblock \doi{10.1007/978-1-4471-3994-2}.

\bibitem[Jarchow(1981)]{jarchow1981}
H.~Jarchow.
\newblock \emph{Locally convex spaces}.
\newblock Math. Leitf\"{a}den. Teubner, Stuttgart, 1981.
\newblock \doi{10.1007/978-3-322-90559-8}.

\bibitem[Kaballo(2014)]{kaballo2014}
W.~Kaballo.
\newblock \emph{Aufbaukurs Funktionalanalysis und Operatortheorie}.
\newblock Springer, Berlin, 2014.
\newblock \doi{10.1007/978-3-642-37794-5}.

\bibitem[Kalton and Portal(2008)]{kalton2008}
N.J. Kalton and P.~Portal.
\newblock Remarks on {$\ell_1$} and {$\ell_\infty$}-maximal regularity for
  power-bounded operators.
\newblock \emph{J. Aust. Math. Soc.}, 84\penalty0 (3):\penalty0 345--365, 2008.
\newblock \doi{10.1017/S1446788708000712}.

\bibitem[Kerjean(2018)]{kerjean2018}
M.~Kerjean.
\newblock \emph{Reflexive spaces of smooth functions: a logical account for
  linear partial differential equations}.
\newblock PhD thesis, Universit\'e Sorbonne Paris Cit\'e, 2018.

\bibitem[Khan(2008)]{khan2008}
L.A. Khan.
\newblock Trans-separability in the strict and compact-open topologies.
\newblock \emph{Bull. Korean Math. Soc.}, 45\penalty0 (4):\penalty0 681--687,
  2008.
\newblock \doi{10.4134/BKMS.2008.45.4.681}.

\bibitem[Komatsu(1964)]{komatsu1964}
H.~Komatsu.
\newblock Semi-groups of operators in locally convex spaces.
\newblock \emph{J. Math. Soc. Japan}, 16\penalty0 (3):\penalty0 230--262, 1964.
\newblock \doi{10.2969/jmsj/01630230}.

\bibitem[K{\=o}mura(1968)]{komura1968}
T.~K{\=o}mura.
\newblock Semigroups of operators in locally convex spaces.
\newblock \emph{J. Funct. Anal.}, 2\penalty0 (3):\penalty0 258--296, 1968.
\newblock \doi{10.1016/0022-1236(68)90008-6}.

\bibitem[K\"{o}the(1979)]{koethe1979}
G.~K\"{o}the.
\newblock \emph{{Topological vector spaces II}}.
\newblock Grundlehren Math. Wiss. 237. Springer, Berlin, 1979.
\newblock \doi{10.1007/978-1-4684-9409-9}.

\bibitem[Kraaij(2016{\natexlab{a}})]{kraaij2016}
R.~Kraaij.
\newblock Strongly continuous and locally equi-continuous semigroups on locally
  convex spaces.
\newblock \emph{Semigroup Forum}, 92\penalty0 (1):\penalty0 158--185,
  2016{\natexlab{a}}.
\newblock \doi{10.1007/s00233-015-9689-1}.

\bibitem[Kraaij(2016{\natexlab{b}})]{kraaij2016a}
R.~Kraaij.
\newblock A {Banach}--{D}ieudonn\'e theorem for the space of bounded continuous
  functions on a separable metric space with the strict topology.
\newblock \emph{Topology Appl.}, 209:\penalty0 181--188, 2016{\natexlab{b}}.
\newblock \doi{10.1016/j.topol.2016.06.003}.

\bibitem[Kraaij(2019)]{kraaij2019}
R.~Kraaij.
\newblock {Corrigendum to: `A {Banach}--{D}ieudonn\'e theorem for the space of
  bounded continuous functions on a separable metric space with the strict
  topology' [Topol. Appl. (2016) 181–188]}.
\newblock \emph{Topology Appl.}, 252:\penalty0 198--199, 2019.
\newblock \doi{10.1016/j.topol.2018.08.010}.

\bibitem[Kruse(2023)]{kruse2023}
K.~Kruse.
\newblock \emph{On vector-valued functions and the $\varepsilon$-product}.
\newblock Habilitation thesis. Hamburg University of Technology, 2023.
\newblock \doi{10.15480/882.4898}.

\bibitem[Kruse and Schwenninger(2022)]{kruse_schwenninger2022}
K.~Kruse and F.L. Schwenninger.
\newblock On equicontinuity and tightness of bi-continuous semigroups.
\newblock \emph{J. Math. Anal. Appl.}, 509\penalty0 (2):\penalty0 1--27, 2022.
\newblock \doi{10.1016/j.jmaa.2021.125985}.

\bibitem[Kruse and Schwenninger(2024)]{kruse_schwenninger2024}
K.~Kruse and F.L. Schwenninger.
\newblock Sun dual theory for bi-continuous semigroups.
\newblock \emph{Anal. Math.}, 50\penalty0 (1):\penalty0 235--280, 2024.
\newblock \doi{10.1007/s10476-024-00014-z}.

\bibitem[Kruse and Seifert(2023{\natexlab{a}})]{kruse_seifert2022a}
K.~Kruse and C.~Seifert.
\newblock Final state observability and cost-uniform approximate
  null-controllability for bi-continuous semigroups.
\newblock \emph{Semigroup Forum}, 106\penalty0 (2):\penalty0 421--443,
  2023{\natexlab{a}}.
\newblock \doi{10.1007/s00233-023-10346-1}.

\bibitem[Kruse and Seifert(2023{\natexlab{b}})]{kruse_seifert2022b}
K.~Kruse and C.~Seifert.
\newblock A note on the {L}umer--{P}hillips theorem for bi-continuous
  semigroups.
\newblock \emph{Z. Anal. Anwend.}, 41\penalty0 (3/4):\penalty0 417--437,
  2023{\natexlab{b}}.
\newblock \doi{10.4171/ZAA/1709}.

\bibitem[K\"uhnemund(2001)]{kuehnemund2001}
F.~K\"uhnemund.
\newblock \emph{Bi-continuous semigroups on spaces with two topologies: Theory
  and applications}.
\newblock PhD thesis, Eberhard-Karls-Universit\"at T\"ubingen, 2001.

\bibitem[LeCrone and Simonett(2011)]{lecrone2011}
J.~LeCrone and G.~Simonett.
\newblock Continuous maximal regularity and analytic semigroups.
\newblock In W.~Feng, Z.~Feng, M.~Grasselli, A.~Ibragimov, X.~Lu, S.~Siegmund,
  and J.~Voigt, editors, \emph{{Dynamical Systems and Differential Equations,
  Proceedings of the 8th AIMS International Conference (Dresden, Germany)}},
  volume 2011 of \emph{AIMS Proceedings}, pages 963--970, Dresden, 2011. AIMS.
\newblock \doi{10.3934/proc.2011.2011.963}.

\bibitem[Liggett(2010)]{liggett2010}
T.M. Liggett.
\newblock \emph{Continuous time {Markov} processes. {An} introduction.}
\newblock Grad. Stud. Math. 113. AMS, Providence, RI, 2010.

\bibitem[Lorenzi(2001)]{lorenzi2001}
A.~Lorenzi.
\newblock \emph{An introduction to identification problems via functional
  analysis}.
\newblock Inverse Ill-posed Probl. Ser. 26. De Gruyter, Berlin, Boston, 2001.
\newblock \doi{10.1515/9783110940923}.

\bibitem[Lorenzi and Rhandi(2021)]{lorenzirhandi2021}
L.~Lorenzi and A.~Rhandi.
\newblock \emph{Semigroups of bounded operators and second-order elliptic and
  parabolic partial differential equations}.
\newblock Chapman \& Hall/CRC Monogr. Res. Notes Math. CRC Press, Boca Raton,
  FL, 2021.

\bibitem[Lunardi(2013)]{lunardi2013}
A.~Lunardi.
\newblock \emph{Analytic semigroups and optimal regularity in parabolic
  problems}.
\newblock Mod. Birkh{\"a}user Class. Birkh{\"a}user, Basel, reprint of the 1995
  hardback edition, 2013.

\bibitem[McIntosh(1969)]{mcintosh1969}
A.~McIntosh.
\newblock On the closed graph theorem.
\newblock \emph{Proc. Amer. Math. Soc.}, 20\penalty0 (2):\penalty0 397--404,
  1969.
\newblock \doi{10.1090/S0002-9939-1969-0234246-5}.

\bibitem[Meise and Vogt(1997)]{meisevogt1997}
R.~Meise and D.~Vogt.
\newblock \emph{Introduction to functional analysis}.
\newblock Oxf. Grad. Texts Math. 2. Clarendon Press, Oxford, 1997.

\bibitem[Michael(1973)]{michael1973}
E.A. Michael.
\newblock On $k$-spaces, $k_{R}$-spaces and {$k(X)$}.
\newblock \emph{Pacific J. Math.}, 47\penalty0 (2):\penalty0 487--498, 1973.
\newblock \doi{10.2140/pjm.1973.47.487}.

\bibitem[Mosiman and Wheeler(1972)]{mosiman1972}
S.E. Mosiman and R.F. Wheeler.
\newblock The strict topology in a completely regular setting: Relations to
  topological measure theory.
\newblock \emph{Canad. J. Math.}, 24\penalty0 (5):\penalty0 873--890, 1972.
\newblock \doi{10.4153/CJM-1972-087-2}.

\bibitem[Pazy(1983)]{pazy1983}
A.~Pazy.
\newblock \emph{Semigroups of linear operators and applications to partial
  differential equations}.
\newblock Appl. Math. Sci. 44. Springer, New York, 1983.
\newblock \doi{10.1007/978-1-4612-5561-1}.

\bibitem[P\'{e}rez~Carreras and Bonet(1987)]{bonet1987}
P.~P\'{e}rez~Carreras and J.~Bonet.
\newblock \emph{Barrelled locally convex spaces}.
\newblock Math. Stud. 131. North-Holland, Amsterdam, 1987.

\bibitem[Pietsch(1972)]{pietsch1972}
A.~Pietsch.
\newblock \emph{Nuclear locally convex spaces}.
\newblock Ergeb. Math. Grenzgeb. (2) 66. Springer, Berlin, 1972.

\bibitem[Pt\'ak(1958)]{ptak1958}
V.~Pt\'ak.
\newblock Completeness and the open mapping theorem.
\newblock \emph{Bull. Soc. Math. France}, 86:\penalty0 41--74, 1958.
\newblock \doi{10.24033/bsmf.1498}.

\bibitem[Qiu(1985)]{qiu1985}
J.~Qiu.
\newblock A new class of locally convex spaces and the generalization of
  {K}alton's closed graph theorem.
\newblock \emph{Acta Math. Sci. Ser. B Engl. Ed.}, 5\penalty0 (4):\penalty0
  389--397, 1985.
\newblock \doi{10.1016/S0252-9602(18)30540-X}.

\bibitem[Qiu(1995)]{qiu1995}
J.~Qiu.
\newblock A general version of {K}alton's closed graph theorem.
\newblock \emph{Acta Math. Sci. Ser. B Engl. Ed.}, 15\penalty0 (2):\penalty0
  161--170, 1995.
\newblock \doi{10.1016/S0252-9602(18)30036-5}.

\bibitem[Qiu and McKennon(1991)]{qiu1991}
J.~Qiu and K.~McKennon.
\newblock Banach-{M}ackey spaces.
\newblock \emph{Int. J. Math. Math. Sci.}, 14\penalty0 (2):\penalty0 215--220,
  1991.
\newblock \doi{10.1155/S0161171291000224}.

\bibitem[Schaefer(1971)]{schaefer1971}
H.H. Schaefer.
\newblock \emph{Topological vector spaces}.
\newblock Grad. Texts in Math. Springer, Berlin, 3rd edition, 1971.
\newblock \doi{10.1007/978-1-4684-9928-5}.

\bibitem[Schmets and Zafarani(1987)]{schmets1987}
J.~Schmets and J.~Zafarani.
\newblock Strict topologies and {(gDF)}-spaces.
\newblock \emph{Arch. Math. (Basel)}, 49\penalty0 (3):\penalty0 227--231, 1987.
\newblock \doi{10.1007/BF01271662}.

\bibitem[Sentilles(1972)]{sentilles1972}
F.D. Sentilles.
\newblock Bounded continuous functions on a completely regular space.
\newblock \emph{Trans. Amer. Math. Soc.}, 168:\penalty0 311--336, 1972.
\newblock \doi{10.2307/1996178}.

\bibitem[Sharpe(1988)]{sharpe1988}
M.~Sharpe.
\newblock \emph{General theory of {Markov} processes}, volume 133 of \emph{Pure
  Appl. Math. (Amst.)}.
\newblock Academic Press, Boston, MA, 1988.

\bibitem[Snipes(1973)]{snipes1973}
R.F. Snipes.
\newblock C-sequential and {S}-bornological topological vector spaces.
\newblock \emph{Math. Ann.}, 202\penalty0 (4):\penalty0 273--283, 1973.
\newblock \doi{10.1007/BF01433457}.

\bibitem[Staffans(2005)]{staffans2005}
O.J. Staffans.
\newblock \emph{Well-posed linear systems}.
\newblock Encyclopedia Math. Appl. 103. Cambridge University Press, Cambridge,
  2005.
\newblock \doi{10.1017/CBO9780511543197}.

\bibitem[Summers(1970)]{summers1970}
W.H. Summers.
\newblock Full-completeness in weighted spaces.
\newblock \emph{Canad. J. Math.}, 22\penalty0 (6):\penalty0 1196--1207, 1970.
\newblock \doi{10.4153/CJM-1970-137-5}.

\bibitem[Summers(1972)]{summers1972}
W.H. Summers.
\newblock Separability in the strict and substrict topologies.
\newblock \emph{Proc. Amer. Math. Soc.}, 35\penalty0 (2):\penalty0 507--514,
  1972.
\newblock \doi{10.2307/2037637}.

\bibitem[Travis(1981)]{travis1981}
C.C. Travis.
\newblock Differentiability of weak solutions to an abstract inhomogeneous
  differential equation.
\newblock \emph{Proc. Amer. Math. Soc.}, 82\penalty0 (3):\penalty0 425--430,
  1981.
\newblock \doi{10.1090/S0002-9939-1981-0612734-2}.

\bibitem[Webb(1968)]{webb1968}
J.~Webb.
\newblock Sequential convergence in locally convex spaces.
\newblock \emph{Math. Proc. Cambridge Philos. Soc.}, 64\penalty0 (2):\penalty0
  341--364, 1968.
\newblock \doi{10.1017/S0305004100042900}.

\bibitem[Wegner(2014)]{wegner2014}
S.-A. Wegner.
\newblock Universal extrapolation spaces for {$\mathrm{C}_{0}$}-semigroups.
\newblock \emph{Ann. Univ. Ferrara}, 60\penalty0 (2):\penalty0 447--463, 2014.
\newblock \doi{10.1007/s11565-013-0189-5}.

\bibitem[Weiss(1989{\natexlab{a}})]{weiss1989a}
G.~Weiss.
\newblock Admissibility of unbounded control operators.
\newblock \emph{SIAM J. Control Optim.}, 27\penalty0 (3):\penalty0 527--545,
  1989{\natexlab{a}}.
\newblock \doi{10.1137/0327028}.

\bibitem[Weiss(1989{\natexlab{b}})]{weiss1989b}
G.~Weiss.
\newblock Admissible observation operators for linear semigroups.
\newblock \emph{Israel J. Math.}, 65\penalty0 (1):\penalty0 17--43,
  1989{\natexlab{b}}.
\newblock \doi{10.1007/BF02788172}.

\bibitem[Wilansky(1981)]{wilansky1981}
A.~Wilansky.
\newblock Mazur spaces.
\newblock \emph{Int. J. Math. Math. Sci.}, 4\penalty0 (1):\penalty0 39--53,
  1981.
\newblock \doi{10.1155/S0161171281000021}.

\bibitem[Wiweger(1961)]{wiweger1961}
A.~Wiweger.
\newblock Linear spaces with mixed topology.
\newblock \emph{Studia Math.}, 20\penalty0 (1):\penalty0 47--68, 1961.
\newblock \doi{10.4064/sm-20-1-47-68}.

\end{thebibliography}
\bibliographystyle{plainnat}
\end{document}